%% file: main.tex
\documentclass{amsart}
\input{structural/Preamble}

\usepackage{todonotes}
\begin{document}

\input{structural/Title}
\input{structural/TOC}

\newcommand{\sub}[1]{\stackrel{#1}{\subseteq}}
\newcommand{\inv}{\mathrm{inv}}
\newcommand{\Lat}{\mathrm{Lat}}
\newcommand{\Typ}{\mathrm{Typ}}
\newcommand{\Rel}{\mathrm{Rel}}

\newenvironment{case}[1]{\begin{description}[leftmargin=0pt]\item[Case \boldmath\ensuremath{#1}]}{\end{description}}

\newcommand{\cfactor}{\gamma}
\newcommand{\Dfactor}{d}

\newcommand{\smallarray}[2]{\ifx\dummy#1\dummy #2\else\ifx\dummy#2\dummy #1\else\begin{array}{@{\,}c@{\,}}#1\\[-5pt]#2\end{array}\fi\fi}
\newcommand{\smallarrayl}[2]{\ifx\dummy#1\dummy #2\,\else\ifx\dummy#2\dummy #1\,\else\begin{array}{@{}c@{\,}}#1\\[-5pt]#2\end{array}\fi\fi}
\newcommand{\smallarrayr}[2]{\ifx\dummy#1\dummy\,#2\else\ifx\dummy#2\dummy\,#1\else\begin{array}{@{\,}c@{}}#1\\[-5pt]#2\end{array}\fi\fi}
\newcommand{\smallarraylr}[2]{\ifx\dummy#1\dummy\,#2\,\else\ifx\dummy#2\dummy\,#1\,\else\begin{array}{@{}c@{}}#1\\[-5pt]#2\end{array}\fi\fi}

\newcommand{\TypeCC}[2]{\left(\smallarraylr{#1}{#2}\right)}
\newcommand{\TypeC}[4]{\left(\smallarrayl{#1}{#2},\smallarrayr{#3}{#4}\right)}
\newcommand{\Typec}[6]{\left(\smallarrayl{#1}{#2},\smallarray{#3}{#4},\smallarrayr{#5}{#6}\right)}
\newcommand{\Typecc}[8]{\left(\smallarrayl{#1}{#2},\smallarray{#3}{#4},\smallarray{#5}{#6},\smallarrayr{#7}{#8}\right)}
\newcommand{\Type}[2]{\left(\cdots,\smallarray{#1}{#2},\cdots\right)}
\newcommand{\Typee}[4]{\left(\cdots,\smallarray{#1}{#2},\smallarray{#3}{#4},\cdots\right)}
\newcommand{\Typeee}[6]{\left(\cdots,\smallarray{#1}{#2},\smallarray{#3}{#4},\smallarray{#5}{#6},\cdots\right)}
\newcommand{\typeC}[2]{\left(#1,#2\right)}
\newcommand{\typec}[3]{\left(#1,#2,#3\right)}
\newcommand{\typecc}[4]{\left(#1,#2,#3,#4\right)}
\newcommand{\type}[1]{\left(\cdots,#1,\cdots\right)}
\newcommand{\typee}[2]{\left(\cdots,#1,#2,\cdots\right)}

\newcommand{\intert}{{\{\circ\bullet\}}}

\newsavebox{\pullback}
\sbox\pullback{%
\begin{tikzpicture}%
\draw (0,0) -- (1ex,0ex);%
\draw (1ex,0ex) -- (1ex,1ex);%
\end{tikzpicture}}
\newcommand{\join}{\bullet}
\newcommand{\meet}{\bullet}

\newcommand{\Ugen}{\ensuremath{\mathrm{(uH)}}\xspace}
\newcommand{\Uf}{\ensuremath{\Ugen^\flat}\xspace}
\newcommand{\Us}{\ensuremath{\Ugen^\sharp}\xspace}
\newcommand{\Un}{\ensuremath{\Ugen^\natural}\xspace}

\newcommand{\Ogen}{\ensuremath{\mathrm{(S)}}\xspace}
\newcommand{\Of}{\ensuremath{\Ogen^\flat}\xspace}
\newcommand{\Os}{\ensuremath{\Ogen^\sharp}\xspace}
\newcommand{\On}{\ensuremath{\Ogen^\natural}\xspace}

\newcommand{\SPgen}{\ensuremath{\mathrm{(A)}}\xspace}
\newcommand{\SPf}{\ensuremath{\SPgen^\flat}\xspace}
\newcommand{\SPn}{\ensuremath{\SPgen^\natural}\xspace}

\newcommand{\gengen}{\ensuremath{\mathrm{(?)}}\xspace}
\newcommand{\genf}{\ensuremath{\gengen^\flat}\xspace}
\newcommand{\gens}{\ensuremath{\gengen^\sharp}\xspace}
\newcommand{\genn}{\ensuremath{\gengen^\natural}\xspace}

\newcommand{\citetypdef}{Proposition 2.1}
\newcommand{\citetypFinite}{Definition 2.3}
\newcommand{\citeRprop}{Proposition 2.5}
\newcommand{\citeSprop}{Proposition 2.6}
\newcommand{\citemainLemmaFinite}{Theorem 2.7}
\newcommand{\citeTypzerodef}{Definition 3.1}
\newcommand{\citelatticeClassification}{Proposition 3.2}
\newcommand{\citemainLemma}{Theorem 3.3}
\newcommand{\citemainLemmaCount}{Corollary 3.4}
\newcommand{\citeXoverKclassification}{Proposition 4.2}
\newcommand{\citeTypdef}{Definition 4.5}
\newcommand{\citeReldef}{Definition 4.6}
\newcommand{\citetdef}{Definition 4.9}
\newcommand{\citepreserveshalf}{Proposition 4.10}
\newcommand{\citeBergmanProp}{Proposition 4.11}
\newcommand{\citeDeltadefinition}{Definition 4.13}
\newcommand{\citeDeltapreserves}{Proposition 4.15}
\newcommand{\citeinducedThm}{Theorem 4.16}
\newcommand{\citeuHsphremark}{Remark 5.8}

\input{1-Introduction}

\subsection*{Acknowledgments}
The author would like to thank Andrew Graham, Siddharth Mahendraker and Zhiyu Zhang for their interest and helpful discussions. The author is grateful to Wei Zhang for his guidance and suggestions. This work was supported in part by the NSF grant DMS--1901642, and in part by the NSF grant DMS--1440140, while the author was in residence at the Simons Laufer Mathematical Sciences Institute in Berkeley, California, during the Spring Semester of 2023.

\input{2-Combinatorics}
\input{3-Lattices}

\input{4-SphericalFunctions}

\input{5-Consequences}

\input{structural/Bibliography}
\input{table}
\input{output.ind}
\end{document}

%% file: structural/Preamble.tex
\usepackage{mathtools,amsthm}
\usepackage[letterpaper, margin=1.3in]{geometry}
\usepackage[utf8]{inputenc}
\usepackage[T2A,T1]{fontenc}
\usepackage{imakeidx}
\usepackage{enumitem}
\usepackage{mathrsfs}
\usepackage{tikz-cd}
\usepackage{amssymb}
\usepackage{xspace}
\usepackage{enumitem}
\usepackage{etoolbox}
\usepackage{stmaryrd}
\usepackage{setspace}
\usepackage{multirow}
\usepackage{pdflscape}
\usepackage{adjustbox}
\usepackage{nicematrix}
\usepackage[pagebackref]{hyperref}
\usepackage[nameinlink]{cleveref}

\hypersetup{colorlinks={true},linkcolor={blue},urlcolor=blue}

\makeindex[columns=2, title={Index of Notation},options={-s structural/index_style.ist}]
\makeatletter
\let\ori@idxitem\@idxitem
\def\@idxitem{\clear@penalties\ori@idxitem}
\def\clear@penalties{\subitem@count=3 }
\newcount\subitem@count
\def\subitem{%
  \advance\subitem@count -1
  \par
  \ifnum\subitem@count>0 \penalty10000 \fi
  \ori@idxitem%
}
\makeatother

\newtheorem{theorem}{Theorem}[subsection]
\newtheorem{lemma}[theorem]{Lemma}
\newtheorem{corollary}[theorem]{Corollary}
\newtheorem{proposition}[theorem]{Proposition}
\newtheorem{hypothesis}[theorem]{Hypothesis}

\newtheorem{innercustomgeneric}{\customgenericname}
\providecommand{\customgenericname}{}
\newcommand{\newcustomtheorem}[2]{%
  \newenvironment{#1}[1]
  {%
   \renewcommand\customgenericname{#2}%
   \renewcommand\theinnercustomgeneric{\kern-0.3em ##1}%
   \innercustomgeneric
  }
  {\endinnercustomgeneric}
}

\newcustomtheorem{specialtheorem}{}
\crefname{specialtheorem}{}{}

\theoremstyle{definition}
\newtheorem{definition}[theorem]{Definition}

\theoremstyle{remark}
\newtheorem{remark}[theorem]{Remark}

\newcommand{\N}{\mathbb{N}}
\newcommand{\Z}{\mathbb{Z}}
\newcommand{\Q}{\mathbb{Q}}

\renewcommand{\C}{\mathbb{C}}
\newcommand{\F}{\mathbb{F}}
\newcommand{\m}{\mathfrak{m}}
\renewcommand{\O}{\mathcal{O}}

\newcommand{\rightiso}{\xrightarrow{\sim}}
\newcommand{\iso}{\simeq}

\newcommand{\defeq}{\vcentcolon=}

\newcommand{\isequal}{\stackrel{?}{=}}
\newcommand{\qbinom}[3]{\genfrac{[}{]}{0pt}{}{#1}{#2}_{#3}}

%% file: structural/Title.tex
\title{Spherical functions on symmetric spaces of Friedberg--Jacquet type}
\author{Murilo Corato-Zanarella}

\begin{abstract}
We give explicit models for spherical functions on $p$-adic symmetric spaces $X=H\backslash G$ for pairs of $p$-adic groups $(G,H)$ of the form $(\mathrm{U}(2r),\mathrm{U}(r)\times \mathrm{U}(r)),$ $(\mathrm{O}(2r),\mathrm{O}(r)\times \mathrm{O}(r)),$ $(\mathrm{Sp}(4r),\mathrm{Sp}(2r)\times\mathrm{Sp}(2r))),$ $(\mathrm{U}(2r+1),\mathrm{U}(r+1)\times \mathrm{U}(r)),$ and $ (\mathrm{O}(2r+1),\mathrm{O}(r+1)\times \mathrm{O}(r)).$ As an application, we completely describe their Hecke module structure.
\end{abstract}
\maketitle

%% file: structural/TOC.tex
\setcounter{tocdepth}{2}
\let\oldtocsection=\tocsection
\let\oldtocsubsection=\tocsubsection
\let\oldtocsubsubsection=\tocsubsubsection
\renewcommand{\tocsection}[2]{\hspace{0em}\oldtocsection{#1}{#2}}
\renewcommand{\tocsubsection}[2]{\hspace{1em}\oldtocsubsection{#1}{#2}}
\renewcommand{\tocsubsubsection}[2]{\hspace{2em}\oldtocsubsubsection{#1}{#2}}
\tableofcontents

%% file: 1-Introduction.tex
\section{Introduction}
\subsection{Main results}
Let $F_0$ be a finite extension of $\Q_p$ for some prime $p>2.$ We will consider a symmetric space of the form $X=\mathbb{H}(F_0)\backslash \mathbb{G}(F_0)$ for one of the following pairs of reductive groups over $F_0.$
\begin{equation*}
    (\mathbb{G},\mathbb{H})\in\left\{\begin{array}{l}(\mathrm{U}(2r),\mathrm{U}(r)\times \mathrm{U}(r)),\quad(\mathrm{O}(2r),\mathrm{O}(r)\times \mathrm{O}(r)),\quad(\mathrm{Sp}(4r),\mathrm{Sp}(2r)\times\mathrm{Sp}(2r))),\\ (\mathrm{U}(2r+1),\mathrm{U}(r+1)\times \mathrm{U}(r)),\quad(\mathrm{O}(2r+1),\mathrm{O}(r+1)\times \mathrm{O}(r))\end{array}\right\}.
\end{equation*}
We denote $G=\mathbb{G}(F_0)$ and $H=\mathbb{H}(F_0).$ The embedding $H\hookrightarrow G$ is induced by a decomposition $V=V_1\oplus V_2$ of the corresponding Hermitian/symmetric/alternating spaces of ranks given as above. The unitary groups are with respect to an unramified quadratic extension $F/F_0$ of $p$-adic fields. In the case $(\mathrm{O}(2r),\mathrm{O}(r)\times\mathrm{O}(r)),$ we assume that the quadratic space $V$ has Hasse invariant $1.$ We will abbreviate the above cases by \Uf, \Of, \SPf, \Us, \Os respectively.

We assume $\mathbb{G}$ is unramified over $F_0,$ so we have a hyperspecial subgroup $K\subseteq G.$ We consider the space $\mathcal{S}(X/K)=\mathcal{S}(X)^K$ of ($\C$-valued) spherical functions on the spherical variety $X.$ It carries an action of the Hecke algebra $\mathcal{H}(G,K)=\mathcal{C}_0^\infty(K\backslash G/K).$ In this paper, we give an explicit description of the $\mathcal{H}(G,K)$-module $\mathcal{S}(X/K)$ in the above cases. In particular, we have
\begin{specialtheorem}{Theorem A}[\Cref{rankThm}]\label{ThmA}
    In the cases \Uf, \Us, \Of, \Os, $\mathcal{S}(X/K)$ is a free $\mathcal{H}(G,K)$-module of rank
    \begin{equation*}
        \begin{array}{c|c|c|c}
            \Uf&\Us&\Of&\Os\\\hline
            2^{r-1}&2^{r-1}&4^{r-1}&4^{r-1}
        \end{array}
    \end{equation*}
\end{specialtheorem}
Moreover, if we let $X_G\defeq\bigsqcup_{H}(H\backslash G)$ where, given $G,$ we sum over all appropriate isomorphism classes of $H$ as above (that is, $1, 2, 4$ isomorphism classes in the case \SPgen, \Ugen, \Ogen respectively), then we construct the following spherical transforms.

\begin{specialtheorem}{Theorem B}[\Cref{sphtransu,sphtranssp}]\label{ThmB}
In the cases \SPf, \Uf, \Us there is an identification of $\mathcal{S}(X_G/K)$ with the algebra $\C[\boldsymbol{\nu}_1,\ldots\boldsymbol{\nu}_r]^{\mathrm{sym}}$ in such a way that $\mathcal{H}(G,K)$ acts as follows: if $\mathcal{H}(G,K)\iso\C[\boldsymbol{\mu}_1,\ldots\boldsymbol{\mu}_n]^{\mathrm{sym}}$ (where $n=r$ in the case \Ugen, $n=2r$ in the case \SPgen) is the Satake transform, then the action $\mathcal{H}(G,K)\to\C[\boldsymbol{\nu}_1,\ldots\boldsymbol{\nu}_r]^{\mathrm{sym}}$ is induced by
\begin{equation*}
\begin{array}{c|c|c}
\SPf&\Uf&\Us\\
\hline
    \begin{array}{rl}\boldsymbol{\mu}_{2k-1}+\boldsymbol{\mu}_{2k}&\mapsto\left(\sqrt{q}+\frac{1}{\sqrt{q}}\right)\boldsymbol{\nu}_k\\
    \boldsymbol{\mu}_{2k-1}\boldsymbol{\mu}_{2k}&\mapsto \left(\sqrt{q}-\frac{1}{\sqrt{q}}\right)^2+\boldsymbol{\nu}_k^2\end{array}&\boldsymbol{\mu}_k\mapsto-\boldsymbol{\nu}_k^2-2&\boldsymbol{\mu}_k\mapsto\boldsymbol{\nu}_k^2+2
\end{array}
\end{equation*}
\end{specialtheorem}

\begin{remark}
We note that the cases \Uf and \Us have been previously settled by Hironaka--Komori in \cite{uHflat} and \cite{uHsharp}, and the case \SPf is a special case of the results of Sakellaridis \cite{Sakellaridis}.
\end{remark}

We expect to be able to obtain analogous spherical transforms in the cases \Ogen, and this will be pursued in forthcoming work. One of the expectations arising from the work of Ben-Zvi--Sakellaridis--Venkatesh \cite{BZSV} is that the $\mathcal{H}(G,K)$-module of spherical functions $\mathcal{S}(X_G/K)$ should be identified with functions on a twisted inertia stack of the dual Hamiltonian $\check{G}$-variety $\check{M},$ namely
\begin{equation*}
    \mathcal{S}(X_G/K)\iso\C\left[g\in\check{G},\ m\in\check{M}\colon (g,q^{-1/2})\cdot m=m\right]^{\check{G}}.
\end{equation*}
This should be in such a way that the Hecke action is given by the composition of the Satake transform $\mathcal{H}(G,K)\iso\C[g\in\check{G}]^{\check{G}}$ with the natural map into the above right hand side (see \cite[Section 9.4]{BZSV}). In the cases \Ogen, the dual object $\check{M}$ is not yet defined--this is due to the presence of roots of type $N$ in $X_G$--and we hope that our explicit descriptions in these toy cases can be used as a testing ground for investigating this issue.

In another direction, the results of this article will be used in the author's upcoming work on a bipartite Euler system involving unitary Friedberg--Jacquet periods \cite{CoratoZanarellaEuler}. A key ingredient in such construction is a local computation on the Hecke module $\mathcal{S}(X/K)$ in the case \Uf. We hope that our results in the case \Os can also be used in similar constructions. This type of relationship between local representation theory and Euler systems has been observed before in the non-bipartite setting. For example, \cite{LSZ-U21,LSZ-SP4,Graham-Shah} use local multiplicity one results to reduce the verification of horizontal Euler system relations to the computation of certain local zeta integrals. Following an idea of Cornut \cite{Cornut1,Cornut2}, similar verification in cases of higher local multiplicity (see \cite{Cornut,Jetchev,BBJHorizontal,BBJVertical,Boumasmoud}) rely on the analysis of certain Hecke modules closely related to the ones considered in this article.

\subsection{Strategy and structure of the paper}
We explain the main steps of our analysis. We warn the reader that some of the notation in this subsection is \emph{not} consistent with the notation in the main text.

We follow the general strategy outlined in the author's previous work \cite{CoratoZanarella}, which analyzed the cases
\begin{equation*}
    (\mathbb{G},\mathbb{H})\in\{(\mathrm{Res}_{F/F_0}\mathrm{GL}(r),\mathrm{U}(r)),(\mathrm{GL}(r),\mathrm{O}(r)),(\mathrm{GL}(2r),\mathrm{Sp}(2r))\}.
\end{equation*}
These will henceforth be denoted by \Un, \On, \SPn respectively. We refer to the reader to the introduction of \cite{CoratoZanarella} for more details on this strategy. Roughly, it consists of:
\begin{enumerate}[leftmargin=*]
    \item We identify $G/K$ with a certain subset $\Lat^\circ(V)\subseteq\Lat(V)$ of lattices in the space $V.$ The Hecke algebra $\mathcal{H}(G,K)$ acts on $\Lat^\circ(V),$ and there is a basis $(T_i)_i\in\mathcal{H}(G,K)$ described via explicit lattice-theoretic correspondences $T_i^{cor}\subseteq\Lat^\circ(V)\times\Lat^\circ(V).$
    \item We parametrize the $H$-orbits in a combinatorial fashion via an invariant $\mathrm{typ}\colon H\backslash \Lat^\circ(V)\hookrightarrow\Typ^0.$ This is a relative Cartan decomposition of $X,$ and allows us to describe (the adjoint of) the action of $\mathcal{H}(G,K)$ on $\mathcal{S}(X/K)$ in terms of some \emph{lattice-theoretic counting questions} of the form: given $e,f\in\Typ^0$ and $\Lambda_1\in\Lat^\circ(V)$ with $\mathrm{typ}(\Lambda_1)=e,$ what is the cardinality of $\{\Lambda_2\in\Lat^\circ(V)\colon \mathrm{typ}(\Lambda_2)=f,\ (\Lambda_1,\Lambda_2)\in T_i^{cor}\}$?
    \item We observe that the operators $T_i^*\colon\C[\Typ^0]\to\C[\Typ^0]$ are, \emph{generically} (or \emph{assymptotically}), given by linear combinations $\Delta_i=\sum_\varepsilon a_\varepsilon\cdot t(\varepsilon)$ for certain partially defined \emph{translations} $t(\varepsilon)\colon\C[\Typ^0]\dashrightarrow\C[\Typ^0].$
    \item We prove that the operators $T_i^*$ are fully determined by the $\Delta_i$ above, in the following way: we embed $\Typ^0$ inside a larger set $\Typ$ (in practice, $\Typ^0$ will be a cone inside $\Typ$) which satisfies i) the translations $t(\varepsilon)\colon\C[\Typ]\to\C[\Typ]$ extend to well-defined operators, ii) we have a submodule $\Rel\subseteq\C[\Typ]$ of \emph{straightening relations} with
    \begin{equation*}
        \mathrm{str}\colon\C[\Typ]/\Rel\rightiso\C[\Typ^0],
    \end{equation*}
    iii) $\Delta_i$ preserves $\Rel,$ and the induced operator $\mathrm{str}^*\Delta_i\colon\C[\Typ^0]\to\C[\Typ^0]$ agrees with $T_i^*.$
\end{enumerate}

However, the current situations are substantially more involved than the ones in \cite{CoratoZanarella}. This stems from the fact that $G$ is not a general linear group. Some of the key differences that we have to overcome are:
\begin{enumerate}[label=(\alph*),leftmargin=*]
    \item Directly computing the generic behavior of the ``natural'' lattice-theoretic basis $T_i$ is difficult. Instead, we compute them for suitably chosen linear combinations $T_{\le i}.$
    \item While the ``natural'' lattice-theoretic operators $T_i$ form a basis of the Hecke algebra $\mathcal{H}(G,K),$ they do not agree with the basis given by the Satake transform. The relationship between these bases is given by special cases of the Lusztig--Kato formula. For the unitary cases, these have been made explicit in \cite{LTXZZ} via geometric Satake. For the orthogonal and symplectic cases, we make them explicit using recent results on the combinatorics of generalized Kostka--Foulkes polynomials \cite{Lecouvey-Lenart} and \cite{Jang-Kwon}.
    \item\label{hardpart} Verifying that $T_i^*$ agrees with $\Delta_i$ directly, as done for the cases \Un, \On, \SPn in \cite{CoratoZanarella}, seems to be not feasible.
\end{enumerate}

Overcoming \labelcref{hardpart} is the main technical hurdle of this paper. The rough idea for this is as follows: Suppose we have two nontrivial operators $T_{\text{small}},T_{\text{big}}$ which are generically given by $\Delta_{\text{small}}$ and $\Delta_{\text{big}}$ such that i) $T_{\text{small}}$ is induced by $\Delta_{\text{small}}$ in the sense above, ii) $T_{\text{small}}$ and $T_{\text{big}}$ commute. Then $\Delta_{\text{small}}$ and $\Delta_{\text{big}}$ must also commute, and we can evaluate the expression $T_{\text{small}}T_{\text{big}}=T_{\text{big}}T_{\text{small}}$ at well chosen elements in order to increase the set where we know that $T_{\text{big}}$ and $\Delta_{\text{big}}$ agree. By repeating this, we can hope to prove that $T_{\text{big}}$ and $\Delta_{\text{big}}$ agree everywhere. We call such procedures \emph{extension arguments}, and the details depend on the shapes of $\Delta_{\text{small}}$ and $\Rel.$

This is how we produce the appropriate $T_{\text{small}}$ to run the above argument.
\begin{itemize}[leftmargin=*]
    \item For a case \genf, we exploit the fact that $G$ has two conjugacy classes of hyperspecial subgroups\footnote{This is why we restrict the Hasse invariant in the case \Of.}, which we denote $K^\circ,K^\bullet.$ We have the $\mathcal{H}(G,K^\circ)$-module $\mathcal{S}(X/K^\circ)$ and the $\mathcal{H}(G,K^\bullet)$-module $\mathcal{S}(X/K^\bullet),$ and intertwining operators $T^{\circ\bullet}$ and $T^{\bullet\circ}$ between them. We use such intertwining operators to play the role of $T_{\text{small}}.$ We show that $T^{\circ\bullet}$ and $T^{\bullet\circ}$ are described by $\Rel^\flat$ via an explicit computation similar to the one done in \cite{CoratoZanarella}.
    \item For a case \gens, we will use the ``smallest'' operator $T_1^\sharp$ as $T_{\text{small}}.$ To show that it is described by $\Rel^\sharp,$ we will exploit that  $T_1^\sharp$ is close to $T_1^\flat,$ and that both $\Rel^\sharp$ and $\Rel^\flat$ contain $\Rel^\natural$ as a ``large'' submodule. Namely, we will have $\Typ^\sharp=\Typ^\flat$ and $\Typ^{0,\sharp}=\Typ^{0,\flat}$ and if we write $T_1^\sharp=T_1^\flat+T_1^{\text{err}}$ and $\Delta_1^\sharp=\Delta_1^\flat+\Delta_1^{\text{err}},$ it is possible to explicitly compute $T_1^{\text{err}}$ and verify that
    \begin{equation*}
        T_1^{\text{err}}=(\mathrm{str}^\sharp-\mathrm{str}^\flat)(\Delta_1^\flat)+\mathrm{str}^\sharp(\Delta_1^{\text{err}}).
    \end{equation*}
    Together with the analysis in the case \genf, this will imply that $T_1^\sharp$ is indeed described by $\Rel^\sharp.$
\end{itemize}
\begin{remark}
    See \Cref{reltable} for a summary of the different straightening relations in this text.
\end{remark}
\begin{remark}
    It should be possible to analyze the case \Of for the other isomorphism class of $V$ by running the above argument backwards.
\end{remark}
\begin{remark}\label{nonlinear}
    As outlined above, the argument for the cases \gens relies on the end result for the cases \genf. This means that the logic in the paper will not be completely linear as written. We will be explicit about this when it happens.
\end{remark}

In \Cref{finFieldSection}, we collect several computations which are used throughout the text. In \Cref{LatticeSection}, we discuss the relative Cartan decomposition of $X$ in terms of lattices. In \Cref{Heckesection}, we fix some notation for the Hecke algebras that appear in this text and collect the explicit versions of the Lusztig--Kato formulas that we will need. \Cref{MainSection} is devoted to the statement and proof of the explicit model of the $\mathcal{H}(G,K)$-module $\mathcal{S}(X/K),$ following the strategy outline above. This is used in \Cref{ResultsSection} to deduce the main results \Cref{ThmA} and \Cref{ThmB}.
\subsection{Notations}
Throughout the whole paper, we will consider three cases: \Ugen, \Ogen and \SPgen, standing for \emph{unramified Hermitian}, \emph{symmetric} and \emph{alternating}. The three cases will be completely independent of each other. Whenever we consider a case \gens, we will implicitly assume that we are in one of the cases \Ugen, \Ogen.

\subsubsection*{$p$-adic fields}
We will denote $F_0$ to be a $p$-adic local field for some prime $p.$\index{p@$p$} We will consider\index{F@$F_0,\ F$}
\begin{case}{\Ugen}
$F/F_0$ an unramified quadratic extension,
\end{case}
\begin{case}{\Ogen}
$F=F_0,$ and we assume $p$ is odd,
\end{case}
\begin{case}{\SPgen}
$F=F_0.$
\end{case}
We denote by $\O_F$ and $\O_{F_0}$ the corresponding rings of integers, with maximal ideals $\m_F$ and $\m_{F_0}.$ We let $q\defeq\#\O_F/\m_F$ and $q_0\defeq\#\O_{F_0}/\m_{F_0}.$\index{q@$q_0,\ q$} Fix $\varpi\in\m_{F_0}$\index{1pi@$\varpi$} an uniformizer.

We denote $\mathrm{Sign}$\index{Sign@$\mathrm{Sign}$} to be the set
\begin{case}{\Ugen}
$\mathrm{Sign}=\{1\},$
\end{case}
\begin{case}{\Ogen}
$\mathrm{Sign}=\O_F^\times/(\O_F^\times)^2=\F_q^\times/(\F_q^\times)^2=\{\pm1\},$
\end{case}
\begin{case}{\SPgen}
$\mathrm{Sign}=\{1\},$
\end{case}
and $\mathrm{sign}\colon\F_q^\times\to\mathrm{Sign}$ and $\mathrm{sign}\colon\O_F^\times\to\mathrm{Sign}$\index{sign@$\mathrm{sign}$} to be the natural maps. We also denote $\epsilon\defeq\mathrm{sign}(-1).$\index{1epsilon@$\epsilon$}

We also denote
\begin{equation*}
\alpha\defeq\left\{\begin{array}{cl}1/2 & \text{in the case \Ugen,}\\
    0 & \text{in the case \Ogen,}\\
    1 & \text{in the case \SPgen,}\end{array}\right.
\quad\text{and}\quad
\cfactor\defeq \begin{cases}
    1 & \text{in the cases \Ugen, \Ogen,} \\
    2 & \text{in the case \SPgen.}
\end{cases}
\end{equation*}\index{1alpha@$\alpha$}\index{1gamma@$\cfactor$}
Moreover, denote
\index{d@$\Dfactor$}
\begin{equation*}
    \Dfactor(r)\defeq q^{\binom{r}{2}+\alpha r}=\begin{cases}q_0^{r^2} & \text{in the case \Ugen,}\\
    q^{\binom{r}{2}} & \text{in the case \Ogen,}\\
    q^{\binom{r+1}{2}} & \text{in the case \SPgen.}\end{cases}
\end{equation*}

\subsubsection*{Combinatorics}
For $n\in\Z_{\ge0},$ the $q$-Pochammer symbol and falling factorials are
\index{0(x;y)n@$(x;y)_n,\ (x)_n$}
\begin{equation*}
    (x;y)_n\defeq\prod_{i=1}^n(1-xy^{i-1}),\quad\text{and}\quad (x)_n\defeq(-1)^n(x;x)_n=\prod_{i=1}^n(x^i-1).
\end{equation*}
The $q$-analogues of binomial coefficients are
\index{0(nm)lambda@$\qbinom{n}{m}{\lambda}$}
\begin{equation*}
    \qbinom{n}{m}{\lambda}=\frac{(n)_\lambda}{(m)_\lambda\cdot(n-m)_\lambda}\quad\text{for integers }0\le m\le n.
\end{equation*}
For $n,m\in\Z,$ we will use the convention that $\qbinom{n}{m}{\lambda}=0$ if either $m<0,$ $m>n$ or $n<0.$ More generally, if $n,m_1,\ldots,m_r$ are integers with $n=m_1+\cdots+m_r,$ we also consider the $\lambda$-multinomials given by
\begin{equation*}
    \qbinom{n}{m_1,\ldots,m_r}{\lambda}=\frac{(n)_\lambda}{(m_1)_\lambda\cdots(m_r)_\lambda}\quad\text{if }m_1,\ldots,m_r\ge0,
\end{equation*}
and given by $0$ otherwise.

For $I$ a finite ordered set and $f\colon I\to\Z$ a function, we denote
\index{inv@$\inv,\ \widetilde{\inv}$}
\begin{equation*}
    \inv(f)\defeq\#\{(i,j)\in I^2\colon i<j\text{ and }f(i)>f(j)\}\quad\text{and}\quad \widetilde{\inv}(f)\defeq\sum_{\displaystyle\substack{(i,j)\in I^2\\i<j}}\max(0,f(i)-f(j)).
\end{equation*}
Moreover, for $J_1,J_2\subseteq I,$ we denote
\begin{equation*}
    \inv(J_1,J_2)\defeq\#\{(j_1,j_2)\in J_1\times J_2\colon j_1<j_2\}.
\end{equation*}

For a sequence of integers $e\in\Z^n,$ and an integer $m\in\Z,$ we denote
\index{1lambdame@$\lambda_m(e)$}\index{1Sigmae@$\Sigma(e)$}
\begin{equation*}
    \lambda_m(e)\defeq\#\{1\le i\le n\colon e_n=m\},\quad\text{and}\quad\Sigma(e)\defeq\sum_{i=1}^ne_i.
\end{equation*}

We also consider the following function:
\begin{equation*}
    W_\beta^{\lambda}(a,n)\defeq\sum_{\substack{\{1,\ldots,n\}=E_0\sqcup E\sqcup E_1\\ \lvert E\rvert=a}}\lambda^{\inv(E_0,E)-\inv(E_1,E)+\beta(\lvert E_0\rvert-\lvert E_1\rvert)},
\end{equation*}\index{W(a,n)@$W_\beta^\lambda(a,n)$}
and we will denote $W_\beta(a,n)\defeq W_\beta^q(a,n).$

\subsubsection*{Lattice diagrams}
For $V_1,V_2$ subspaces of a vector space, we will often use the shorthand $V_1\sub{a}V_2$\index{0sub@$\sub{k}$} to denote that $V_1\subseteq V_2$ and $\dim(V_2/V_1)=a.$ Similarly, for $L_1,L_2$ lattices, we use $L_1\sub{a}L_2$ to indicate that $L_1\subseteq L_2$ and $\mathrm{length}(L_2/L_1)=a.$

In diagrams of lattices or subspaces, containment will always be from left to right (and from bottom to top for vertical lines). Labels will indicate the length of the quotients. We will also use the symbol $\meet$\index{0meet@$\meet$} to denote joins and meets. For example, the following diagram
\begin{equation*}
\begin{tikzcd}
    L_2\arrow[r,dash]&L_4\arrow[dl,phantom,"\join",very near start]\\
    L_1\arrow[ur,phantom,"\meet",very near start]\arrow[u,dash]\arrow[r,dash,swap,"a"]&L_3\arrow[u,dash]
\end{tikzcd}
\end{equation*}
indicates that $L_1=L_2\cap L_3,$ that $L_4=L_2+L_3$ and that $L_1\sub{a}L_3.$

%% file: 2-Combinatorics.tex
\section{Combinatorics}\label{finFieldSection}
We first record the following two recursions for the $W_\beta^\lambda(a,n).$ Thinking about the last element of $\{1,\ldots,n+1\},$ we obtain
\begin{equation}\label{Wrec1}
    W_\beta^\lambda(a,n+1)=(\lambda^\beta+\lambda^{-\beta})W_\beta^\lambda(a,n)+W_{\beta+1}^\lambda(a-1,n)\quad\text{for all }a,n\ge0.
\end{equation}
Similarly, thinking about the first element of $\{1,\ldots,n+1\},$ we obtain
\begin{equation}\label{Wrec2}
    W_\beta^\lambda(a,n+1)=(\lambda^{a+\beta}+\lambda^{-a-\beta})W_\beta^\lambda(a,n)+W_\beta^\lambda(a-1,n)\quad\text{for all }a,n\ge0.
\end{equation}

\subsection{\texorpdfstring{$q$}{q}-combinatorics identities}
\begin{lemma}\label{qidentityforR}
    For integers $a,b\ge0,$ we have
    \begin{equation*}
        \sum_{c=0}^{\min(a,b)}(\lambda;\lambda^2)_c\qbinom{a+b}{2c,a-c,b-c}{\lambda}\lambda^{(a-c)(b-c)}=\qbinom{a+b}{a}{\lambda^2}.
    \end{equation*}
\end{lemma}
\begin{proof}
If suffices to prove this for $\lambda=-q$ for $q$ prime powers. Take $V$ a nondegenerate $\F_{q^2}/\F_q$-Hermitian space of dimension $a+b.$ The right hand side is the number of $\F_{q^2}$-subspaces $L$ of dimension $a.$ Summing over the possible $m=\dim(L\cap L^\perp),$ we have
\begin{equation*}
    \qbinom{a+b}{a}{q^2}=\sum_{m\ge0}Q(a-m,m,b-m),
\end{equation*}
in the notation of \cite[Theorem 2.7]{CoratoZanarella} (in the case \Ugen), which is precisely the identity we want.
\end{proof}
\begin{lemma}\label{qidentityforR2}
    For integers $a,b\ge0,$ we have
    \begin{equation*}
        \sum_{c=0}^{\min(a,b)}(\lambda;\lambda^2)_c\qbinom{a+b}{2c,a-c,b-c}{\lambda}\lambda^{(a-c)(b-c)+c}=\qbinom{a+b}{a}{\lambda^2}\frac{\lambda^a+\lambda^b}{\lambda^{a+b}+1}.
    \end{equation*}
\end{lemma}
\begin{proof}
Denote
\begin{equation*}
\begin{split}
    A(a,b)&\defeq\sum_{c=0}^{\min(a,b)}(\lambda;\lambda^2)_c\qbinom{a+b}{2c,a-c,b-c}{\lambda}\lambda^{(a-c)(b-c)}.\\
    B(a,b)&\defeq\sum_{c=0}^{\min(a,b)}(\lambda;\lambda^2)_c\qbinom{a+b}{2c,a-c,b-c}{\lambda}\lambda^{(a-c)(b-c)+c}.
\end{split}
\end{equation*}
From the identity
\begin{equation*}
    \qbinom{a+b}{2c,a-c,b-c}{\lambda}=\qbinom{a+b-1}{2c-1,a-c,b-c}{\lambda}+\lambda^{2c}\qbinom{a+b-1}{2c,a-c,b-1-c}{\lambda}+\lambda^{b+c}\qbinom{a+b-1}{2c,a-c-1,b-c}{\lambda}
\end{equation*}
we get
\begin{equation*}
\begin{split}
    (\lambda;\lambda^2)_c\lambda^{(a-c)(b-c)}\qbinom{a+b}{2c,a-c,b-c}{\lambda}=\ &(\lambda;\lambda^2)_c\lambda^{(a-c)(b-1-c)+a+c}\qbinom{a+(b-1)}{2c,a-c,(b-1)-c}{\lambda}\\
    &-(\lambda;\lambda^2)_{c-1}\lambda^{(a-c)(b-c)}\qbinom{a+b}{2(c-1),a-c,b-c}{\lambda}(q^{a+b-1}-1)\\
    &+(\lambda;\lambda^2)_c\lambda^{(a-1-c)(b-c)+2b}\qbinom{(a-1)+b}{2c,(a-1)-c,b-c}{\lambda}.
\end{split}
\end{equation*}
Summing over and changing $b\mapsto b+1,$ this is
\begin{equation*}
    A(a,b+1)=\lambda^aB(a,b)-(\lambda^{a+b}-1)A(a-1,b)+\lambda^{2b+2}A(a-1,b+1).
\end{equation*}
and thus we can use the previous lemma to verify that $B(a,b)=\qbinom{a+b}{a}{\lambda^2}\frac{\lambda^a+\lambda^b}{\lambda^{a+b}+1}.$
\end{proof}

\begin{lemma}\label{qidentityforL}
For integers $a,b$ with $a+b\ge0,$ we have
\begin{equation*}
    \sum_{c\ge0}\qbinom{a+b}{c,a+c,b-2c}{\lambda^2}\lambda^{2c(a+c)}(-\lambda;\lambda)_{b-2c}=\qbinom{2(a+b)}{b}{\lambda}.
\end{equation*}
\end{lemma}
\begin{proof}
If suffices to prove this for $\lambda=q$ for $q$ prime powers. Take $V$ a nondegenerate $\F_q$-alternating space of dimension $2(a+b).$ The right hand side is the number of $\F_{q}$-subspaces $L$ of dimension $b.$ Summing over the possible $2n=\dim(L/L\cap L^\perp),$ we have
\begin{equation*}
    \qbinom{2(a+b)}{b}{q}=\sum_{n\ge0}Q(n,b-2n,a+n),
\end{equation*}
in the notation of \cite[Theorem 2.7]{CoratoZanarella} (in the case \SPgen), which is precisely the identity we want.
\end{proof}

\subsection{Computations on incidence algebras}
\begin{definition}\label{circledastdef}\index{I(P,R)@$I(P,R)$}
    Let $P$ be a locally finite\footnote{This means that for any $x,y\in P,$ the set $[x,y]\defeq\{z\in P\colon x\le z\le y\}$ is finite.} poset and $R$ a value ring. The \emph{incidence algebra} $I(P,R)$ is the set of functions $f\colon P^2\to R$ such that $f(x,y)=0$ if $x>y.$ This is equipped with the convolution operator
    \index{0(*)@$\circledast$}
    \begin{equation*}
        (f\circledast g)(x,y)\defeq\sum_{z\in [x,y]}f(x,z)g(z,y),
    \end{equation*}
    and with the identity $\mathrm{id}(x,y)\defeq\begin{cases}
        1 & \text{if }x=y, \\
        0 & \text{otherwise.}
    \end{cases}$
\end{definition}

\begin{definition}
    Given a function $h\colon P\to R^\times,$ we consider the induced algebra automorphism $h\in\mathrm{Aut}(I(P,R))$ given by $(h\cdot f)(x,y)\defeq f(x,y)\cdot\frac{h(y)}{h(x)}.$
\end{definition}

The following is well-known.
\begin{proposition}\label{circledastinv}
    For $f\in I(P,R),$ the following are equivalent.
    \begin{enumerate}
        \item $f$ has a left inverse.
        \item $f$ has a right inverse.
        \item $f(x,x)\in R^\times$ for all $x\in P.$
    \end{enumerate}
    In such case, we denote $f^\circleddash\in I(P,R)$\index{0(-)@$\circleddash$} to be the two-sided inverse of $f.$
\end{proposition}
\begin{proof}
    Note that (1)$\implies$(3) since if $(g\circledast f)=\delta,$ then $1=(g\circledast f)(x,x)=g(x,x)f(x,x).$ Similarly, (2)$\implies$(3).

    Given $f,g\in I(P,R),$ we define $f\widehat{\circledast}g\in I(P,R)$ by
    \begin{equation*}
        (f\widehat{\circledast}g)(x,y)=\sum_{z\in(x,y)}f(x,z)g(z,y).
    \end{equation*}
    Note that
    \begin{equation*}
        (f\circledast g)(x,y)=(f\widehat{\circledast}g)(x,y) +f(x,x)g(x,y)+f(x,y)g(y,y)-\mathrm{id}(x,y)
    \end{equation*}
    Now given $f\in I(P,R),$ consider
    \begin{equation*}
        \tilde{f}\defeq 2\mathrm{id}-f+f\widehat{\circledast}f-f\widehat{\circledast}f\widehat{\circledast}f+\cdots
    \end{equation*}
    This is well-defined, since for any given $(x,y)\in R,$ the evaluation of the above expression only involves finitely many terms. By construction, we have $f\widehat{\circledast} \tilde{f}=\tilde{f}\widehat{\circledast} f=2\mathrm{id}-f-\tilde{f}.$

    Now assume that $f$ satisfies (3). Consider first the case that $f(x,x)=1$ for all $\in P.$ Note that $\tilde{f}(x,x)=1$ for all $x\in P$ as well. So we have
    \begin{equation*}
        f\circledast \tilde{f}=\tilde{f}\circledast f=f\widehat{\circledast}\tilde{f}+f+\tilde{f}-\mathrm{id}=\mathrm{id}.
    \end{equation*}
    and thus $f^\circleddash\defeq\tilde{f}$ is a two-sided inverse of $f.$ This verifies (1) and (2). For the general case, consider $g(x,y)\defeq f(y,y)^{-1}f(x,y).$ Then let $\tilde{f}(x,y)\defeq f(x,x)\tilde{g}(x,y).$ Then
    \begin{equation*}
        (\tilde{f}\circledast f)=\sum_{z\in[x,y]}f(x,x)\tilde{g}(x,z)f(y,y)^{-1}g(z,y)=f(x,x)f(y,y)^{-1}(\tilde{g}\circledast g)(x,y)=\delta(x,y).
    \end{equation*}
    This verifies (1). Similarly, we can verify (2) by taking $g=f(x,x)^{-1}f(x,y)$ and $\tilde{f}(x,y)\defeq f(y,y)\tilde{g}(x,y).$
\end{proof}

\begin{proposition}\label{qbinomialformula}
    We have that $\left(\qbinom{y}{x}{\lambda}\right)_{(x,y)}^\circleddash=\left((-1)^{y-x}\qbinom{y}{x}{\lambda}\lambda^{\binom{y-x}{2}}\right)_{(x,y)}$ in $I(\Z_{\ge0},\Z[\lambda]).$
\end{proposition}
\begin{proof}
    Their $\circledast$-product is
    \begin{equation*}
        \sum_{z=x}^y\qbinom{y}{z}{\lambda}(-1)^{z-x}\qbinom{z}{x}{\lambda}\lambda^{\binom{z-x}{2}}=\qbinom{y}{x}{\lambda}\sum_{z=x}^y(-1)^{z-x}\lambda^{\binom{z-x}{2}}\qbinom{y-x}{z-x}{\lambda}=\qbinom{y}{x}{\lambda}\sum_{z=0}^{y-x}(-1)^z\lambda^{\binom{z}{2}}\qbinom{y-x}{z}{\lambda}.
    \end{equation*}
    By the $\lambda$-binomial theorem, this is $\qbinom{y}{x}{\lambda}(-1;\lambda)_{y-x}=\delta(x,y).$
\end{proof}

We consider the following functions $A_i^\lambda(x,y),$ $B_i^\lambda(x,y),$ $C_i^\lambda(x,y),$ $D_i^\lambda(x,y)$ and $E_i^\lambda(x,y)$ of $I(\Z_{\ge0},\Z(\lambda))$ for $i\in\{1,2,3,4,5\}.$\footnote{Whenever we write $\frac{y-x}{2},$ we are implicitly denoting that the function is $0$ unless $y\equiv x\mod 2.$}
\bgroup
\everymath{\displaystyle}
\begin{center}
\begin{tabular}{c|c|c|c}
    $i$&$A_i^\lambda(x,y)$&$B_i^\lambda(x,y)$&$C_i^\lambda(x,y)$\\
    \hline
    $1$&$(-\lambda;\lambda)_{y-x}\qbinom{y}{x}{\lambda^2}$&$W_{0}^\lambda(x,y)\lambda^{\frac{1}{2}(y^2+y-x^2-x)}$&$\qbinom{y}{\lfloor\frac{y-x}{2}\rfloor}{\lambda^2}$\\
    \hline
    $2$&$(-\lambda;\lambda^2)_{y-x}\qbinom{2y}{2x}{-\lambda}$&$W_{0}^{\lambda^2}(x,y)\lambda^{y^2-x^2}$&$\qbinom{2y}{y-x}{-\lambda}$\\
    \hline
    $3$&$(-\lambda;\lambda^2)_{y-x}\qbinom{2y+1}{2x+1}{-\lambda}$&$W_{1/2}^{\lambda^2}(x,y)\lambda^{y^2+y-x^2-x}$&$\qbinom{2y+1}{y-x}{-\lambda}$\\
    \hline
    $4$&$(-\lambda;\lambda)_{y-x}\qbinom{y}{x}{\lambda^2}\frac{\lambda^x+1}{\lambda^y+1}$&$W_{0}^\lambda(x,y)\lambda^{\frac{1}{2}(y^2-y-x^2+x)}$&$\qbinom{y}{\frac{y-x}{2}}{\lambda^2}\frac{\lambda^x+1}{\lambda^y+1}$\\
    \hline
    $5$&$(-\lambda;\lambda)_{y-x}\qbinom{y}{x}{\lambda^2}$&$W_{1/2}^\lambda(x,y)\lambda^{\frac{1}{2}(y^2-x^2)}$&$\qbinom{y}{\frac{y-x}{2}}{\lambda^2}$
\end{tabular}
\end{center}
\begin{center}
\begin{tabular}{c|c|c}
$i$&$D_i^\lambda(x,y)$&$E_i^\lambda(x,y)$\\
\hline
$1$&$\lambda^{\frac{1}{2}(y^2+y-x^2+x)}\binom{y}{\lfloor\frac{y-x}{2}\rfloor}$&$(-1)^{y-x}\qbinom{2y}{y-x}{\lambda}\frac{\lambda^{x+y+1}-\lambda^{y-x}}{\lambda^{x+y+1}-1}$\\
\hline
$2$&$\lambda^{y^2-x^2}\binom{y}{\frac{y-x}{2}}$&$(-1)^{y-x}\qbinom{2y}{y-x}{\lambda^2}\frac{\lambda^{x+y}+\lambda^{y-x}}{\lambda^{2y}+1}$\\
\hline
$3$&$\lambda^{y^2+y-x^2-x}\binom{y}{\lfloor\frac{y-x}{2}\rfloor}$&$(-1)^{y-x}\qbinom{2y+1}{y-x}{\lambda^2}\frac{\lambda^{x+y+1}-\lambda^{y-x}}{\lambda^{2y+1}-1}$\\
\hline
$4$&$\lambda^{\frac{1}{2}(y^2-y-x^2+x)}\binom{y}{\frac{y-x}{2}}$&$(-1)^{y-x}\qbinom{2y}{y-x}{\lambda}\frac{\lambda^x+1}{\lambda^y+1}$\\
\hline
$5$&$\lambda^{\frac{1}{2}(y^2-x^2)}\binom{y}{\frac{y-x}{2}}$&$(-1)^{y-x}\qbinom{2y}{y-x}{\lambda}$
\end{tabular}
\end{center}
\egroup
    
\begin{proposition}\label{circledast1}
    For each $i\in\{1,2,3,4,5\},$ we have $C_i^\lambda\circledast(A_i^\lambda)^{\circleddash}=E_i^\lambda.$
\end{proposition}
\begin{proof}
    Let $S(x,y)\defeq\begin{cases}1&\text{if }y-x\in\{0,1\},\\0&\text{otherwise.}\end{cases}$ Then we have $C_1^\lambda=S\circledast C_5^\lambda.$ We can also check that $E_1^\lambda=S\circledast E_5^\lambda,$ as
    \begin{equation*}
        (-1)^{y-x}\qbinom{2y}{y-x}{\lambda}+(-1)^{y-x-1}\qbinom{2y}{y-x-1}{\lambda}=(-1)^{y-x}\qbinom{2y}{y-x}{\lambda}\frac{\lambda^{x+y+1}-\lambda^{y-x}}{\lambda^{x+y+1}-1}.
    \end{equation*}
    Since $A_1^\lambda=A_5^\lambda,$ this means that
    \begin{equation*}
        C_1^\lambda\circledast(A_1^\lambda)^{\circleddash}=E_1^\lambda\iff C_5^\lambda\circledast(A_5^\lambda)^{\circleddash}=E_5^\lambda.
    \end{equation*}
    Note also that if $h(x)=\lambda^x+1,$ then $A_5^\lambda=h\cdot A_4^\lambda,$ $C_5^\lambda=h\cdot C_4^\lambda,$ and $E_5^\lambda=h\cdot E_4^\lambda.$ Hence
    \begin{equation*}
        C_4^\lambda\circledast(A_4^\lambda)^{\circleddash}=E_4^\lambda\iff C_5^\lambda\circledast(A_5^\lambda)^{\circleddash}=E_5^\lambda.
    \end{equation*}

    For $i=5,$ multiplying \Cref{qbinomialformula} by $h(x)=(-\lambda;\lambda)_x$ yields
    \begin{equation*}
        (A_5^\lambda)^\circleddash(x,y)=(-1)^{y-x}\lambda^{\binom{y-x}{2}}(-\lambda;\lambda)_{y-x}\qbinom{y}{x}{\lambda^2},
    \end{equation*}
    and thus we want to prove, setting $z=x+2c,$ that
    \begin{equation*}
        \sum_{c\ge0}\lambda^{\binom{y-x-2c}{2}}(-\lambda;\lambda)_{y-x-2c}\qbinom{y}{c,x+c,y-x-2c}{\lambda^2}\isequal\qbinom{2y}{y-x}{\lambda}.
    \end{equation*}
    This is \Cref{qidentityforL} for $(x,y-x,\lambda^{-1}).$ By the above, this also implies the identity for $i\in\{1,4\}.$

    We are left to consider $i\in\{2,3\}.$ Consider the following functions $A,C,E\in I(\frac{1}{2}\Z_{\ge0},\Z(\lambda)).$ We denote $A(x,y)=C(x,y)=E(x,y)=0$ if $y-x\not\in\Z,$ and otherwise we let
    \begin{equation*}
        A^\lambda(x,y)=(\lambda;\lambda^2)_{y-x}\qbinom{2y}{2x}{\lambda},\quad C^\lambda(x,y)=\qbinom{2y}{y-x}{\lambda},\quad E^\lambda(x,y)=\qbinom{2y}{y-x}{\lambda^2}\frac{\lambda^{x+y}+\lambda^{y-x}}{\lambda^{2y}+1}.
    \end{equation*}
    Note that for $F\in\{A,C,E\},$ we have
    \begin{equation*}
        F^\lambda(x,y)=\begin{cases}
            F_2^{-\lambda}(x,y)&\text{if }x,y\in\Z,\\F_3^{-\lambda}(x-1/2,y-1/2)&\text{if }x,y\in\frac{1}{2}+\Z.
        \end{cases}
    \end{equation*}
    Thus it remains to see that $C^\lambda\circledast(A^\lambda)^\circleddash\isequal E^\lambda.$ One can check that \Cref{qbinomialformula} for $\lambda^2$ implies that
    \begin{equation*}
        (A^\lambda)^{\circleddash}(x,y)=(-1)^{y-x}\lambda^{2\binom{y-x}{2}}(\lambda;\lambda^2)_{y-x}\qbinom{2y}{2x}{\lambda}\quad\text{for }y-x\in\Z,
    \end{equation*}
    and thus we want to prove, setting $z=y-c,$ that for all $x,y\in\frac{1}{2}+\Z$ with $y-x\in\Z,$ we have
    \begin{equation*}
        \sum_{c\ge0}(-1)^c\lambda^{c^2-c}(\lambda;\lambda^2)_c\qbinom{2y}{2c,y-x-c,x+y-c}{\lambda}\isequal\qbinom{2y}{y-x}{\lambda^2}\frac{\lambda^{x+y}+\lambda^{y-x}}{\lambda^{2y}+1}.
    \end{equation*}
    This is \Cref{qidentityforR2} for $(x+y,y-x,\lambda^{-1}).$
\end{proof}
\begin{proposition}\label{circledast2}
    For each $i\in\{1,2,3,4,5\},$ we have $E_i^\lambda\circledast B_i^\lambda=D_i^\lambda.$
\end{proposition}
\begin{proof}
    Let $S(x,y)\defeq\begin{cases}1&\text{if }y-x\in\{0,1\},\\0&\text{otherwise.}\end{cases}$

    For $h_1(x)=\lambda^{x^2}$ and $h_2(x)=\lambda^x,$ we have $D_3^{\lambda}=h_2\cdot (h_1\cdot S\circledast D_5^{\lambda^2}),$ as well as $B_3^\lambda=h_2\cdot B_5^{\lambda^2}.$ We can also check that $E_3^{\lambda}=h_2\cdot (h_1\cdot S\circledast E_5^{\lambda^2}),$ as
    \begin{equation*}
        (-1)^{y-x}\qbinom{2y}{y-x}{\lambda^2}+\lambda^{2x+1}(-1)^{y-x-1}\qbinom{2y}{y-x-1}{\lambda^2}=(-1)^{y-x}\qbinom{2y+1}{y-x}{\lambda^2}\frac{\lambda^{2x+1}-1}{\lambda^{2y+1}-1}.
    \end{equation*}
    This means that
    \begin{equation*}
        E_3^\lambda\circledast B_3^\lambda=D_3^\lambda\iff E_5^\lambda\circledast B_5^\lambda=D_5^\lambda.
    \end{equation*}
    We also have $F_2=h_2\cdot F_4^{\lambda^2}$ for all $F\in\{B,D,E\}.$ This means that
    \begin{equation*}
        E_2^\lambda\circledast B_2^\lambda=D_2^\lambda\iff E_4^\lambda\circledast B_4^\lambda=D_4^\lambda.
    \end{equation*}
    Moreover, if $h_3(x)=\lambda^{(x^2-x)/2},$ we have $B_1^\lambda=h_2\cdot B_4^\lambda,$ as well as $D_1^\lambda=h_2\cdot(h_3\cdot S\circledast D_4^\lambda).$ We can also check that $E_1^\lambda=h_2\cdot(h_3\cdot S\circledast E_4^\lambda),$ as
    \begin{equation*}
        (-1)^{y-x}\qbinom{2y}{y-x}{\lambda}\frac{\lambda^x+1}{\lambda^y+1}+\lambda^x(-1)^{y-x-1}\qbinom{2y}{y-x-1}{\lambda}\frac{\lambda^{x+1}+1}{\lambda^y+1}=(-1)^{y-x}\qbinom{2y}{y-x}{\lambda}\frac{\lambda^{2x+1}-1}{\lambda^{x+y+1}-1}.
    \end{equation*}
    This means that
    \begin{equation*}
        E_1^\lambda\circledast B_1^\lambda=D_1^\lambda\iff E_4^\lambda\circledast B_4^\lambda=D_4^\lambda.
    \end{equation*}

    Now we prove the identity for $i\in\{4,5\}.$ For $\beta\in\{0,1/2\},$ we consider $F_\beta\in I(\Z_{\ge0},\Z(\lambda))$ to be 
    \begin{equation*}
    F_0(x,y)\defeq (-1)^{y-x}\lambda^{x^2-y^2}\qbinom{2y}{y-x}{\lambda^2}\frac{\lambda^{y+x}+\lambda^{y-x}}{\lambda^{2y}+1}\quad\text{and}\quad F_{1/2}(x,y)\defeq(-1)^{y-x}\lambda^{x^2-y^2}\qbinom{2y}{y-x}{\lambda^2}.
    \end{equation*}
    Denote also $D(x,y)\defeq\binom{y}{\frac{y-x}{2}}.$ Multiplying by $h(x)=\lambda^{x-x^2}$ resp.\ $h(x)=\lambda^{-x^2},$ the identities for $i=4$ resp.\ $i=5$ are equivalent to
    \begin{equation*}
        F_\beta\circledast W_\beta^{\lambda^2}\isequal D\quad\text{for }\beta\in\{0,1/2\}.
    \end{equation*}
    We extend $F_\beta,W_\beta^{\lambda^2},D$ to elements of $I(\Z,\Z(\lambda))$ by denoting them to be $0$ if $y<0,$ and by
    \begin{equation*}
        F_\beta(x,y)=F_\beta(-x,y),\quad D(x,y)=D(-x,y),\quad W_\beta^{\lambda^2}(x,y)=0\quad\text{if }x<0.
    \end{equation*}
    We will prove that $F_\beta\circledast W_\beta^{\lambda^2}\isequal D$ as elements of $I(\Z,\Z(\lambda)).$ Note that $D(x,y)$ satisfies the recurrence
    \begin{equation*}
        D(x,y)=D(x-1,y-1)+D(x+1,y-1)\quad\text{for }y\ge1,\ x\in\Z.
    \end{equation*}
    We prove that $(F_\beta\circledast W_\beta^{\lambda^2})(x,y)\isequal D(x,y)$ by induction on $y^2-x^2.$ The base case is when $\lvert x\rvert\ge y,$ where both sides are $0$ unless $\lvert x\rvert=y,$ in which case they are both $1.$ By induction hypothesis, if $\lvert x\rvert<y,$ we have
    \begin{equation*}
    \begin{split}
        D(x,y)&=D(x-1,y-1)+D(x+1,y-1)\\&=\sum_{z=x-1}^{y-1}F_{\beta}(x-1,z)W_\beta^{\lambda^2}(z,y-1)+\sum_{z=x+1}^{y-1}F_{\beta}(x+1,z)W_\beta^{\lambda^2}(z,y-1)
    \end{split}
    \end{equation*}
    and we want to prove
    \begin{equation*}
    \begin{split}
        D(x,y)&\isequal\sum_{z=x}^yF_\beta(x,z)W_\beta^{\lambda^2}(z,y)\\
        &\stackrel{\eqref{Wrec2}}{=}\sum_{z=x}^yF_{\beta}(x,z)\left((\lambda^{2z+2\beta}+\lambda^{-2z-2\beta})W_\beta^{\lambda^2}(z,y-1)+W_\beta^{\lambda^2}(z-1,y-1)\right)\\
        &=\sum_{z=x}^{y-1}F_{\beta}(x,z)(\lambda^{2z+2\beta}+\lambda^{-2z-2\beta})W_\beta^{\lambda^2}(z,y-1)+\sum_{z=x-1}^{y-1}F_{\beta}(x,z+1)W_\beta^{\lambda^2}(z,y-1).
    \end{split}
    \end{equation*}
    Hence, it remains to check that
    \begin{equation*}
        F_\beta(x-1,y)+F_\beta(x+1,y)\isequal(\lambda^{2y+2\beta}+\lambda^{-2y-2\beta})F_\beta(x,y)+F_\beta(x,y+1)\quad\text{for all }y\ge \max(0,x-1).
    \end{equation*}
    For $\beta=1/2,$ this is equivalent to (changing $\lambda^2$ to $\lambda$)
    \begin{equation*}
        \qbinom{2(y+1)}{(y+1)-x}{\lambda}\isequal(\lambda^{2y+1}+1)\qbinom{2y}{y-x}{\lambda}+\lambda^{y-x+1}\qbinom{2y}{y-x+1}{\lambda}+\lambda^{x+y+1}\qbinom{2y}{y-x-1}{\lambda}
    \end{equation*}
    which is easily verified. For $\beta=0,$ this is equivalent to (changing $\lambda^2$ to $\lambda$)
    \begin{equation*}
    \begin{split}
        \qbinom{2(y+1)}{(y+1)-x}{\lambda}\frac{\lambda^x+1}{\lambda^{y+1}+1}\isequal\ &(\lambda^{2y}+1)\qbinom{2y}{y-x}{\lambda}\frac{\lambda^x+1}{\lambda^y+1}+\lambda^{y-x+1}\qbinom{2y}{y-(x-1)}{\lambda}\frac{\lambda^{x-1}+1}{\lambda^y+1}\\
        &+\lambda^{x+y}\qbinom{2y}{y-(x+1)}{\lambda}\frac{\lambda^{x+1}+1}{\lambda^y+1}
    \end{split}
    \end{equation*}
    which is also easily verified.
\end{proof}

\subsection{Combinatorics over finite fields}
Let $V$ be a finite dimensional $\F_q$-vector space of dimension $2r$ for some $r\ge1.$ We consider nondegenerate pairings $\langle\cdot,\cdot\rangle\colon V\times V\to \F_q$ which are
\begin{case}{\Ugen}
$\F_q/\F_{q_0}$-Hermitian,
\end{case}
\begin{case}{\Ogen}
symmetric,
\end{case}
\begin{case}{\SPgen}
alternating.
\end{case}
We recall some of the notation and results from \cite[Section 2]{CoratoZanarella}. It is well known that the isomorphism classes of such $(V,\langle\cdot,\cdot\rangle)$ are in bijection with $\{(0,1)\}\sqcup(\Z_{>0}\times\mathrm{Sign})$ via
\index{typ@$\mathrm{typ}$}\index{typ@$\mathrm{typ}$! |seealso{\cite[\citetypdef]{CoratoZanarella}}}
\begin{equation*}
    \mathrm{typ}(V)\defeq\left(\frac{1}{\cfactor}\dim_{\F_q}V,\ \mathrm{sign}(\det(V,\langle\cdot,\cdot\rangle))\right).
\end{equation*}

\begin{definition}[{\cite[\citetypFinite]{CoratoZanarella}}]
    For $(V,\langle\cdot,\cdot\rangle)$ with $\mathrm{typ}(V)=(r,\chi)\in\Z_{>0}\times\mathrm{Sign},$ we consider:
\begin{enumerate}
    \item Given $0\le b\le r$ we denote
    \index{S(b,r,chi)@$S\TypeC{b}{}{r}{\chi}$}\index{S(b,r,chi)@$S\TypeC{b}{}{r}{\chi}$! |seealso{\cite[\citetypFinite]{CoratoZanarella}}}
    \begin{equation*}
        S\TypeC{b}{}{r}{\chi}\defeq\#\{W\subseteq V\colon \dim W=b,\ W\subseteq W^\perp\}
    \end{equation*}
    to be the number of $b$-dimensional isotropic subspaces of $V.$
    \item Given $0\le a\le r$ and $\eta\in\mathrm{Sign},$ we denote
    \index{R(a,eta,r,chi)@$R\TypeC{a}{\eta}{r}{\chi}$}\index{R(a,psi,r,chi)@$R\TypeC{a}{\eta}{r}{\chi}$! |seealso{\cite[\citetypFinite]{CoratoZanarella}}}
    \begin{equation*}
        R\TypeC{a}{\eta}{r}{\chi}\defeq\#\{W\subseteq V\colon W\text{ is nondegenerate, }\mathrm{typ}(W)=(a,\eta)\}.
    \end{equation*}
\end{enumerate}
\end{definition}
\begin{proposition}[{\cite[\citeRprop, \citeSprop]{CoratoZanarella}}]
 Denote
\index{e(r,chi)@$e(r,\chi)$}
\begin{equation*}
    e(r,\chi)\defeq\begin{cases}q^{r/2}+\epsilon^{r/2}\chi&\text{if }r\text{ is even,}\\1&\text{if }r\text{ is odd.}\end{cases}
\end{equation*}
Then we have
\begin{equation*}
    R\TypeC{a}{\eta}{r}{\chi}=\begin{cases}\displaystyle(-q_0)^{a(r-a)}\qbinom{r}{a}{-q_0}&\text{in the case \Ugen,}\\\displaystyle\frac{q^{\lfloor a(r-a)/2\rfloor}}{2}\frac{\lfloor r/2\rfloor_{q^2}}{\lfloor a/2\rfloor_{q^2}\lfloor(r-a)/2\rfloor_{q^2}}\frac{e(a,\eta)e(r-a,\chi\eta)}{e(r,\chi)}&\text{in the case \Ogen,}\\\displaystyle q^{2a(r-a)}\qbinom{r}{a}{q^2}&\text{in the case \SPgen.}\end{cases}
\end{equation*}
and
\begin{equation*}
    S\TypeC{b}{}{r}{\chi}=\begin{cases}\displaystyle(-q_0;q_0^2)_b\qbinom{r}{2b}{-q_0}&\text{in the case \Ugen,}\\\displaystyle(-q;q)_b\qbinom{\lfloor r/2\rfloor}{b}{q^2}\frac{e(r-2b,\chi\epsilon^b)}{e(r,\chi)}&\text{in the case \Ogen,}\\\displaystyle(-q;q)_b\qbinom{r}{b}{q^2}&\text{in the case \SPgen.}\end{cases}
\end{equation*}
\end{proposition}

For the rest of this subsection, we assume that $V$ has a Lagrangian subspace. That is, we assume that $\mathrm{typ}(V)=(\frac{2}{\cfactor}\cdot r,\epsilon^r).$

\begin{definition}\label{Ldefinition}\index{L(a,b,psi,r)@$L\Typec{a}{}{b}{\chi}{r}{}$}
For $0\le a,b\le r,$ we denote $L\Typec{a}{}{b}{\chi}{r}{}$ to be, for a fixed $0\sub{a} L\cap L^\perp\subseteq L\subseteq V$ with $\mathrm{typ}(L/L\cap L^\perp)=(b,\chi),$ the number of Lagrangians in $V$ which are disjoint from $L.$
\end{definition}

\begin{proposition}\label{Lquantitysubsub}
We have
\begin{equation*}
    L\Typec{r}{}{0}{1}{r}{}=\Dfactor(r).
\end{equation*}
\end{proposition}
\begin{proof}
We are counting the number of Lagrangians which are disjoint to a fixed Lagrangian $L.$ The total number of Lagrangians is (see \cite[\citeSprop]{CoratoZanarella})
\begin{equation*}
    \begin{cases}(-q_0;q_0^2)_r & \text{in the case \Ugen,}\\
    (-1;q)_r & \text{in the case \Ogen,}\\
    (-q;q)_r & \text{in the case \SPgen.}\end{cases}
\end{equation*}
So considering the possible intersections with $L,$ we have
\begin{equation*}
    \sum_{i=0}^r\qbinom{r}{i}{q}L\Typec{r-i}{}{0}{1}{r-i}{}=\begin{cases}(-q_0;q_0^2)_r & \text{in the case \Ugen,}\\
    (-1;q)_r & \text{in the case \Ogen,}\\
    (-q;q)_r & \text{in the case \SPgen.}\end{cases}
\end{equation*}
which we can rewrite as
\begin{equation*}
    \sum_{i=0}^r\qbinom{r}{i}{q}L\Typec{i}{}{0}{1}{i}{}=\begin{cases}(-q_0;q_0^2)_r & \text{in the case \Ugen,}\\
    (-1;q)_r & \text{in the case \Ogen,}\\
    (-q;q)_r & \text{in the case \SPgen.}\end{cases}
\end{equation*}
These relations for all $r$ fully determine the $L\Typec{i}{}{0}{1}{i}{}.$ By the $q$-binomial theorem, we have
\begin{equation*}
    \sum_{i=0}^r\qbinom{r}{i}{q}q^{\binom{i}{2}+\alpha i}=(-q^\alpha;q)_r,
\end{equation*}
and now the desired formula for $L\Typec{i}{}{0}{1}{i}{}$ follows.
\end{proof}
\begin{proposition}\label{Lquantitysub}
    We have
    \begin{equation*}
        L\Typec{a}{}{b}{\chi}{r}{}=L\Typec{0}{}{b}{\chi}{r-a}{}\frac{\Dfactor(r)}{\Dfactor(r-a)}.
    \end{equation*}
\end{proposition}
\begin{proof}
Fix $L$ a subspace satisfying $0\sub{a} L\cap L^\perp\subseteq L\subseteq V$ with $\mathrm{typ}(L/L\cap L^\perp)=(b,\chi).$

Suppose $I$ is a Lagrangian which is disjoint to $L.$ We consider: $L_{I,0}\defeq I\cap (L+L^\perp)$ and $L_{I,1}\defeq L_{I,0}+(L\cap L^\perp).$ Note that
\begin{equation*}
L_{I,1}=(I+(L\cap L^\perp))\cap(L+L^\perp)=(I\cap(L+L^\perp))^\perp\cap(L\cap L^\perp)^\perp=L_{I,1}^\perp,
\end{equation*}
that is, $L_{I,1}$ is a Lagrangian. Note also that since $I\cap L=0,$ we have $L_{I,0}\cap (L\cap L^\perp)=0,$ $L_{I,1}\cap L=L\cap L^\perp.$ Pictorially:
\begin{equation*}
\begin{tikzcd}
&V\arrow[ldd,phantom,"\join",very near start]&\\I\arrow[ur,dash,"r"]&L+L^\perp\arrow[u,dash,"a"]&\\L_{I,0}\arrow[uur,phantom,"\meet",very near start]\arrow[ur,dash,"r"]\arrow[r,dash,"a"]\arrow[u,dash,"a"]&L_{I,1}\arrow[dl,phantom,"\join",very near start]\arrow[u,dash,"r-a"]&\phantom{a}\\0\arrow[u,dash,"r-a"]\arrow[r,dash,"a"]\arrow[ur,phantom,"\meet",very near start]&L\cap L^\perp\arrow[ur,phantom,"\meet",very near start]\arrow[u,dash,"r-a"]\arrow[r,dash,"{(b,\chi)}"]&L
\end{tikzcd}
\end{equation*}

Conversely, fix a Lagrangian $L_1$ as above: satisfying $L_1\subseteq L+L^\perp$ and $L_1\cap L=L\cap L^\perp.$ The number of such $L_1$ is $L\Typec{0}{}{b}{\chi}{r-a}{}.$ Further fix $L_0$ a complement of $L\cap L^\perp$ in $L_1.$ There are $q^{a(r-a)}$ such $L_0.$ Finally, the number of $I$ such that $L_{I,0}=L_0$ and $L_{I,1}=L_1$ is simply the number of Lagrangians in $L_0^\perp/L_0$ which are disjoint to $L_1/L_0.$ Thus by \Cref{Lquantitysubsub}, we have
\begin{equation*}
L\Typec{a}{}{b}{\chi}{r}{}=L\Typec{0}{}{b}{\chi}{r-a}{}\cdot q^{a(r-a)}\cdot \Dfactor(a)
\end{equation*}
and now the claim follows from the fact that $\Dfactor(a)\Dfactor(r-a)q^{a(r-a)}=\Dfactor(r).$
\end{proof}

\begin{proposition}\label{Lquantity}
We have
\begin{equation*}
    L\Typec{0}{}{b}{\chi}{r}{}=\begin{cases}(-q_0)^{\binom{b}{2}}\frac{(2r-b)_{-q_0}}{(r-b)_{q_0^2}} & \text{in the case \Ugen,}\\
    2q^{\lfloor b/2\rfloor\cdot\lfloor(b-1)/2\rfloor}\frac{(r-\lceil b/2\rceil)_{q^2}}{(r-b)_q\cdot e(2r-b,\chi\epsilon^r)} & \text{in the case \Ogen,}\\
    q^{b^2}\frac{(r-b)_{q^2}}{(r-2b)_q} & \text{in the case \SPgen.}\end{cases}
\end{equation*}
\end{proposition}
\begin{proof}
Fix $L\subseteq V$ nondegenerate with $\mathrm{typ}(L)=(b,\chi).$ Denote $L_c\TypeC{b}{\chi}{r}{}$ to be the number of Lagrangians $I$ such that the intersection $I\cap L$ has dimension $c.$ Note that $L_0\TypeC{b}{\chi}{r}{}=L\Typec{0}{}{b}{\chi}{r}{}.$

We claim that
\begin{equation}\label{Lsubclaim}
    L_c\TypeC{b}{\chi}{r}{}=S\TypeC{c}{}{b}{\chi}\cdot L\Typec{0}{}{b-\frac{2}{\cfactor}\cdot c}{\chi\epsilon^c}{r-c}{}.
\end{equation}
This would allow us to recursively compute $L_0\TypeC{b}{\chi}{r}{},$ as we can write the total number of Lagrangians of $V$ as $\sum_{c\ge0}L_c\TypeC{b}{\chi}{r}{}.$ To see this claim, note that if $I$ is a Lagrangian, then $S_I\defeq I\cap L$ is isotropic. Conversely, given $0\sub{c}S\subseteq L$ isotropic, we have
\begin{equation*}
    (L\cap S^\perp)\cap(L\cap S^\perp)^\perp=(L\cap S^\perp)\cap(L^\perp+S)=L\cap(L^\perp+S)=S+(L\cap L^\perp)=S.
\end{equation*}
That is to say, $L\cap S^\perp$ is nondegenerate in $S^\perp/S.$ Thus the number of Lagrangians $I$ with $S_I=S$ is $L\Typec{0}{}{b-\frac{2}{\cfactor}\cdot c}{\chi\epsilon^c}{r-c}{}.$ Hence \eqref{Lsubclaim} follows.

With that, we have
\begin{equation*}
\sum_{c=0}^{\lfloor \cfactor\cdot b/2\rfloor}S\TypeC{c}{}{b}{\chi}\cdot L\Typec{0}{}{b-\frac{2}{\cfactor}\cdot c}{\chi\epsilon^c}{r-c}{}=\begin{cases}(-q_0;q_0^2)_r & \text{in the case \Ugen,}\\
    (-1;q)_r & \text{in the case \Ogen,}\\
    (-q;q)_r & \text{in the case \SPgen.}\end{cases}
\end{equation*}
As we vary $0\le b\le r,$ and $\chi,$ this recursively determine all the $L\Typec{0}{}{b}{\chi}{r}{}.$
\begin{case}{\Ugen}
We need to prove
\begin{equation*}
    \sum_{c\ge0}(-q_0,q_0^2)_c\qbinom{b}{2c}{-q_0}(-q_0)^{\binom{b-2c}{2}}\frac{(2r-b)_{-q_0}}{(r+c-b)_{q_0^2}}\isequal(-q_0;q_0^2)_r.
\end{equation*}
This is equivalent to
\begin{equation*}
    \sum_{c\ge0}(-q_0)^{\binom{b-2c}{2}}\qbinom{r}{c,r-b+c,b-2c}{q^2}(q_0;-q_0)_{b-2c}\isequal\qbinom{2r}{b}{-q_0}
\end{equation*}
which is exactly \Cref{qidentityforL} for $(a,b,\lambda)=(r-b,b,-q_0^{-1}).$
\end{case}
\begin{case}{\Ogen}
We need to prove
\begin{equation*}
    \sum_{c\ge0}(-q;q)_c\qbinom{\lfloor b/2\rfloor}{c}{q^2}\frac{e(b-2c,\chi\epsilon^c)}{e(b,\chi)}\cdot 2q^{(\lfloor b/2\rfloor-c)\cdot(\lfloor(b-1)/2\rfloor-c)}\frac{(r-\lceil b/2\rceil)_{q^2}}{(r-b+c)_q\cdot e(2r-b,\chi\epsilon^r)}\isequal(-1;q)_r.
\end{equation*}
This is equivalent to
\begin{equation*}
    \sum_{c\ge0}q^{\lfloor b/2\rfloor\cdot\lfloor(b-1)/2\rfloor+c^2+c-bc}\frac{(r)_q\cdot e(b-2c,\chi\epsilon^c)}{(\lfloor b/2\rfloor-c)_{q^2}(c)_q(r-b+c)_q)}\isequal\frac{(r)_{q^2}\cdot e(b,\chi)e(2r-b,\chi\epsilon^r)}{\lfloor b/2\rfloor_{q^2}(r-\rceil b/2\rceil)_{q^2}(q^r+1)}.
\end{equation*}

For $b=2k+1$ and changing variables $c\mapsto k-c,$ this is
\begin{equation*}
    \sum_{c\ge0}q^{c^2}(-1)^c(q;q^2)_c\qbinom{r-1}{2c,k-c,r-k-1-c}{q}\isequal\qbinom{r-1}{k}{q^2}
\end{equation*}
which is exactly \Cref{qidentityforR} for $(a,b,\lambda)=(r-k-1,k,q^{-1}).$

For $b=2k$ and changing variables $c\mapsto k-c,$ this is
\begin{equation*}
    \sum_{c\ge0}q^{c^2}(-1)^c(q;q^2)_c\qbinom{r}{2c,k-c,r-k-c}{q}\left(1+\chi\epsilon^k q^{-c}\right)\isequal\qbinom{r}{k}{q^2}\left(1+\chi\epsilon^k\frac{q^k+q^{r-k}}{q^r+1}\right).
\end{equation*}
The part without $\chi$ is again \Cref{qidentityforR} for $(a,b,\lambda)=(r-k,k,q^{-1}),$ and the part with $\chi$ is \Cref{qidentityforR2} for $(a,b,\lambda)=(r-k,k,q^{-1}).$
\end{case}
\begin{case}{\SPgen}
We need to prove
\begin{equation*}
    \sum_{c\ge0}(-q;q)_c\qbinom{b}{c}{q^2}q^{(b-c)^2}\frac{(r-b)_{q^2}}{(r+c-2b)_q}\isequal(-q;q)_r.
\end{equation*}
This is equivalent to
\begin{equation*}
    \sum_{c\ge0}q^{(b-c)^2}\qbinom{r}{c,r-2b+c,2(b-c)}{q}(q;q^2)_{b-c}(-1)^{b-c}\isequal\qbinom{r}{b}{q^2}.
\end{equation*}
This is
\begin{equation*}
    \sum_{c\ge0}q^{c^2}(-1)^{c}\qbinom{r}{2c,b-c,r-b-c}{q}(q;q^2)_{c}\isequal\qbinom{r}{b}{q^2}
\end{equation*}
which is exactly \Cref{qidentityforR} for $(a,b,\lambda)=(r-b,b,q^{-1}).$\qedhere
\end{case}
\end{proof}

%% file: 3-Lattices.tex
\section{Preliminaries}
We consider a finite dimensional $F$-vector space $V.$ Let $\Lat(V)$\index{Lat(V)@$\Lat(V)$} denote the set of $\O_F$-lattices of $V.$ We will consider forms $\langle\cdot,\cdot\rangle\colon V\times V\to F$ which are
\begin{case}{\Ugen}
Hermitian over $F/F_0,$
\end{case}
\begin{case}{\Ogen}
symmetric,
\end{case}
\begin{case}{\SPgen}
alternating.
\end{case}
We will further assume that $(V,\langle\cdot,\cdot\rangle)$ has a self-dual lattice, and we let $\Lat^\circ(V)\subseteq\Lat(V)$\index{Latcirc(V)@$\Lat^\circ(V)$} denote the subset of such self-dual $\O_F$-lattices. We denote $G(V)=G(V,\langle\cdot,\cdot\rangle)\subseteq\mathrm{GL}(V)$ to be the $p$-adic group of automorphisms that respect $\langle\cdot,\cdot\rangle.$

\subsection{Preliminaries on lattices}\label{LatticeSection}
Let $V$ be a vector space as above, and assume we are given direct sum decomposition $V=V_1\oplus V_2.$ For $i\in\{1,2\},$ we consider the maps $\mathrm{int}_i\colon\Lat(V)\to\Lat(V_i)$\index{int12@$\mathrm{int}_1,\ \mathrm{int}_2$} and $\mathrm{proj}_i\colon\Lat(V)\to\Lat(V_i)$\index{proj12@$\mathrm{proj}_1,\ \mathrm{proj}_2$} where $\mathrm{int}_i(\Lambda)=\Lambda\cap V_i$ and $\mathrm{proj}_i(\Lambda)$ is the projection of $\Lambda$ onto $V_i.$
\begin{proposition}\label{LatBij}
We have a bijection
\begin{equation*}
    \Lat(V)\longleftrightarrow\left\{(L_1^-,L_1^+,L_2^-,L_2^+,\varphi)\colon \begin{array}{l}
        L_1^-,L_1^+\in\Lat(V_1),\ L_1^-\subseteq L_1^+,\\
        L_2^-,L_2^+\in\Lat(V_2),\ L_2^-\subseteq L_2^+,\\
        \varphi\colon L_1^+/L_1^-\rightiso L_2^+/L_2^-
    \end{array}\right\},
\end{equation*}
given by $L\mapsto(\mathrm{int}_1(\Lambda),\mathrm{proj}_1(\Lambda),\mathrm{int}_2(\Lambda),\mathrm{proj}_2(\Lambda),\varphi)$ where $\varphi(x)=y$ if and only if $(\tilde{x},\tilde{y})\in\Lambda$ for (any) lifts $\tilde{x},\tilde{y}\in\Lambda.$
\end{proposition}
\begin{proof}
Given $L\in\Lat(V),$ denote $L_i^-\defeq\mathrm{int}_i(L)$ and $L_i^+\defeq\mathrm{proj}_i(L).$ Then we have $L_1^-\oplus L_2^-\subseteq L\subseteq L_1^+\oplus L_2^+$ and $\O_F$-module isomorphisms
\begin{equation*}
\begin{tikzcd}
    &L/(L_1^-\oplus L_2^-)\arrow{dl}{\sim}[swap]{\mathrm{proj}_1}\arrow{dr}{\mathrm{proj}_2}[swap]{\sim}&\\
    L_1^+/L_1^-&&L_2^+/L_2^-
\end{tikzcd}
\end{equation*}
Then $\varphi\colon L_1^+/L_1^-\rightiso L_2^+/L_2^-$ denotes their composition.

Conversely, given $(L_1^-,L_1^+,L_2^-,L_2^+,\varphi),$ the inverse of the above map is simply $L\defeq\{(x,y)\in V\colon \varphi(x)=y\}.$ This is an $\O_F$-module since $\varphi$ is an $\O_F$-module homomorphism, and thus $\Lambda\in\Lat(V)$ since $L_1^-\oplus L_2^-\subseteq L\subseteq L_1^+\oplus L_2^+.$ This is indeed the inverse since $\varphi$ is an isomorphism.
\end{proof}

Now equip $V$ with a form $\langle\cdot,\cdot\rangle$ as above, and we assume that the decomposition $V=V_1\oplus V_2$ is an orthogonal decomposition. We denote $H_1,H_2$ for the group of automorphisms of  $V_1,V_2$ respecting the form. We denote $H\defeq H_1\times H_2.$

\begin{proposition}\label{LatcircBij}
We have a bijection
\begin{equation*}
    \Lat^\circ(V)\longleftrightarrow\left\{(L_1,L_2,\varphi)\colon\begin{array}{cl}L_1\in\Lat(V_1),&L_1\subseteq L_1^\vee,\\L_2\in\Lat(V_2),&L_2\subseteq L_2^\vee,\\\varphi\colon L_1^\vee/L_1\rightiso L_2^\vee/L_2,&\langle\varphi(x_1),\varphi(x_2)\rangle\equiv-\langle x_1,x_2\rangle\mod\O_F\end{array}\right\},
\end{equation*}
given by $\Lambda\mapsto(\mathrm{int}_1(\Lambda),\mathrm{int}_2(\Lambda),\varphi)$ where $\varphi(x)=y$ if and only if $(\tilde{x},\tilde{y})\in\Lambda$ for (any) lifts $\tilde{x},\tilde{y}\in\Lambda.$
\end{proposition}
\begin{proof}
First note that for any $\Lambda\in\Lat(V)$ we have that $\mathrm{int}_i(\Lambda)^\vee=\mathrm{proj}_i(\Lambda^\vee).$

Now let $\Lambda\in\Lat^\circ(V).$ Then the above implies that $\mathrm{proj}_i(\Lambda)=\mathrm{int}_i(\Lambda)^\vee.$ By \Cref{LatBij}, we have $\varphi\colon L_1^\vee/L_1\rightiso L_2^\vee/L_2,$ and since $\Lambda$ is integral, for any $x_1,x_2\in L_1^\vee$ we must have
\begin{equation*}
    \langle\widetilde{x_1},\widetilde{x_2}\rangle+\langle\widetilde{\varphi(x_1)},\widetilde{\varphi(x_2)}\rangle=\langle(\widetilde{x_1},\widetilde{\varphi(x_1)}),(\widetilde{x_2},\widetilde{\varphi(x_2)})\rangle\in\O_F.
\end{equation*}

Conversely, given $(L_1,L_2,\varphi)$ as above, we consider $\Lambda\in\Lat(V)$ associated to $(L_1,L_1^\vee,L_2,L_2^\vee,\varphi)$ by \Cref{LatBij}. If $(x_1,y_1),(x_2,y_2)\in\Lambda,$ then
\begin{equation*}
    \langle(x_1,y_1),(x_2,y_2)\rangle\equiv\langle(x_1,\varphi(x_1)),(x_2,\varphi(x_2))\rangle\equiv\langle x_1,\varphi(x_1)\rangle+\langle x_2,\varphi(x_2)\rangle\in\O_F
\end{equation*}
and thus $\Lambda$ is integral. Thus $L_1\oplus L_2\subseteq\Lambda\subseteq\Lambda^\vee\subseteq L_1^\vee\oplus L_2^\vee.$ Since $\Lambda/(L_1\oplus L_2)\rightiso L_1^\vee/L_1\iso L_2^\vee/L_2,$ we have
\begin{equation*}
    \mathrm{length}\left(\frac{L_1^\vee\oplus L_2^\vee}{L_1\oplus L_2}\right)=2\cdot\mathrm{length}\left(\frac{\Lambda}{L_1\oplus L_2}\right)=\mathrm{length}\left(\frac{\Lambda}{L_1\oplus L_2}\right)+\mathrm{length}\left(\frac{L_1^\vee\oplus L_2^\vee}{\Lambda^\vee}\right)
\end{equation*}
and thus we conclude $\mathrm{length}(\Lambda^\vee/\Lambda)=0,$ which means $\Lambda\in\Lat^\circ(V).$
\end{proof}

\begin{definition}
Denote
\index{Typ0nat@$\Typ^{0,\natural},\ \Typ^{0,\natural}_r$}\index{Typ0nat@$\Typ^{0,\natural},\ \Typ^{0,\natural}_r$! |seealso{\cite[\citeTypzerodef]{CoratoZanarella}}}\index{Typ0@$\Typ^0,\ \Typ^0_r$}
\begin{equation*}
\begin{split}
    \Typ^{0,\natural}&\defeq\{(e^0,\chi^0)\colon  e^0\colon\Z\to\Z_{\ge0},\ \chi^0\colon\Z\to\mathrm{Sign},\text{ s.t. }e^0(i)=0\implies\chi^0(i)=1\},\\
    \Typ^{0}&\defeq\{(e^0,\chi^0)\colon  e^0\colon\Z_{\ge0}\to\Z_{\ge0},\ \chi^0\colon\Z_{\ge0}\to\mathrm{Sign},\text{ s.t. }e^0(i)=0\implies\chi^0(i)=1\}.
\end{split}
\end{equation*}
so that $\Typ^0\subseteq\Typ^{0,\natural}.$ We also denote $\Typ^{0,\natural}_r$ and $\Typ^{0}_r$ the subsets where $\sum_ie^0(i)=r.$
\end{definition}

Recall from \cite[\citelatticeClassification]{CoratoZanarella} that we have an injection
\index{typnat@$\mathrm{typ}^\natural$}\index{typnat@$\mathrm{typ}^\natural$! |seealso{\cite[\citelatticeClassification]{CoratoZanarella}}}
\begin{equation}\label{typninj}
    \mathrm{typ}^\natural\colon G\backslash\Lat(V)\hookrightarrow\Typ^{0,\natural}
\end{equation}
characterized by the fact that if $\mathrm{typ}^\natural(L)=(e^0,\chi^0),$ then there exists an orthogonal decomposition $L=\bigoplus_{i\in\Z}L^{(i)}$ with $L^{(i),\vee}=\varpi^{-i}L^{(i)}$ and $\mathrm{typ}(L^{(i)}/\varpi L^{(i)}),\varpi^{-i}\langle\cdot,\cdot\rangle)=(e^0(i),\chi^0(i)).$

\begin{proposition}\label{typhelp}
Let $L\in\Lat(V)$ be an integral lattice. Assume we are given $\varphi_0\colon L^\vee/L\rightiso L^\vee/L$ satisfying $\langle\varphi_0(x_1),\varphi_0(x_2)\rangle\equiv\langle x_1,x_2\rangle\mod\O_F.$ Then there exist $g\in G$ such that $gL=L$ and such that $g$ induces $\varphi_0.$
\end{proposition}
\begin{proof}
Choose a decomposition $L=\bigoplus_{i\ge0}L^{(i)}$ as above. Then $L^\vee/L\iso\bigoplus_{i>0}\varpi^{-i}L^{(i)}/L^{(i)}.$ If $\dim V^{(0)}>0,$ we may proceed by induction on $\dim V$ by considering $L^{(\ge1)}\in\Lat(V^{(\ge1)}).$ So we may further assume $\dim L^{(0)}=0.$ This means that $L\subseteq\varpi L^\vee.$

Now we proceed by Hensel lifting. We will prove the following sub-claim: if $L\in\Lat(V)$ is a lattice with $L\subseteq\varpi L^\vee$ and we are given $n\ge0$ and $\varphi_n\colon L^\vee/\varpi^nL\rightiso L^\vee/\varpi^nL$ satisfying
\begin{equation*}
    \langle\varphi_n(x_1),\varphi_n(x_2)\rangle\equiv\langle x_1,x_2\rangle\mod\varpi^n\O_F,
\end{equation*}
then there exist $\varphi_{n+1}\colon L^\vee/\varpi^{n+1}L\rightiso L^\vee/\varpi^{n+1}L$ extending $\varphi_n$ and satisfying
\begin{equation*}
    \langle\varphi_{n+1}(x_1),\varphi_{n+1}(x_2)\rangle\equiv\langle x_1,x_2\rangle\mod\varpi^{n+1}\O_F.
\end{equation*}

Given this sub-claim, we can then choose a compatible system $(\varphi_n)_{n\ge0}$ with $\varphi_0$ as given. Then the limiting map gives the element $g$ we are looking for.

We first consider the case $n=0.$ Suppose we are given $\varphi_0$ as above. We choose an arbitrary lift $\widetilde{\varphi_0}\colon L^\vee/\varpi L\rightiso L^\vee/\varpi L.$ We will construct $\varphi_{1}$ as $\widetilde{\varphi_0}+\phi$ for some well-chosen map $\phi\colon L^\vee/\varpi L^\vee\to L/\varpi L.$ Note that
\begin{equation*}
    \langle\widetilde{\varphi_0}(x_1)+\phi(x_1),\widetilde{\varphi_0}(x_2)+\phi(x_2)\rangle\equiv\langle\widetilde{\varphi_0}(x_1),\widetilde{\varphi_0}(x_2)\rangle+\langle\widetilde{\varphi_0}(x_1),\phi(x_2)\rangle+\langle\phi(x_1),\widetilde{\varphi_0}(x_2)\rangle\mod\varpi\O_F
\end{equation*}
as we are assuming that $L\subseteq\varpi L^\vee.$ By assumption, we have
\begin{equation*}
    B(x_1,x_2)\defeq \langle\widetilde{\varphi_0}(x_1),\widetilde{\varphi_0}(x_2)\rangle-\langle x_1,x_2\rangle\in\O_F.
\end{equation*}
Now consider the $\O_F/\varpi\O_F$-vector space $W\defeq L/\varpi L.$ Note that if $W^*\defeq L^\vee/\varpi L^\vee,$ the form $\langle\cdot,\cdot\rangle$ induces a perfect pairing $W\otimes W^*\to\O_F/\varpi\O_F.$ Now $\widetilde{\varphi_0}$ induces a map $\widetilde{\varphi_0}\colon W^*\rightiso W^*,$ and we are reduced to find $\phi\colon W^*\to W$ such that for $x_1,x_2\in W^*,$ we have
\begin{equation*}
    B(x_1,x_2)=\langle\phi(x_1),\widetilde{\varphi_0}(x_2)\rangle+\langle\widetilde{\varphi_0}(x_1),\phi(x_2)\rangle.
\end{equation*}
This can always be done as we are assuming that $\widetilde{\varphi_0}$ is an isomorphism, as well as i) in the case \Ugen, that $F/F_0$ is unramified, ii) in the case \Ogen, that $p>2.$

Now we note that the general sub-claim follows from the case $n=0.$ This is because the sub-claim for $n\ge1$ and $(V,\langle\cdot,\cdot\rangle)$ follows from the sub-claim for $n=0$ and $(V,\varpi^{-n}\langle\cdot,\cdot\rangle).$
\end{proof}

Denote $r_i=\frac{2}{\cfactor}\dim_{F}V_i.$ Assume we are given $\Lambda\in\Lat^\circ(V),$ and denote $(e_i^0,\chi_i^0)=\mathrm{typ}^\natural(\mathrm{int}_i(\Lambda))$ for $i\in\{1,2\}.$ Then \Cref{LatcircBij}, means that we have
\begin{equation*}
    e_1^0(i)=e_2^0(i),\quad \chi_1^0(i)=\epsilon^{e_2^0(i)}\chi_2^0(i)\quad \text{for }i\ge1
\end{equation*}
and thus
\begin{equation}\label{eqtyp}
    r_1-e_1^0(0)=r_2-e_2^0(0),\quad \chi_1^0(0)=\epsilon^{r_2-e_2^0(0)}\chi_2^0(0)\cdot\mathrm{sign}(\det(V)).
\end{equation}
In other words, $(e_1^0,\chi_1^0)$ and $(e_2^0,\chi_2^0)$ determine each other.
\begin{proposition}\label{typinj}
    For each $i\in\{1,2\},$ the map $\mathrm{typ}^\natural\circ\mathrm{int}_i\colon\Lat(V)\to\Typ^0$ induces an injection $\mathrm{typ}^\natural\circ\mathrm{int}_i\colon H\backslash\Lat^\circ(V)\hookrightarrow\Typ^0.$
\end{proposition}
\begin{proof}
By the above, it suffices to see that if $\Lambda,\Lambda'\in\Lat(V)$ agree under both $\mathrm{typ}^\natural\circ\mathrm{int}_1$ and $\mathrm{typ}^\natural\circ\mathrm{int}_2,$ then there exist $h\in H$ with $h\Lambda=\Lambda'.$

By \eqref{typninj} for both $V_1,V_2,$ we conclude that if $\Lambda,\Lambda'\in\Lat^\circ(V)$ agree under $\Typ^0,$ then we can find $(g_1,g_2)\in H$ such that $((g_1,g_2)\Lambda)\cap V_i=\Lambda'\cap V_i$ for both $i\in\{1,2\}.$ So we may assume without loss of generality that $\Lambda\cap V_i=\Lambda'\cap V_i$ for both $i\in\{1,2\}.$

So under \Cref{LatcircBij}, $\Lambda,\Lambda'$ correspond to $(L_1,L_2,\varphi),(L_1,L_2,\varphi').$ Now let $\varphi_0\colon L_2^\vee/L_2\rightiso L_2^\vee/L_2$ be $\varphi_0=\varphi'\circ\varphi^{-1}.$ Note that it satisfies $\langle\varphi_0(y_1),\varphi_0(y_2)\rangle\equiv\langle y_1,y_2\rangle\mod\O_F.$ By \Cref{typhelp}, we can find $g_2\in H_2$ such that $g_2L_2=L_2$ and such that $g_2$ induces $\varphi_0.$ Then $(1,g_2)\cdot\Lambda$ corresponds to $(L_1,L_2,\varphi_0\circ\varphi)=(L_1,L_2,\varphi'),$ that is, $(1,g_2)\cdot\Lambda=\Lambda'.$
\end{proof}

\begin{proposition}\label{typbij}
Assume without loss of generality that $r_1\le r_2.$ Then the injection
\begin{equation*}
    \mathrm{typ}^\natural\circ\mathrm{int}_1\colon H\backslash\Lat^\circ(V)\hookrightarrow \mathrm{typ}^\natural(\Lat(V_1))\cap \Typ^0
\end{equation*}
is a bijection if and only if either i) $r_1<r_2$ or ii) $r_1=r_2$ and $\mathrm{sign}(\det(V))=\epsilon^{r_1}.$
\end{proposition}
\begin{proof}
If $r_1=r_2$ and the image of $\mathrm{typ}_{V_1}$ contained an element with $e^0(0)=0,$ then \eqref{eqtyp} would imply that $\mathrm{sign}(\det(V))=\epsilon^{r_1}.$

Conversely, if $\mathrm{sign}(\det(V))=\epsilon^{r_1}.$ then by \eqref{eqtyp}, we must have $f\colon V_1\rightiso(-V_2).$ Then for any $L\in\Lat(V_1)$ integral, we have a lattice $\Lambda\in\Lat^\circ(V)$ given by the data $(L,f(L),f)$ as in \Cref{LatcircBij}. It satisfies $\mathrm{typ}_{V_1}(\Lambda)=\mathrm{typ}^\natural(L).$

For the case $r_1<r_2,$ we can find a subspace $V_2'\subseteq V_2$ with $V_1\iso(-V_2'),$ and then an orthogonal decomposition $V_2=(V_2')^\perp\oplus V_2'.$ Then if $V'\defeq V_1\oplus V_2',$ we certainly have $\mathrm{typ}_{V_1}(\Lat^\circ(V))\supseteq\mathrm{typ}_{V_1}(\Lat^\circ(V')),$ and thus the claim follows from the previous case.
\end{proof}
\begin{remark}
In the case $r_1=r_2$ but $V_1\not\iso(-V_2),$ the image is $\mathrm{typ}^\natural(\Lat(V_1))\cap\{(e^0,\chi^0)\in\Typ^0\colon e^0(0)>0\}.$
\end{remark}

\subsection{Hecke algebras}\label{Heckesection}
Consider $(V,\langle\cdot,\cdot\rangle)$ as above, and we fix a $\Lambda^\circ_{fix}\in\Lat^\circ(V).$ This determines a hyperspecial subgroup $K\subseteq G$ given by $K=\mathrm{Stab}(\Lambda^\circ_{fix}).$ Let $N\defeq\dim_FV.$
\begin{definition}\label{Hdef}
    We consider the ($R$-valued) Hecke algebra $\mathcal{H}(G,K)\defeq C_c^\infty(K\backslash G\slash K,R)$ to be the convolution algebra of ($R$-valued) locally constant compactly supported $K$ bi-invariant functions from $G$ to $R.$
\end{definition}

We may choose an $\O_F$-basis $e_1,\ldots,e_N$ of $\Lambda^{\circ}_{fix}$ such that $\langle e_i, e_j\rangle$ is nonzero if and only if $i+j=N+1.$ We let $A\subseteq G$ denote the subgroup of automorphisms which act as scalar on every $e_i$ for $1\le i\le N.$ As $K$ is hyperspecial, we have the Cartan decomposition $G=KAK,$ and thus $\mathcal{H}(G,K)\rightiso R[A/A\cap K]^W$ where $W=N_G(A)/A$ is the Weyl group of $A$ in $G.$

\begin{definition}\label{Tidef}
    For $\underline{t}=(t_1,\ldots,t_N)\in\Z^N$ satisfying $t_i+t_{N+1-i}=0,$ and for $a\in F_0^\times,$ we denote $a^{\underline{t}}\in A$ to be such that $a^{\underline{t}}\cdot e_i=a^{t_i}e_i.$ For $0\le k\le\lfloor N/2\rfloor,$ we denote
    \begin{equation*}
        \underline{t}_k\defeq(\underbrace{1,\ldots,1}_k,\underbrace{0,\ldots,0}_{N-2k},\underbrace{-1,\ldots,-1}_k)
    \end{equation*}\index{1muk@$\mu_k$}
    and we let $T_k\in\mathcal{H}(G,K)$ be $T_k\defeq\mathrm{char}(K\varpi^{\underline{t}_k}K).$
\end{definition}
\begin{remark}
    We note that the $T_i\in\mathcal{H}(G,K)$ do not depend on the choice of basis $e_i$ as above, and it is well-known that $\mathcal{H}(G,K)=R[T_1,\ldots,T_{\lfloor N/2\rfloor}].$
\end{remark}

We now fix some notation for the Satake transform of $\mathcal{H}(G,K).$ We denote by $G^0$ the connected component of $G.$ Note that $G^0\neq G$ only in the case \Ogen, where we have $G^0=G\cap\mathrm{SL}(V).$ We note that if $K^0\defeq K\cap G^0,$ then $\mathcal{H}(G,K)\hookrightarrow\mathcal{H}(G^0,K^0)$ via restriction, and that this is an isomorphism except for the case \Ogen with $N$ even.

If $\dim_FV=N,$ the Langlands dual group of $G^0,$ which we denote ${}^LG,$ is
\begin{center}
\begin{tabular}{c|c|c}
    \Ugen&\Ogen&\SPgen\\
    \hline
    $\mathrm{GL}_N^{\mathrm{ext}}(\C)$&$\begin{cases}\mathrm{SO}(N,\C)&\text{if $N$ is even,}\\\mathrm{Sp}(N-1,\C)&\text{if $N$ is odd}\end{cases}$&$\mathrm{SO}(N+1,\C)$
\end{tabular}
\end{center}
Here we denote $\mathrm{GL}_N^{\mathrm{ext}}(\C)\defeq\mathrm{GL}_N(\C)\rtimes\{1,\sigma\}$ where
\begin{equation*}
    \sigma(g)=J\cdot (g^{\intercal})^{-1}\cdot J^{-1}\quad\text{for}\quad J\defeq\begin{psmallmatrix}&&&&1\\&&&-1&\\&&\iddots&&\\&(-1)^{N-2}&&&\\(-1)^{N-1}&&&&\end{psmallmatrix}.
\end{equation*}

For the cases \Ogen and \SPgen, denote $M=\begin{cases}
    2\lfloor N/2\rfloor&\text{in case \Ogen,}\\N+1&\text{in case \SPgen,}
\end{cases}$ so that ${}^LG$ is a classical group for a space of dimension $M.$ Note that $\lfloor N/2\rfloor=\lfloor M/2\rfloor.$ Let $T\subseteq{}^LG$ be the maximal torus, and we may assume, without loss of generality, that it consists of the diagonal matrices. Let $\mathbb{X}^*(T)$ denote the group of characters of $T,$ and for $1\le i\le M$ we denote $\mu_i\in\mathbb{X}^*(T)$ to be the projection to the $i$-th factor. Then $\mathbb{X}^*(T)$ is a free abelian group generated by $\mu_1,\ldots,\mu_{\lfloor N/2\rfloor},$ and we have $\mu_{M+1-i}=-\mu_i.$ For $\mu\in\mathbb{X}^*(T),$ denote $[\mu]\in R[\mathbb{X}^*(T)]$ the corresponding element. For $1\le i\le \lfloor N/2\rfloor,$ we consider $\boldsymbol{\mu}_i\defeq[\mu_i]+[-\mu_i].$ Then we have that
\begin{equation}\label{SatOSP}
    R[\boldsymbol{\mu}_1,\ldots,\boldsymbol{\mu}_{\lfloor N/2\rfloor}]^{\mathrm{sym}}\hookrightarrow R[\mathbb{X}^*(T)]^{W_T}
\end{equation}
with equality except in case \Ogen with $N$ even.

For the case \Ugen, we follow \cite[Section B.1]{LTXZZ}. For $T\subseteq\mathrm{GL}_N\subseteq\mathrm{GL}_N^{\mathrm{ext}}$ the maximal torus, its group of characters $\mathbb{X}^*(T)$ is a free abelian group generated by the projection to the $i$-th factor $\mu_i\in\mathbb{X}^*(T).$ For $1\le i\le \lfloor N/2\rfloor,$ denote $\boldsymbol{\mu}_i\defeq[\mu_i-\mu_{N+1-i}]+[\mu_{N+1-i}-\mu_i]\in R[\mathbb{X}^*(T)].$ Then as in \cite[Remark B.1.1]{LTXZZ}, we have
\begin{equation}\label{SatU}
    R[\mathbb{X}^*(T)^\sigma]^{W_T}=R[\boldsymbol{\mu}_1,\ldots,\boldsymbol{\mu}_{\lfloor N/2\rfloor}]^{\mathrm{sym}}.
\end{equation}
\begin{definition}\label{Satdef}
    Fix a choice of a square root of $q$ in $R,$ which we take to be $q_0$ in case \Ugen. We consider the Satake transform $\mathrm{Sat}\colon\mathcal{H}(G,K)\hookrightarrow\mathcal{H}(G^0,K^0)\to R[\mathbb{X}^*(T)]^{W_T},$ which induces $\mathrm{Sat}\colon\mathcal{H}(G,K)\to R[\boldsymbol{\mu}_1,\ldots,\boldsymbol{\mu}_{\lfloor N/2\rfloor}]^{\mathrm{sym}}$\index{Sat@$\mathrm{Sat}$} \index{1mui@$\boldsymbol{\mu}_i$} via \eqref{SatOSP} and \eqref{SatU}.
\end{definition}
One can give explicit formulas for the Satake transforms of the Hecke operators $T_k\in\mathcal{H}(G,K),$ as for example in \cite{Macdonald} for the case that $G$ is split. We will not directly use such formulas, and will instead rely on the following.
\begin{theorem}[Lusztig--Kato formulas]\label{LKform}
Denote $n\defeq\lfloor N/2\rfloor.$ For $0\le\delta\le n,$ we have the following identities
\begin{center}
    \begin{tabular}{c|c}
        \SPgen&$\displaystyle\sum_{\delta\le i\le n}\qbinom{i}{\lfloor\frac{i-\delta}{2}\rfloor}{q^2}\mathrm{Sat}(T_{n-i})=q^{\frac{1}{2}\left(n^2+n-(\delta^2+\delta)\right)}\sum_{\delta\le i\le n}\binom{i}{\lfloor\frac{i-\delta}{2}\rfloor}\cdot s_{n-i}(\boldsymbol{\mu})$\\
        \hline
        \begin{tabular}{c}\Ugen\\$N$ even\end{tabular}&$\displaystyle\sum_{\delta\le i\le n}\qbinom{2i}{i-\delta}{-q_0}\mathrm{Sat}(T_{n-i})=q_0^{n^2-\delta^2}\sum_{\substack{\delta\le i\le n\\ i\equiv\delta\mod 2}}\binom{i}{\frac{i-\delta}{2}}\cdot s_{n-i}(\boldsymbol{\mu})$\\
        \hline
        \begin{tabular}{c}\Ugen\\$N$ odd\end{tabular}&$\displaystyle\sum_{\delta\le i\le n}\qbinom{2i+1}{i-\delta}{-q_0}\mathrm{Sat}(T_{n-i})=q_0^{n^2+n-(\delta^2+\delta)}\sum_{\delta\le i\le n}\binom{i}{\lfloor\frac{i-\delta}{2}\rfloor}\cdot s_{n-i}(\boldsymbol{\mu})$\\
        \hline
        \begin{tabular}{c}\Ogen\\$N$ even\end{tabular}&$\displaystyle\sum_{\substack{\delta\le i\le n\\ i\equiv\delta\mod 2}}\qbinom{i}{\frac{i-\delta}{2}}{q^2}\frac{q^\delta+1}{q^i+1}\mathrm{Sat}(T_{n-i})=q^{\frac{1}{2}(n^2-n-(\delta^2-\delta))}\sum_{\substack{\delta\le i\le n\\ i\equiv\delta\mod 2}}\binom{i}{\frac{i-\delta}{2}}\cdot s_{n-i}(\boldsymbol{\mu})$\\
        \hline
        \begin{tabular}{c}\Ogen\\$N$ odd\end{tabular}&$\displaystyle\sum_{\substack{\delta\le i\le n\\ i\equiv\delta\mod 2}}\qbinom{i}{\frac{i-\delta}{2}}{q^2}\mathrm{Sat}(T_{n-i})=q^{\frac{1}{2}\left(n^2-\delta^2\right)}\sum_{\substack{\delta\le i\le n\\ i\equiv\delta\mod 2}}\binom{i}{\frac{i-\delta}{2}}\cdot s_{n-i}(\boldsymbol{\mu})$\\
    \end{tabular}
\end{center}
\end{theorem}
\begin{proof}
    The case \Ugen is proved in \cite[Lemma B.1.2]{LTXZZ}, so we focus on the cases \Ogen and \SPgen.

    We recall the Lusztig--Kato formula \cite[Theorem 1.5]{Kato}. Let $P^{++}$ be the set of dominant weights of ${}^LG$ and $\rho^\vee$ be half the sum of the positive coroots. For $\lambda\in P^{++},$ let $\chi_\lambda$ be the character of the irreducible representation of highest weight $\lambda.$ Then the Lusztig--Kato formula states that
    \begin{equation}\label{LKeq}
        \sum_{\substack{\mu\in P^{++}\\\mu\le \lambda}}q^{\langle\lambda-\mu,\rho^\vee\rangle}K_{\lambda,\mu}(q^{-1})\cdot\mathrm{Sat}(T_\mu)=q^{\langle\lambda,\rho^\vee\rangle}\chi_\lambda
    \end{equation}
    where $T_\mu$ are defined similarly as in \Cref{Tidef}, and $K_{\lambda,\mu}(t)$ are the generalized Kostka--Foulkes polynomials.
    
    Denote $k\defeq n-\delta.$ We will consider the Lusztig--Kato formula for the representation $\rho_k\defeq\bigwedge^k\mathrm{std}$ of ${}^LG.$ Denote by $\chi(\rho_k)$ the restriction of its character to $T.$ Then we have
    \begin{equation*}
        \chi(\rho_k)=\begin{cases}
            \displaystyle\sum_{0\le i\le k}\binom{\lfloor N/2\rfloor-k+i}{\lfloor i/2\rfloor}\cdot s_{k-i}(\boldsymbol{\mu})&\text{in case \SPgen,}\\
            \displaystyle\sum_{\substack{0\le i\le k\\i\equiv0\mod2}}\binom{\lfloor N/2\rfloor-k+i}{i/2}\cdot s_{k-i}(\boldsymbol{\mu})&\text{in case \Ogen}.
        \end{cases}
    \end{equation*}
    This is because $\chi(\rho_k)=s_k([\mu_1],\ldots,[\mu_M])$ and we have the formal identities
    \begin{equation*}
    \begin{split}
        s_k(\alpha_1,\ldots,\alpha_r,1,\alpha_r^{-1},\ldots\alpha_r^{-1})&=\sum_{0\le i\le k}\binom{r-k+i}{\lfloor i/2\rfloor}\cdot s_{k-i}(\alpha_1+\alpha_1^{-1},\ldots,\alpha_r+\alpha_r^{-1}),\\
        s_k(\alpha_1,\ldots,\alpha_r,\alpha_r^{-1},\ldots\alpha_r^{-1})&=\sum_{\substack{0\le i\le k\\i\equiv0\mod2}}\binom{r-k+i}{i/2}\cdot s_{k-i}(\alpha_1+\alpha_1^{-1},\ldots,\alpha_r+\alpha_r^{-1}).
    \end{split}
    \end{equation*}

    Note that by changing $i\mapsto i-\delta,$ we can rewrite this as
    \begin{equation*}
        \chi(\rho_k)=\begin{cases}
            \displaystyle\sum_{\delta\le i\le n}\binom{i}{\lfloor\frac{i-\delta}{2}\rfloor}\cdot s_{n-i}(\boldsymbol{\mu})&\text{in case \SPgen,}\\
            \displaystyle\sum_{\substack{\delta\le i\le n\\i\equiv\delta\mod2}}\binom{i}{\frac{i-\delta}{2}}\cdot s_{n-i}(\boldsymbol{\mu})&\text{in case \Ogen.}
        \end{cases}
    \end{equation*}
    Thus it remains to compute the left hand side of \eqref{LKeq}.

    In the case \Ogen with $N$ odd, ${}^LG$ has type $C_n.$ Note that $\chi(\rho_k)=\sum_{i=0}^{\lfloor k/2\rfloor}\chi_{\omega_{k-2i}}$ where $\omega_i$ are the fundamental weights\footnote{This is because the contraction map $\bigwedge^k\rho_{std}\to\bigwedge^{k-2}\rho_{std}$ given by the symplectic form is surjective with irreducible kernel of highest weight $\omega_k.$}. It remains to see that
    \begin{equation*}
        q^{\frac{1}{2}(i^2-\delta^2)}\sum_{j=0}^{\lfloor (i-\delta)/2\rfloor}K^{C_n}_{\omega_{n-\delta-2j},\omega_{n-i}}(q^{-1})\isequal \qbinom{i}{\frac{i-\delta}{2}}{q^2},
    \end{equation*}
    which is equivalent to
    \begin{equation*}
        \sum_{j=0}^{\lfloor (i-\delta)/2\rfloor}K^{C_n}_{\omega_{n-\delta-2j},\omega_{n-i}}(t)\isequal \qbinom{i}{\frac{i-\delta}{2}}{t^2}.
    \end{equation*}
    By the work of Lecouvey--Lenart \cite[Theorem 7.2, Remark 7.6]{Lecouvey-Lenart}, we have, denoting $\gamma_k=(\underbrace{1,\ldots,1}_k,0,\ldots,0,\underbrace{-1,\ldots,-1}_k),$ that
    \begin{equation*}
        K^{C_n}_{\omega_{n-a},\omega_{n-b}}(t)=K^{C_b}_{\omega_{b-a},0}(t)=\begin{cases}
            K^{A_{b-1}}_{\gamma_{(b-a)/2},0}(t^2)&\text{if $b-a$ is even,}\\
            0&\text{otherwise.}
        \end{cases}
    \end{equation*}
    Now the claimed formula follows from the well-known fact (see  \cite[Ch. III, §6, Ex. 2]{Macdonaldbook}) that
    \begin{equation*}
        K^{A_{i-1}}_{\gamma_m,0}(t)=\qbinom{i}{m}{t}-\qbinom{i}{m-1}{t}.
    \end{equation*}

    For the cases where ${}^LG$ is orthogonal of rank $M,$ $\rho_k$ is an irreducible representation except when $k=n$ and $M=2n,$ in which case it is a sum of two irreducible representations. Similarly as above\footnote{See \cite[\S 2.2 remark (iv)]{Lecouvey}, for example.}, it suffices to consider the case that $\mu=0.$ That is, we need to prove that
    \begin{equation*}
        q^{\frac{1}{2}(2nk-k^2+k)}K^{B_n}_{\omega_k,0}(q^{-1})\isequal\qbinom{n}{\lfloor \frac{k}{2}\rfloor}{q^2}\quad\text{for $0\le k\le n$}
    \end{equation*}
    and
    \begin{equation*}
        q^{\frac{1}{2}(2nk-k^2-k)}K^{D_n}_{\omega_k,0}(q^{-1})\isequal\frac{1}{q^n+1}\qbinom{n}{\frac{k}{2}}{q^2}\cdot\begin{cases}
            q^{n-k}+1&\text{if $k<n,$}\\1&\text{if $k=n$}
        \end{cases}\quad\text{for $0\le k\le n$}.
    \end{equation*}
    These are equivalent to
    \begin{equation*}
        K^{B_n}_{\omega_k,0}(t)\isequal\qbinom{n}{\lfloor k/2\rfloor}{t^2}\cdot\begin{cases}
            t^{k/2}&\text{if $k$ is even,}\\
            t^{n-\lfloor k/2\rfloor}&\text{if $k$ is odd}
        \end{cases}\quad\text{for $0\le k\le n$}
    \end{equation*}
    and
    \begin{equation*}
        K^{D_n}_{\omega_k,0}(t)\isequal\frac{1}{t^n+1}\qbinom{n}{\frac{k}{2}}{t^2}\cdot\begin{cases}
            t^{n-k/2}+t^{k/2}&\text{if $k<n,$}\\
            t^{n/2}&\text{if $k=n$}
        \end{cases}\quad\text{for $0\le k\le n$}.
    \end{equation*}
    We can verify these identities by using the combinatorial formula of Jang--Kwon \cite[Theorem 5.6]{Jang-Kwon} for orthogonal Kostka--Foulkes polynomials with second parameter $0.$ In our simpler case, $\mu$ in the notation of loc.\ cit.\ is a column partition of size $k.$ A semistandard Young tableaux of shape $\mu$ in the letters $\{1,\ldots,M\}$ is given by a subset $S\subseteq\{1,\ldots,M\}$ of size $k.$ For such a subset $S,$ the condition \cite[Equation (4.3)]{Jang-Kwon} is vacuous. Consider the sets
    \begin{equation*}
        \mathscr{S}(k,M)\defeq\left\{S\subseteq\{1,\ldots,M\}\colon \begin{array}{l}S\text{ has size }k\text{ and is of the form }S=(a_1,b_1]\sqcup\cdots\sqcup(a_l,b_l]\\\text{where }l\in\Z_{\ge1}\text{ and }a_i,b_i\text{ are even integers with}\\0\le a_1<b_1<a_2<b_2<\cdots<a_l<b_l\le M\end{array}\right\}
    \end{equation*}
    and
    \begin{equation*}
        \mathscr{S}'(k,M)\defeq\begin{cases}
            \mathscr{S}(k,M)\sqcup\mathscr{S}(M-k,M)&\text{if $k\neq M-k$,}\\
            \mathscr{S}(k,M)&\text{otherwise.}
        \end{cases}
    \end{equation*}
    Define, for $S\in\mathscr{S}(k,M),$ its \emph{weight} to be $w(S)\defeq\frac{1}{2}\sum_{i=1}^l(a_i+b_i).$ Then, using \cite[Lemma 5.3]{Jang-Kwon}, we can check that \cite[Theorem 5.6]{Jang-Kwon} translates to
    \begin{equation*}
        K^{B_n}_{\omega_k,0}(t)=\sum_{S\in\mathscr{S}'(k,2n+1)}t^{w(S)},\quad K^{D_n}_{\omega_k,0}(t)=\frac{1}{t^n+1}\sum_{S\in\mathscr{S}'(k,2n)}t^{w(S)}.
    \end{equation*}
    
    Denote
    \begin{equation*}
        w(k,N)=\sum_{S\in\mathscr{S}(k,N)}t^{w(S)}.
    \end{equation*}
    It remains to check that
    \begin{equation*}
        w(k,N)\isequal\begin{cases}
            t^{k/2}\qbinom{\lfloor N/2\rfloor}{k/2}{t^2}&\text{if $k$ is even,}\\
            0&\text{if $k$ is odd.}
        \end{cases}
    \end{equation*}
    Clearly $w(k,N)=w(k,2\lfloor N/2\rfloor).$ So consider without loss of generality the case $N=2n.$ Thinking about the last segment $(a_l,b_l]$ of $S\in\mathscr{S}(k,2n),$ we obtain the recursion
    \begin{equation*}
        w(k,2n)=w(k,2(n-1))+\sum_{i\ge1}t^{2n-i}w(k-2i,2(n-i-1)).
    \end{equation*}
    We may rewrite this as the following finite recurrence
    \begin{equation*}
        w(k,2n)=w(k,2(n-1))+t\cdot w(k-2,2(n-1))+t(t^{2(n-1)}-1)\cdot w(k-2,2(n-2)),
    \end{equation*}
    from which the desired formula easily follows by induction.
\end{proof}

%% file: 4-SphericalFunctions.tex
\section{Explicit model of spherical functions}\label{MainSection}
Let $r\ge1$ be a positive integer, and fix a coefficient ring $R\neq0.$ In the case \Ogen, we assume that $2\in R^\times.$

Let $V^\flat\defeq F^{2\cfactor r},$ that is,
\begin{case}{\Ugen}
$V^\flat=F^{2r},$
\end{case}
\begin{case}{\Ogen}
$V^\flat=F^{2r},$
\end{case}
\begin{case}{\SPgen}
$V^\flat=F^{4r}.$
\end{case}
We equip it with a pairing $\langle\cdot,\cdot\rangle^\flat\colon V^\flat\times V^\flat\to F$\index{Vflat@$V^\flat$} which is:
\begin{case}{\Ugen}
$F/F_0$-Hermitian,
\end{case}
\begin{case}{\Ogen}
symmetric,
\end{case}
\begin{case}{\SPgen}
alternating.
\end{case}
We assume that both $(V^\flat,\langle\cdot,\cdot\rangle^\flat)$ and $(V^\flat,\varpi\langle\cdot,\cdot\rangle^\flat)$ admit self-dual lattices. in the cases \Ugen and \SPgen this is automatic, and in the case \Ogen this nails $(V^\flat,\langle\cdot,\cdot\rangle)$ down to $\mathrm{sign}(\det(V^\flat,\langle\cdot,\cdot\rangle))=\epsilon^r.$ We denote $\Lat^\bullet(V^\flat)\defeq\Lat^\circ(V^\flat,\varpi\langle\cdot,\cdot\rangle^\flat).$\index{Latbullet@$\Lat^\bullet(V^\flat)$} That is, a lattice $\Lambda^\bullet\in\Lat^\bullet(V^\flat)$ satisfies that $\Lambda^{\bullet,\vee}=\varpi\Lambda^\bullet.$

Furthermore, in the cases \Ugen and \Ogen, we let $V^\sharp\defeq F^{2r+1}$\index{Vsharp@$V^\sharp$} and equip it with a pairing $\langle\cdot,\cdot\rangle^\sharp\colon V^\sharp\times V^\sharp\to F$ which is:
\begin{case}{\Ugen}
$F/F_0$-Hermitian,
\end{case}
\begin{case}{\Ogen}
symmetric.
\end{case}
We also assume that $V^\sharp$ admits self-dual lattices, and we denote $\psi\defeq\mathrm{sign}(\mathrm{disc}(V)).$\index{1psi@$\psi$} Note that $\binom{2r}{2}\equiv r\mod 2,$ and so $\psi=\epsilon^r\mathrm{sign}(\det(V)).$

We let $G^\flat$ resp.\ $G^\sharp$ denote the automorphism group of $(V^\flat,\langle\cdot,\cdot\rangle^\flat)$ resp.\ $(V^\sharp,\langle\cdot,\cdot\rangle^\sharp).$ We will consider orthogonal decompositions $V^\flat=V_1^\flat\oplus V_2^\flat$ resp.\ $V^\sharp=V_1^\sharp\oplus V_2^\sharp$ where $V_2^\flat$ resp.\ $V_2^\sharp$ has dimension $\cfactor r.$ For such decompositions, we denote $H_1^\flat,H_2^\flat$ resp.\ $H_1^\sharp,H_2^\sharp$ the corresponding automorphism groups. We consider
\index{Xflatsharp@$X^\flat,\ X^\sharp$}
\begin{equation*}
    X^\flat\defeq\bigsqcup_{V^\flat=V_1^\flat\oplus V_2^\flat}(H_1^\flat\times H_2^\flat)\backslash G^\flat,\quad X^\sharp\defeq\bigsqcup_{V^\sharp=V_1^\sharp\oplus V_2^\sharp}(H_1^\sharp\times H_2^\sharp)\backslash G^\sharp,
\end{equation*}
where the disjoint unions are over the isomorphism classes of $V_1^?,V_2^?.$

Now fix choices $\Lambda^{\circ,\flat}_{fix}\in\Lat^\circ(V^\flat),$ $\Lambda^{\bullet,\flat}_{fix}\in\Lat^\bullet(V^\flat)$ and  $\Lambda^{\sharp}_{fix}\in\Lat^\circ(V^\sharp).$ Their stabilizers determine, respectively, hyperspecial subgroups $K^{\flat},K^{\bullet,\flat}\subseteq G^\flat$ and $K^{\sharp}\subseteq G^\sharp.$ We consider
\index{S(X/K)@$\mathcal{S}(X^?/K^?)$}
\begin{equation*}
    \mathcal{S}(X^\flat/K^{\flat})\defeq\mathcal{S}(X^\flat)^{K^\flat},\quad\mathcal{S}(X^\flat/K^{\bullet,\flat})\defeq\mathcal{S}(X^\flat)^{K^{\bullet,\flat}},\quad\mathcal{S}(X^\sharp/K^\sharp)\defeq\mathcal{S}(X^\sharp)^{K^\sharp}
\end{equation*}
to be the spaces of ($R$-valued) locally constant, compactly supported spherical functions on $X^\flat$ and $X^\sharp.$ As in \Cref{Hdef}, we also consider the Hecke algebras $\mathcal{H}(G^\flat,K^{\flat}),\mathcal{H}(G^\flat,K^{\bullet,\flat}),\mathcal{H}(G^\sharp,K^\sharp)$\index{H(G,K)@$\mathcal{H}(G^?,K^?)$} which act on the above spaces via convolution.

\subsection{Combinatorial description of spherical functions}
\input{4.1-Setup}
\subsection{Straightening relations}
\input{4.2-StraighteningRelations}
\subsection{Generic computation}
\input{4.3-BigHecke}
\subsection{Straightening relations and \texorpdfstring{$T^{\circ\bullet}_r$}{T(circ,bullet)r}}
\input{4.4-SmallHeckeflat}

\subsection{Straightening relations and \texorpdfstring{$T_{\le1,r}^\sharp$}{T(<=1,r)sharp}}
\input{4.5-SmallHeckesharp}
\subsection{Extension argument}
\input{4.6-Extension}
\subsection{Compatibility with the Satake transform}
\input{4.7-Satake}

%% file: 4.1-Setup.tex
Since $G^\flat/K^\flat=\Lat^\circ(V^\flat),$ $G^\flat/K^{\bullet,\flat}=\Lat^\bullet(G^\flat)$ and $G^\sharp/K^\sharp=\Lat^\circ(V^\sharp),$ we have 
\begin{equation*}
    X^\flat/K^\flat=\bigsqcup_{V^\flat=V_1^\flat\oplus V_2^\flat}(H_1^\flat\times H_2^\flat)\backslash \Lat^\circ(V^\flat),\quad X^\flat/K^{\bullet,\flat}=\bigsqcup_{V^\flat=V_1^\flat\oplus V_2^\flat}(H_1^\flat\times H_2^\flat)\backslash \Lat^\bullet(V^\flat),
\end{equation*}
and
\begin{equation*}
    X^\sharp/K^\sharp=\bigsqcup_{V^\sharp=V_1^\sharp\oplus V_2^\sharp}(H_1^\sharp\times H_2^\sharp)\backslash \Lat^\circ(V^\sharp).
\end{equation*}

\begin{proposition}\index{typfs@$\mathrm{typ}^\flat,\ \mathrm{typ}^\sharp$}
    The maps $\mathrm{typ}^\natural\circ\mathrm{int}_2$ induce bijections
    \begin{equation*}
        \mathrm{typ}^{\flat}\colon X^\flat/K^{\flat}\rightiso\Typ^0_r,\quad \mathrm{typ}^{\bullet,\flat}\colon X^\flat/K^{\bullet,\flat}\rightiso\Typ^0_r,\quad\mathrm{typ}^\sharp\colon X^\sharp/K^\sharp\rightiso\Typ^0_r.
    \end{equation*}
    Note that in the second map, we are considering the space $(V^\flat,\varpi\langle\cdot,\cdot\rangle^\flat).$
\end{proposition}
\begin{proof}
    This follows from \Cref{typbij}, together with the corresponding statement \cite[\citeXoverKclassification]{CoratoZanarella}.
\end{proof}

For $0\le k\le \cfactor r,$ we have the Hecke operators
\index{Ticircbulletflatsharp@$T_{k,r}^{\flat},\ T_{k,r}^{\bullet,\flat},\ T_{k,r}^\sharp$}
\begin{equation*}
    T_k^{\flat}\in\mathcal{H}(G^\flat,K^\flat),\quad T_k^{\bullet,\flat}\in\mathcal{H}(G^\flat,K^{\bullet,\flat})\quad\text{and}\quad T_k^\sharp\in\mathcal{H}(G^\sharp,K^\sharp)
\end{equation*}
as defined in \Cref{Tidef}. We also consider the intertwining operators
\index{Tcircbullet@$T^{\bullet\circ}_r,\ T^{\circ\bullet}_r$}
\begin{equation*}
    T^{\bullet\circ}\colon\mathcal{S}(X^\flat/K^\flat)\to\mathcal{S}(X^\flat/K^{\bullet,\flat})\quad\text{and}\quad T^{\circ\bullet}\colon\mathcal{S}(X^\flat/K^{\bullet,\flat})\to\mathcal{S}(X^\flat/K^\flat)
\end{equation*}
given by convolution by $\mathrm{char}(K^\flat K^{\bullet,\flat})$ and $\mathrm{char}(K^{\bullet,\flat}K^\flat)$ respectively.

Putting it all together, we have the following ``combinatorial'' description of the structure of the $\mathcal{H}(G^?,K^?)$-modules $\mathcal{S}(X^?/K^?).$
\begin{corollary}
    The above maps induce identifications
    \begin{equation*}
        \mathrm{typ}^{\flat}\colon\mathcal{S}(X^\flat/K^{\flat})\rightiso R[\Typ_r^0],\quad\mathrm{typ}^{\bullet,\flat}\colon\mathcal{S}(X^\flat/K^{\bullet,\flat})\rightiso R[\Typ_r^0],\quad \mathrm{typ}^{\sharp}\colon\mathcal{S}(X^\sharp/L^\sharp)\rightiso R[\Typ_r^0].
    \end{equation*}
    The (adjoint of the) action of the Hecke algebras gets translated to the following \emph{counting problems}. For $?\in\{\flat,\sharp\},$ given $\delta_1,\delta_2\in\Typ_r^0,$ the coefficient of $\delta_2$ in $T_k^{?,*}(\delta_1)$ is given by: for $\Lambda_1\in\Lat^\circ(V^?)$ with $\mathrm{typ}^?(\Lambda_1)=\delta_1,$ the quantity
    \begin{equation*}
        \#\{\Lambda_2\in\Lat^\circ(V^?)\colon \mathrm{typ}^?(\Lambda_2)=\delta_2,\ \mathrm{relpos}(\Lambda_1,\Lambda_2)=(\underbrace{1,\ldots,1}_k,\underbrace{0,\ldots,0}_{N-2k},\underbrace{-1,\ldots,-1}_k)\}.
    \end{equation*}

    Furthermore, the (adjoint of the) action of the intertwining operators gets translated to the following \emph{counting problems}. Given $\delta^\circ,\delta^\bullet\in\Typ_r^0,$ the coefficient of $\delta^\circ$ in $T^{\circ\bullet,*}(\delta^\bullet)$ is given by: for $\Lambda^\bullet\in\Lat^\bullet(V^\flat)$ with $\mathrm{typ}^{\bullet,\flat}(\Lambda^\bullet)=\delta^\bullet,$ the quantity
    \begin{equation*}
        \#\{\Lambda^\circ\in\Lat^\circ(V^\flat)\colon \mathrm{typ}^{\flat}(\Lambda^\circ)=\delta^\circ,\ \varpi\Lambda^\bullet\subseteq\Lambda^\circ\subseteq\Lambda^\bullet\}.
    \end{equation*}
\end{corollary}

We will analyze these combinatorial problems for each isomorphism class of $V_2$ at a time. Given such $V_2,$ we may without loss of generality assume that we have $V^\flat=V_1^\flat\oplus V_2$ and $V^\sharp=V_1^\sharp\oplus V_2$ where
\begin{equation*}
    V_1^\sharp=Fu\oplus V_1^\flat\quad\text{for}\quad \langle u,u\rangle\in\O_F^\times,\ \mathrm{sign}(\langle u,u\rangle)=\psi.
\end{equation*}
Moreover, for any given $\delta\in\Typ_r^0,$ we may choose $\Lambda^\sharp\in\Lat^\circ(V^\sharp)$ satisfying $\mathrm{typ}^\sharp(\Lambda^\sharp)=\delta$ to be of the form $\Lambda^\sharp=\O_Fu\oplus\Lambda^\flat$ for $\Lambda^\flat\in\Lat^\circ(V^\flat).$

%% file: 4.2-StraighteningRelations.tex
We begin by recalling some notation and results from \cite[Section 4.2]{CoratoZanarella}.
\begin{definition}[{\cite[\citeTypdef]{CoratoZanarella}}]
    For each $r\ge0,$ we consider $\Typ_r\defeq(\Z\times\mathrm{Sign})^r,$ and we denote $\Typ\defeq\bigsqcup_{r\ge0}\Typ_r.$\index{Typ@$\Typ,\ \Typ_r$}\index{Typ@$\Typ,\ \Typ_r$! |seealso{\cite[\citeTypdef]{CoratoZanarella}}} For $a,b\ge0,$ we consider the concatenation product $\star\colon\Typ_a\times\Typ_b\to\Typ_{a+b}.$\index{0*@$\star$} This makes $R[\Typ]$ into a graded noncommutative $R$-ring with graded pieces $\mathrm{Gr}_r(R[\Typ])=R[\Typ_r].$
\end{definition}
For $e\in\Z^r,$ $\chi\in\mathrm{Sign}^r,$ we will denote the element associated to $(e,\chi)\in\Typ_r$ by
\index{0(e_i,chi_i)@$\Typee{e_i}{\chi_i}{e_{i+1}}{\chi_{i+1}}$}\index{1delta(e,chi)@$\delta(e,\chi)$}
\begin{equation*}
    \delta(e,\chi)=\Typee{e_i}{\chi_i}{e_{i+1}}{\chi_{i+1}}\in R[\Typ].
\end{equation*}
Moreover, for a character $s\colon\mathrm{Sign}^r\to\{\pm1\},$ we denote
\index{1deltas(e)@$\delta_s(e)$}
\begin{equation*}
    \delta_s(e)\defeq\sum_{\chi\in\mathrm{Sign}^r}s(\chi)\Typee{e_i}{\chi_i}{e_{i+1}}{\chi_{i+1}}\in R[\Typ].
\end{equation*}
If $(s_1,\ldots,s_r)\in(\widehat{\mathrm{Sign}})^r=\mathrm{Sign}^r$ are such that $s\colon\mathrm{Sign}^r\to\{\pm1\}$ is given by $s(\chi)=\prod_{i}s_i(\chi_i),$ then we also denote
\index{0(e_i,s_i)Sigma@$\Typee{e_i}{s_i}{e_{i+1}}{s_{i+1}}^\Sigma$}
\begin{equation*}
    \Typee{e_i}{s_i}{e_{i+1}}{s_{i+1}}^\Sigma=\delta_s(e).
\end{equation*}

We recall that \cite[\citeReldef]{CoratoZanarella} defines a two-sided ideal $\Rel^\natural\subseteq R[\Typ]$\index{Relnat@$\Rel^\natural,\ \Rel^\natural_r$}\index{Relnat@$\Rel^\natural,\ \Rel^\natural_r$! |seealso{\cite[\citeReldef]{CoratoZanarella}}} in such a way that if $\Rel^\natural_r\defeq\mathrm{Gr}_r(\Rel^\natural),$ then there is a natural map (see \cite[\citeBergmanProp]{CoratoZanarella})
\index{strnat@$\mathrm{str}^\natural$}\index{strnat@$\mathrm{str}^\natural$! |seealso{\cite[\citeBergmanProp]{CoratoZanarella}}}
\begin{equation*}
    \mathrm{str}^\natural\colon R[\Typ_r]/\Rel^\natural_r\rightiso R[\Typ^{0,\natural}_r].
\end{equation*}
The ``smallest'' elements in $\Rel^\natural$ are summarized in \Cref{reltable}, which also includes straightening relations for the cases \genf and \gens, which we now define.
\begin{definition}\index{Relflat@$\Rel^\flat,\ \Rel_r^\flat$}
    We consider the homogeneous left ideal $\Rel^\flat\subseteq R[\Typ]$ given by $\Rel^\flat=\Rel^\natural+\Rel_{1}^\flat$ where $\Rel_{1}^\flat$ is generated by the following degree $1$ elements. We also denote $\Rel^\flat_r\defeq\mathrm{Gr}_r(\Rel^\flat).$
    \begin{case}{\Uf}
    For $m>1$ an integer,
    \begin{equation*}
        \Rel^\flat(m)\defeq(-m)-(-m+2)-q_0^{m-1}((m)-(m-2))
    \end{equation*}
    as well as
    \begin{equation*}
        \Rel^\flat(1)\defeq(-1)-(1).
    \end{equation*}
    \end{case}
    \begin{case}{\Of}
    For $m\ge1$ an integer and $s\in\widehat{\mathrm{Sign}},$
    \begin{equation*}
        \Rel^\flat\TypeCC{m}{s}\defeq\TypeCC{-m}{s}^\Sigma-\TypeCC{m}{s}^\Sigma.
    \end{equation*}
    \end{case}
    \begin{case}{\SPf}
    For $m>1$ an integer,
    \begin{equation*}
        \Rel^\flat(m)\defeq(-m)-(q+1)(-m+1)+q(-m+2)+q^{3(m-1)}\left((m-2)-(q+1)(m-1)+q(m)\right)
    \end{equation*}
    and
    \begin{equation*}
        \Rel^\flat(1)\defeq(-1)-(q+1)(0)+q(1).
    \end{equation*}
    \end{case}
\end{definition}

\begin{definition}\index{Relsharp@$\Rel^\sharp,\ \Rel_r^\sharp$}
    We consider the homogeneous left ideal $\Rel^\sharp\subseteq R[\Typ]$ given by $\Rel^\sharp=\Rel^\natural+\Rel_{1}^\sharp$ where $\Rel_{1}^\sharp$ is generated by the following degree $1$ elements. We also denote $\Rel^\sharp_r\defeq\mathrm{Gr}_r(\Rel^\sharp).$
    \begin{case}{\Us}
    For $m>1$ an integer,
    \begin{equation*}
        \Rel^\sharp(m)\defeq (-m)-(-m+2)-(-q)^{m-1}((m)-(m-2))
    \end{equation*}
    as well as
    \begin{equation*}
        \Rel^\sharp(1)\defeq(-1)-(1).
    \end{equation*}
    \end{case}
    \begin{case}{\Os}
    For $m\ge1$ odd and $s\in\widehat{\mathrm{Sign}}$:
    \begin{equation*}
        \Rel^\sharp\TypeCC{m}{s}\defeq\TypeCC{-m}{s}^\Sigma-\TypeCC{-m+2}{s}^\Sigma-q^{(m-1)/2}\left(\TypeCC{m}{s}^\Sigma-\TypeCC{m-2}{s}^\Sigma\right)
    \end{equation*}
    For $m\ge2$ even and $s\in\widehat{\mathrm{Sign}}$:
    \begin{equation*}
    \begin{split}
        \Rel^\sharp\TypeCC{m}{s}\defeq\ &\TypeCC{-m}{s}^\Sigma-\TypeCC{-m+2}{s}^\Sigma-\psi\TypeCC{-m+2}{-s}^\Sigma+\psi\TypeCC{-m+4}{-s}^\Sigma\\
        &+q^{(m-2)/2}\left(\psi\TypeCC{m}{-s}^\Sigma-\psi\TypeCC{m-2}{-s}^\Sigma-\TypeCC{m-2}{s}^\Sigma+\TypeCC{m-4}{s}^\Sigma\right).
    \end{split}
    \end{equation*}
    \end{case}
\end{definition}
\begin{remark}
    in the case \Os, since we are assuming $2\in R^\times,$ we have $\TypeCC{-1}{\chi}\equiv\TypeCC{1}{\chi}\mod\Rel_1^\sharp,$ as well as
    \begin{equation*}
        \TypeCC{-2}{\chi}\equiv(1+\psi\chi)\TypeCC{0}{\chi}-\psi\chi\TypeCC{2}{\chi}\mod\Rel_1^\sharp.
    \end{equation*}
\end{remark}

\begin{definition}[{\cite[\citetdef]{CoratoZanarella}}]
    For $\varepsilon\in\Z^r,$ we consider the translation endomorphisms $t(\varepsilon)\colon R[\Typ_r]\to R[\Typ_r]$\index{t(epsilon)@$t(\varepsilon)$}\index{t(epsilon)@$t(\varepsilon)$! |seealso{\cite[\citetdef]{CoratoZanarella}}} given by $\delta(e,\chi)\mapsto\delta(e+\varepsilon,\chi).$
\end{definition}
\begin{proposition}\label{preservesfs}
    The following operators preserve $\Rel_1^?$:
    \begin{equation*}
        \begin{array}{c|c|c|c|c}
            \SPf&\Uf&\Us&\Of&\Os \\
        \hline
            t(-1)+q^3\cdot t(1) & t(-1)+q_0\cdot t(1)&t(-1)-q\cdot t(1)& t(-1)+t(1)&t(-2)+q\cdot t(2)
        \end{array}
    \end{equation*}
\end{proposition}
\begin{proof}
\begin{case}{\SPf}
We have
\begin{equation*}
\begin{split}
    (t(-1)+q^3\cdot t(1))&\Rel^\flat(1)=\Rel^\flat(2),\\
    (t(-1)+q^3\cdot t(1))&\Rel^\flat(2)=\Rel^\flat(3)+2q^3\Rel^\flat(1),\\
    (t(-1)+q^3\cdot t(1))&\Rel^\flat(m)=\Rel^\flat(m+1)+q^3\Rel^\flat(m-1)\quad\text{for }m>2.
\end{split}
\end{equation*}
\end{case}
\begin{case}{\Uf}
We have
\begin{equation*}
\begin{split}
    (t(-1)+q_0\cdot t(1))&\Rel^\flat(1)=\Rel^\flat(2)\\
    (t(-1)+q_0\cdot t(1))&\Rel^\flat(2)=\Rel^\flat(3)+2q_0\Rel^\flat(1)\\
    (t(-1)+q_0\cdot t(1))&\Rel^\flat(m)=\Rel^\flat(m+1)+q_0\Rel^\flat(m-1)\quad\text{for }m>2.
\end{split}
\end{equation*}
\end{case}
\begin{case}{\Us}
This follows formally from the case \Uf by changing $q_0\mapsto -q.$
\end{case}
\begin{case}{\Of}
We have
\begin{equation*}
\begin{split}
    (t(-1)+t(1))&\Rel^\flat\TypeCC{1}{s}=\Rel^\flat\TypeCC{2}{s},\\
    (t(-1)+t(1))&\Rel^\flat\TypeCC{m}{s}=\Rel^\flat\TypeCC{m-1}{s}+\Rel^\flat\TypeCC{m+1}{s}\quad\text{for }m>1.
\end{split}
\end{equation*}
\end{case}
\begin{case}{\Os}
We have
\begin{equation*}
\begin{split}
    (t(-2)+q\cdot t(2))&\Rel^\sharp\TypeCC{1}{s}=2\Rel^\sharp\TypeCC{3}{s},\\
    (t(-2)+q\cdot t(2))&\Rel^\sharp\TypeCC{2}{s}=2\Rel^\sharp\TypeCC{4}{s},\\
    (t(-2)+1\cdot t(2))&\Rel^\sharp\TypeCC{m}{s}=\Rel^\sharp\TypeCC{m+2}{s}+q\Rel^\sharp\TypeCC{m-2}{s}\quad\text{for }m>2.\qedhere
\end{split}
\end{equation*}
\end{case}
\end{proof}

\begin{proposition}\index{strfs@$\mathrm{str}^\flat,\ \mathrm{str}^\sharp$}
    We have canonical isomorphisms
    \begin{equation*}
        \mathrm{str}^\flat\colon R[\Typ_r]/\Rel^\flat_r\rightiso R[\Typ^0_r]\quad\text{and}\quad\mathrm{str}^\sharp\colon R[\Typ_r]/\Rel^\sharp_r\rightiso R[\Typ^0_r]
    \end{equation*}
    characterized uniquely by the fact that if $\delta(e,\chi)\in\Typ_r$ is with $e_1\ge\cdots\ge e_r\ge0,$ then $\mathrm{str}^\flat(\delta(e,\chi))=\mathrm{str}^\sharp(\delta(e,\chi))=(e^0,\chi^0),$ where $e^0(i)=\lambda_i(e)$ and $\chi(i)=\prod_{e_j=i}\chi_j.$
\end{proposition}
\begin{proof}
As in the proof of \cite[\citeBergmanProp]{CoratoZanarella}, we apply Bergman's diamond lemma \cite[Theorem 1.2]{Bergman}. We closely follow such proof. The only additional ambiguities that need to be checked are the ones that involve $\Rel_1^\flat$ and $\Rel_1^\sharp.$

We first recall some notation from \cite[\citeBergmanProp]{CoratoZanarella}. In particular, we will be using the notation $\Type{e_i}{s_i}^\Sigma$ as in \cite[Section 4.2]{CoratoZanarella}. Recall that we apply Bergman's diamond lemma with respect to the generating set $\TypeCC{a}{s}^\Sigma$ for $a\in\Z$ and $s\in\widehat{\mathrm{Sign}},$ with the partial order $\TypeCC{a}{s_1}\le\TypeCC{b}{s_2}$ if either $a<b$ or $(a,s_1)=(b,s_2).$ This is extended lexicographically to the monomials. We let
\begin{equation*}
    J_2\defeq\left\{\TypeC{a}{s_1}{b}{s_2}\colon \TypeCC{a}{s_1}\not\ge\TypeCC{b}{s_2}\right\}
\end{equation*}
be the index set of the generators $r\TypeC{a}{s_1}{b}{s_2}$ of $\Rel^\natural$ considered in \cite[\citeBergmanProp]{CoratoZanarella}.

Similarly, we let
\begin{equation*}
    J_1\defeq\left\{\TypeCC{a}{s}\colon a<0\right\}
\end{equation*}
be the index set of the generators $r^\flat\TypeCC{a}{s}$ of $\Rel_1^\flat$ and $r^\sharp\TypeCC{a}{s}$ of $\Rel_1^\sharp.$ These are simply $r^?\TypeCC{a}{s}=\Rel^?\TypeCC{a}{s},$ except in the cases $a\in\{-1,-2\}$ for \Os, where we have $r^?\TypeCC{a}{s}=\frac{1}{2}\Rel^?\TypeCC{a}{s}.$

As in \cite[\citeBergmanProp]{CoratoZanarella} for the $r\TypeC{a}{s_1}{b}{s_2}$, these are so that $r^?\TypeCC{a}{s}-\TypeCC{a}{s}^\Sigma$ only contain terms which are larger than $\TypeCC{a}{s}.$ We also extend the notation to $r^?\TypeCC{a}{s}=0$ if $a\ge0.$

Let
\begin{equation*}
\begin{split}
    I_2^?\TypeC{a}{s_1}{b}{s_2}\defeq R\left[r\TypeC{a'}{s_1'}{b'}{s_2'},\quad \TypeCC{a'}{s_1'}^\Sigma\star r^?\TypeCC{b'}{s_2'}\colon\quad\quad\TypeC{a'}{s_1'}{b'}{s_2'}>\TypeC{a}{s_1}{b}{s_2}\right]
\end{split}
\end{equation*}
and let $J_{1,2}\defeq J_2\cap\left\{\TypeC{a}{s_1}{b}{s_2}\colon \TypeCC{b}{s_2}\in J_1\right\}.$ For $\TypeC{a}{s_1}{b}{s_2}\in J_{1,2},$ we define
\begin{equation*}
    D^?\TypeC{a}{s_1}{b}{s_2}\defeq  r\TypeC{a}{s_1}{b}{s_2}-\TypeCC{a}{s_1}^\Sigma\star r^?\TypeCC{b}{s_2}.
\end{equation*}
Then we need to prove that for all $\TypeC{a}{s_1}{b}{s_2}\in J_{1,2}$ and $?\in\{\flat,\sharp\},$ we have
\begin{equation*}
    P^?\TypeC{a}{s_1}{b}{s_2}\colon\quad\left(D^?\TypeC{a}{s_1}{b}{s_2}\in I_2^?\TypeC{a}{s_1}{b}{s_2}\right).
\end{equation*}

Let $\beta=1$ in the cases \SPgen, \Ugen and $\beta=2$ in the case \Ogen. We consider the following operator $\varphi.$
\begin{equation*}
\begin{array}{c|c}
\SPf&t(-1,0)+q^2t(0,-1)+q^5t(0,1)+q^7t(1,0)\\
\hline
\Uf&t(-1,0)-q_0t(0,-1)-q_0^2t(0,1)+q_0^3t(1,0)\\
\hline
\Us&t(-1,0)-q_0t(0,-1)+q_0^3t(0,1)-q_0^4t(1,0)\\
\hline
\Of&t(-2,0)+qt(0,-2)+qt(0,2)+q^2t(2,0)\\
\hline
\Os&t(-2,0)+qt(0,-2)+q^2t(0,2)+q^3t(2,0)
\end{array}
\end{equation*}
This is built-in so that it preserves $\Rel_2^?,$ by \cite[\citepreserveshalf]{CoratoZanarella} and \Cref{preservesfs}.

Furthermore, we have
\begin{equation*}
    \varphi\left(r\TypeC{a}{s_1}{b}{s_2}\right)\equiv r\TypeC{a-\beta}{s_1}{b}{s_2}\mod I_2^?\TypeC{a-\beta}{s_1}{b}{s_2}\quad\text{for }\TypeC{a}{s_1}{b}{s_2}\in J_2,
\end{equation*}
as well as
\begin{equation*}
    \varphi\left(\TypeCC{a}{s_1}^\Sigma\star r^?\TypeCC{b}{s_2}\right)\equiv \TypeCC{a-\beta}{s_1}^\Sigma\star r^?\TypeCC{b}{s_2}\mod I_2^?\TypeC{a-\beta}{s_1}{b}{s_2}\quad\text{for }\TypeCC{b}{s_2}\in J_1.
\end{equation*}
In particular, $\varphi\left(I_2^?\TypeC{a}{s_1}{b}{s_2}\right)\subseteq I_2^?\TypeC{a-\beta}{s_1}{b}{s_2}.$ Then
\begin{equation*}
    \varphi\left(D^?\TypeC{a}{s_1}{b}{s_2}\right)\equiv D^?\TypeC{a-\beta}{s_1}{b}{s_2}\mod I_2^?\TypeC{a-\beta}{s_1}{b}{s_2}\quad\text{for }\TypeC{a}{s_1}{b}{s_2}\in J_{1,2}.
\end{equation*}
This proves
\begin{equation}\label{strred1}
    P^?\TypeC{a+\beta}{s_1}{b}{s_2}\implies P^?\TypeC{a}{s_1}{b}{s_2}\quad\text{for }\TypeC{a+\beta}{s_1}{b}{s_2}\in J_{1,2}.
\end{equation}

Now consider the following operator $\varphi'.$
\begin{equation*}
\begin{array}{c|c}
\SPf&t(-1,-1)+q^3t(-1,1)+q^7t(1,-1)+q^{10}t(1,1)\\
\hline
\Uf&t(-1,-1)+q_0t(-1,1)+q_0^3t(1,-1)+q_0^4t(1,1)\\
\hline
\Us&t(-1,-1)-qt(-1,1)-q^2t(1,-1)+q^3t(1,1)\\
\hline
\Of&t(-2,-2)+t(-2,2)+q^2t(2,-2)+q^2t(2,2)\\
\hline
\Os&t(-2,-2)+qt(-2,2)+q^3t(2,-2)+q^4t(2,2)
\end{array}
\end{equation*}
Similarly as $\varphi,$ we have
\begin{equation*}
    \varphi'\left(r\TypeC{a}{s_1}{b}{s_2}\right)\equiv r\TypeC{a-\beta}{s_1}{b-\beta}{s_2}\mod I_2^?\TypeC{a-\beta}{s_1}{b-\beta}{s_2}\quad\text{for }\TypeC{a}{s_1}{b}{s_2}\in J_2,
\end{equation*}
as well as
\begin{equation*}
    \varphi'\left(\TypeCC{a}{s_1}^\Sigma\star r^?\TypeCC{b}{s_2}\right)\equiv \TypeCC{a-\beta}{s_1}^\Sigma\star r^?\TypeCC{b-\beta}{s_2}\mod I_2^?\TypeC{a-\beta}{s_1}{b-\beta}{s_2}\quad\text{for }\TypeCC{b}{s_2}\in J_1.
\end{equation*}
This implies
\begin{equation*}
    \varphi'\left(D^?\TypeC{a}{s_1}{b}{s_2}\right)\equiv D^?\TypeC{a-\beta}{s_1}{b-\beta}{s_2}\mod I_2^?\TypeC{a-\beta}{s_1}{b-\beta}{s_2}\quad\text{for }\TypeC{a}{s_1}{b}{s_2}\in J_{1,2}.
\end{equation*}
This proves
\begin{equation}\label{strred2}
    P^?\TypeC{a+\beta}{s_1}{b+\beta}{s_2}\implies P^?\TypeC{a}{s_1}{b}{s_2}\quad\text{for }\TypeC{a+\beta}{s_1}{b+\beta}{s_2}\in J'.
\end{equation}

Now \eqref{strred1} and \eqref{strred2} reduces the claim to consider $P^?(-2,-1)$ in the cases \Ugen and \SPgen, and
\begin{equation*}
\begin{split}
    &P^?\TypeC{-1}{-s}{-1}{s},\quad P^?\TypeC{-2}{-s}{-2}{s},\quad P^?\TypeC{-2}{s_1}{-1}{s_2},\\
    &P^?\TypeC{-3}{s}{-1}{s},\quad P^?\TypeC{-3}{s_1}{-2}{s_2},\quad P^?\TypeC{-4}{s}{-2}{s}.
\end{split}
\end{equation*}
in the cases \Ogen. This is a finite computation which can be easily verified on a computer.\footnote{See \href{https://github.com/murilocorato/Straightening-relations}{https://github.com/murilocorato/Straightening-relations} for an implementation of this verification on Sage.}
\end{proof}

%% file: 4.3-BigHecke.tex
As mentioned in the introduction, computing the generic behavior of $T_i^{?,*}$ directly seems to be not feasible. Instead, we will consider the following linear combinations.
\begin{definition}\index{Tless@$T^\flat_{\le k,r},\ T^\sharp_{\le k,r}$}
We consider the following operators $T_{\le k,r}^\flat\in\mathcal{H}(G^\flat,K^\flat)$ and $T_{\le k,r}^\sharp\in\mathcal{H}(G^\sharp,K^\sharp)$ for $0\le k\le \cfactor r$
\begin{equation*}
    T_{\le k,r}^\flat=\sum_{i=0}^kS\TypeC{k-i}{}{2r-\frac{2}{\cfactor}i}{\epsilon^{r-i}}T_{i,r}^\flat,\quad T_{\le k,r}^\sharp=\sum_{i=0}^kS\TypeC{k-i}{}{2r+1-2i}{\epsilon^{r-i}\psi}T_{i,r}^\sharp.
\end{equation*}
\end{definition}
The goal of this subsection is to compute $T_{\le k,r}^{?,*}(\delta)$ for $\delta$ \emph{generic}. The answer to this question will be given in by the following operators.
\begin{definition}\index{t2epsilon@$t^2(\varepsilon)$}
For $\varepsilon\in \Z^{\cfactor r},$ we consider the following \emph{translation} $ t^2(\varepsilon)\colon R[\Typ_r]\to R[\Typ_r].$
\begin{equation*}
    t^2(\varepsilon)\Type{e_i}{\chi_i}\defeq\begin{cases}\Type{e_i+2\varepsilon_i}{\chi_i}&\text{in the cases \Ugen, \Ogen,}\\
    \type{e_i+\varepsilon_{2i-1}+\varepsilon_{2i}}&\text{in the case \SPgen}.\end{cases}
\end{equation*}
\end{definition}
\begin{definition}\index{1Deltafsless@$\Delta^\flat_{\le k,r},\ \Delta^\sharp_{\le k,r}$}
We define $\Delta^?_{\le k,r}\colon R[\Typ_r]\to R[\Typ_r]$ for $?\in\{\flat,\sharp\}$ by
\begin{equation*}
    \Delta^\flat_{\le k,r}\defeq\sum_{\varepsilon\in\{-1,0,1\}^{\cfactor r}}q^{\widetilde{\inv}(\varepsilon)}\frac{\Dfactor(\lambda_1(\varepsilon))\cdot \Dfactor(\cfactor r-\lambda_{-1}(\varepsilon))}{\Dfactor(\cfactor r-k)}W_0(\cfactor r-k,\lambda_0(\varepsilon))\cdot t^2(\varepsilon)
\end{equation*}
and
\begin{equation*}
    \Delta^\sharp_{\le k,r}\defeq\sum_{\varepsilon\in\{-1,0,1\}^{\cfactor r}}q^{\widetilde{\inv}(\varepsilon)}\frac{\Dfactor(\lambda_1(\varepsilon))\cdot\Dfactor(\cfactor r-\lambda_{-1}(\varepsilon))}{\Dfactor(\cfactor r-k)}q^{\frac{k+\Sigma(\varepsilon)}{2}}W_{1/2}(\cfactor r-k,\lambda_0(\varepsilon))\cdot t^2(\varepsilon).
\end{equation*}
\end{definition}

\begin{proposition}\label{preserve}
    For $?\in\{\flat,\sharp\},$ the operators $\Delta_{\le k,r}^?$ preserve $\Rel_r^?.$
\end{proposition}
\begin{proof}
That they preserve $\Rel_r^\natural$ follows easily from \cite[\citepreserveshalf]{CoratoZanarella}.

To see that they preserve $\Rel_1^?,$ we need to compare the coefficients of $(\varepsilon,-1)$ and $(\varepsilon,1)$ in the expression for $\Delta_{\le k,r}^?.$ Let $\beta=\alpha+\begin{cases}
    1/2 & \text{if }?=\sharp \\
    0 & \text{if }?=\flat.
\end{cases}$ Then the power of $q$ in the coefficient of $t^2(\varepsilon)$ in $\Delta_{\le k,r}^?$
\begin{equation*}
    w(\varepsilon)\defeq\widetilde{\inv}(\varepsilon)+\frac{\lambda_1(\varepsilon)^2+(\lambda_1(\varepsilon)+\lambda_0(\varepsilon))^2-(\cfactor r-k)^2}{2}+\beta(k+\lambda_1(\varepsilon)-\lambda_{-1}(\varepsilon)).
\end{equation*}
Now we have
\begin{equation*}
    w(\varepsilon,1)-w(\varepsilon,-1)=-2\lambda_1(\varepsilon)-\lambda_0(\varepsilon)+\frac{2\lambda_1(\varepsilon)+1+2\lambda_1(\varepsilon)+2\lambda_0(\varepsilon)+1}{2}+2\beta=2\beta+1.
\end{equation*}
and thus it remains to see that the following operators preserve $\Rel_1^?.$
\begin{equation*}
\begin{array}{c|c|c|c|c}
    \SPf&\Uf&\Us&\Of&\Os \\
\hline
    t(-1)+q^3\cdot t(1)&t(-2)+q\cdot t(2)&t(-2)+q^2\cdot t(2)& t(-2)+t(2)&t(-2)+q\cdot t(2)
\end{array}
\end{equation*}
This follows from \Cref{preservesfs}.
\end{proof}

Note that the definitions of $T_{\le k,r}^?$ are precisely so that the coefficient of $\delta_1$ in $T_{\le k,r}^{?,*}(\delta_2)$ can be given as follows: given $\Lambda_2^?\in\Lat^{\circ}(V^?)$ with $\mathrm{typ}^?(\Lambda_2^?)=\delta_2,$ such coefficient is the quantity
\begin{equation*}
    \sum_{\substack{L^?\in\Lat(V^?)\\\varpi L^{?,\vee}\subseteq L^?\sub{k}\Lambda_2^?}}\#\{\Lambda_1^?\in\Lat^\circ(V^?)\colon\mathrm{typ}^?(\Lambda_1^?)=\delta_1,\ L^?\subseteq\Lambda_1^?\}.
\end{equation*}
In other words, we may write
\begin{equation}\label{TlessSum}
    T_{\le k,r}^{?,*}(\mathrm{typ}^?(\Lambda_2^?))=\sum_{\substack{L^?\in\Lat(V^?)\\\varpi L^{?,\vee}\subseteq L^?\sub{k}\Lambda_2^?}}\sum_{\substack{\Lambda_1^?\in\Lat^\circ(V^?)\\ L^?\subseteq\Lambda_1^?}}\mathrm{typ}^?(\Lambda_1^?).
\end{equation}

We will eventually reduce all these computations to certain computations in $V_2.$ This will allow us to invoke the computations of \cite[Section 3]{CoratoZanarella}. For this reduction, we first reduce, in the cases \gens, the computations to $V^\flat.$
\begin{proposition}\label{TlessSums}
Assume we are given $\Lambda_2^\sharp=\O_Fu\oplus\Lambda_2^\flat$ and $k\ge1.$ Then we have
\begin{equation*}
\begin{split}
    T_{\le k,r}^{\sharp,*}(\mathrm{typ}^\sharp(\Lambda_2^\sharp))=\ &T_{\le k,r}^{\flat,*}(\mathrm{typ}^\flat(\Lambda_2^\flat))+q^{r-k+\alpha}(q^{r-k+1}-1)\cdot T_{\le k-1,r}^{\flat,*}(\mathrm{typ}^\flat(\Lambda_2^{\flat}))\\
    &+q^\alpha(q^\alpha+1)\sum_{\substack{L^\flat\in\Lat(V^\flat)\\\varpi L^{\flat,\vee}\subseteq L^\flat\sub{k-1}\Lambda_2^\flat}}\sum_{\substack{\Lambda^\flat\in\Lat^\circ(V^\flat)\\ L^\flat\subseteq\Lambda^\flat}}\sum_{\substack{\Lambda^\flat_-\in\Lat(V^\flat)\\ \mathrm{typ}(\Lambda^\flat/\Lambda^\flat_-)=(1,\epsilon\psi)\\\varpi L^{\flat,\vee}\subseteq\Lambda^\flat_-}}\mathrm{typ}^\natural(\mathrm{int}_2(\Lambda_-^\flat)).
\end{split}
\end{equation*}
\end{proposition}
\begin{proof}
Note that $\mathrm{typ}^\sharp(\Lambda_2^\sharp)=\mathrm{typ}^\flat(\Lambda_2^\flat).$ We have two cases according to whether $L^\sharp$ is of the form $\O_Fu\oplus L^\flat$ or not. Denote $L^\flat_-\defeq\mathrm{int}_\flat(L^\sharp)$ and $L^\flat_+\defeq\mathrm{proj}_\flat(L^\sharp).$

\begin{enumerate}[label=(\alph*),leftmargin=*]
    \item If $L^\sharp=\O_Fu\oplus L^\flat$: then we must also have $\Lambda_1^\sharp=\O_Fu\oplus \Lambda_1^\flat.$ In total, this gives a contribution of $T_{\le k,r}^\flat(\mathrm{typ}^\flat(\Lambda_2^\flat))$ to \eqref{TlessSum}.
    \item If $L_-^\flat\neq L_+^\flat$: note that we necessarily have $\mathrm{int}_{Fu}(L)=\varpi\O_Fu$ and $\mathrm{proj}_{Fu}(L)=\O_Fu.$ As in \Cref{LatBij}, we get also the data of $\varphi_L\colon L^\flat_+/L^\flat_-\rightiso \O_Fu/\varpi\O_Fu.$ In such terms, having $\varpi L^{\sharp,\vee}\subseteq L^\sharp$ is the same as having
    \begin{equation*}
        \begin{tikzcd}
            &L^\flat_-\arrow[r,dash,"1"]&L^\flat_+\arrow[dll,phantom,"\join",very near start]\arrow[r,dash,"k-1"]&\Lambda_2^\flat\\
            \varpi(L^\flat_+)^\vee\arrow[r,dash,swap,"(1{,}\epsilon\psi)"]\arrow[ur,dash]\arrow[urr,phantom,"\meet",very near start]&\varpi(L^\flat_-)^\vee\arrow[ur,dash]&&
        \end{tikzcd}
    \end{equation*}
    as well as having that $\varphi_L$ preserves $\langle\cdot,\cdot\rangle\mod\varpi\O_F.$
    \begin{enumerate}[label=(b.\arabic*),leftmargin=*]
        \item For the cases that $\Lambda_1^\sharp=\O_Fu\oplus\Lambda_1^\flat,$ the contribution to \eqref{TlessSum} is thus
        \begin{equation*}
        \begin{split}
            &T_{\le k-1,r}^{\flat,*}(\mathrm{typ}^\flat(\Lambda_2^{\flat}))\cdot R\TypeC{1}{\epsilon\psi}{2(r-k+1)}{\epsilon^{r-k+1}}\cdot\begin{cases}q_0+1&\text{in the case \Ugen}\\2&\text{in the case \Ogen}\end{cases}\\
            =\ &q^{r-k+\alpha}(q^{r-k+1}-1)\cdot T_{\le k-1,r}^{\flat,*}(\mathrm{typ}^\flat(\Lambda_2^{\flat}))
        \end{split}
        \end{equation*}
        \item For the remaining cases, denote $\Lambda^\flat_-\defeq\mathrm{int}_\flat(\Lambda_1^\flat)$ and $\Lambda^\flat_+\defeq\mathrm{proj}_\flat(\Lambda_1^\flat)$ and assume $\Lambda^\flat_-\neq \Lambda^\flat_+.$ Note that $\mathrm{typ}^\sharp(\Lambda_1^\sharp)=\mathrm{typ}^\natural(\mathrm{int}_2(\Lambda_-^\flat)).$
        
        Note that we necessarily must have $\mathrm{int}_{Fu}(\Lambda^\flat_1)=\varpi\O_Fu$ and $\mathrm{proj}_{Fu}(\Lambda^\flat_1)=\varpi^{-1}\O_Fu.$ Let $\varphi\colon \Lambda^\flat_+/\Lambda^\flat_-\rightiso\varpi^{-1}\O_Fu/\varpi\O_Fu$ be as in \Cref{LatcircBij} and recall this must respect $\langle\cdot,\cdot\rangle\mod\O_F.$ Now the condition that $L^\sharp\subseteq\Lambda_1^\sharp$ translate to having $L^\flat_\pm\subseteq\Lambda^\flat_\pm$ as well as the following commutative diagram
        \begin{equation*}
            \begin{tikzcd}
                L^\flat_+/L^\flat_-\arrow{d}\arrow{r}{\sim}[swap]{\varphi_L}&\O_Fu/\varpi\O_Fu\arrow[d,hook]\\
                \Lambda^\flat_+/\Lambda^\flat_-\arrow{r}{\sim}[swap]{\varphi}&\varpi^{-1}\O_Fu/\varpi\O_Fu
            \end{tikzcd}
        \end{equation*}
        Now denote $\Lambda^\flat$ to be the unique self-dual lattice between $\Lambda^\flat_-$ and $\Lambda^\flat_+$ (corresponding to $\O_Fu/\varpi\O_Fu$ under $\varphi$). Then we have
        \begin{equation*}
            \begin{tikzcd}
                &L^\flat_-\arrow[r,dash,"1"]\arrow[dr,dash]\arrow[drr,phantom,"\meet",very near start]&L^\flat_+\arrow[dr,dash,"k-1"]\\
                &&\Lambda^\flat_-\arrow[r,dash,swap,"(1{,}\epsilon\psi)"]&\Lambda^\flat\arrow[llu,phantom,"\join",very near start]
            \end{tikzcd}
        \end{equation*}
        All in all, given the data of $L^\flat_+,\Lambda^\flat,\Lambda_-^\flat$ with $\Lambda^\flat\in\Lat^\circ(V^\flat),$ $\varpi(\Lambda^\flat_-)^\vee\not\subseteq\Lambda^\flat_-$ and
        \begin{equation*}
            \begin{tikzcd}
                &&\Lambda_2^\flat\\
                \varpi(L^\flat_+)^\vee\arrow[r,dash]\arrow[dr,dash]&L^\flat_+\arrow[ur,dash,"k-1"]\arrow[dr,dash,"k-1"]&\\
                &\Lambda^\flat_-\arrow[r,dash,swap,"(1{,}\epsilon\psi)"]&\Lambda^\flat
            \end{tikzcd}
        \end{equation*}
        then the contribution to \eqref{TlessSum} is $\mathrm{typ}^\natural(\mathrm{int}_2(\Lambda^\flat_-))$ times the number of choices of $\varphi,$ which is $\begin{cases}
             q_0(q_0+1) & \text{in the case \Ugen}, \\
             2 & \text{in the case \Ogen},
        \end{cases}$ which we may write as $q^\alpha(q^\alpha+1).$\qedhere
    \end{enumerate}
\end{enumerate}
\end{proof}

The next step is to reduce all the computations in $V^\flat$ to computations in $V_2.$
\begin{lemma}\label{lemmaselfdual}
    Assume $L\in\Lat(V^\flat)$ satisfy i) $L\sub{2a}L^\vee,$ ii) $\mathrm{int}_2(L)\sub{a}\mathrm{int}_2(L^\vee),$ iii) $\mathrm{int}_2(L^\vee)\subseteq\mathrm{int}_2(L^\vee)^\vee.$ Then $I\defeq L+\mathrm{int}_2(L^\vee)$ satisfies $I^\vee=I.$
\end{lemma}
\begin{proof}
    Since $I/L\iso \mathrm{int}_2(L^\vee)/\mathrm{int}_2(L),$ we have $\mathrm{length}(I/L)=a$ by (ii). As $\mathrm{length}(L^\vee/L)=2a$ by (i), it remains to check that $I\subseteq I^\vee.$ As $L\subseteq L^\vee,$ this boils down to (iii).
\end{proof}

\begin{proposition}\label{gencomp1}
Let $\Lambda\in\Lat^\circ(V^\flat)$ and denote $L_0\defeq\mathrm{int}_2(\Lambda).$ We assume that it satisfies $L_0\subseteq \varpi L_0^{\vee}.$ Fix $L_-,L_+\in\Lat(V_2)$ such that $\varpi L_0\subseteq L_-\subseteq L_+\subseteq L_0.$ Then the number of $L\in\Lat(V^\flat)$ satisfying
i) $\varpi L^\vee\subseteq L\subseteq\Lambda,$ ii) $\mathrm{int}_2(L)=L_+,$ iii) $\mathrm{int}_2(\varpi L^\vee)=L_-$ is
\begin{equation*}
    \frac{\Dfactor(\dim(\mathrm{int}_2(\Lambda)/L_-))}{\Dfactor(\dim(L_+/L_-))}.
\end{equation*}
Moreover, in such cases we must have that $\dim(L_+/L_-)+\dim(\Lambda/L)=\cfactor r.$
\end{proposition}
\begin{proof}
    Denote $I\defeq\varpi\Lambda+\mathrm{int}_2(\Lambda).$ By \Cref{lemmaselfdual}, applied to $(V^\flat,\varpi\langle\cdot,\cdot\rangle$) and $a=\gamma\cdot r,$ we have $\varpi I^\vee=I.$ 

    Given $L$ satisfying i), we have that ii) is equivalent to $L\cap I=\varpi\Lambda+L_+,$ which is equivalent to $\varpi L^\vee+I=\varpi(\varpi \Lambda+L_+).$ Similarly, given i) we have that iii) is equivalent to $\varpi L^\vee\cap I=\varpi\Lambda+L_-.$

    Now denote $N=\varpi L^\vee.$ Then we are counting the number of $N$ satisfying $N\subseteq\varpi N^\vee$ as well as
    \begin{equation*}
        \begin{tikzcd}
            &N\arrow[drr,dash]&&\\
            \varpi\Lambda+L_-\arrow[ur,dash,"a"]\arrow[dr,dash,swap,"b"]\arrow[drr,dash]\arrow[rrr,phantom,"\meet",very near start]&&&\varpi(\varpi\Lambda+L_+)^\vee\arrow[lll,phantom,"\join",very near start]\\
            &\varpi\Lambda+L_+\arrow[r,dash,swap,"a"]&I\arrow[ur,dash,"a"]&
        \end{tikzcd}
    \end{equation*}
    The second claim follows as $\dim(\Lambda/L)=\dim(N/\varpi\Lambda)=\dim(I/\varpi\Lambda)-b=\cfactor r-\dim(L_+/L_-).$
    
    Now note that $(\varpi\Lambda+L_-)+N$ is a Lagrangian in $\varpi(\varpi\Lambda+L_+)^\vee/(\varpi\Lambda+L_+)$ which is disjoint from $I.$ By \Cref{Lquantitysubsub}, there are $\Dfactor(a)$ possible such $(\varpi\Lambda+L_-)+N.$ Given that, there are $q^{ab}$ possible $N.$ Thus the final result is
    \begin{equation*}
        q^{ab}\Dfactor(a)=\frac{\Dfactor(a+b)}{\Dfactor(b)}.\qedhere
    \end{equation*}
\end{proof}

\begin{proposition}\label{gencomp2}
Assume $L\in\Lat(V^\flat)$ is such that $\varpi L^\vee\subseteq L\sub{2k} L^\vee.$ Denote $L_-\defeq\mathrm{int}_2(\varpi L^\vee)$ and $L_+\defeq\mathrm{int}_2(L).$ Assume further that $L_-\subseteq\varpi^2 L_-^\vee$ and that $\dim(\varpi^{-1}L_-/L_+)=k.$ Then given $L_+\subseteq L_0\subseteq \varpi^{-1}L_-,$ the number of $\Lambda\in\Lat^\circ(V^\flat)$ satisfying i) $L\subseteq \Lambda,$ ii) $\mathrm{int}_2(\Lambda)=L_0$ is $\Dfactor(\dim(\varpi^{-1}L_-/L_0)).$
\end{proposition}
\begin{proof}
Let $I\defeq L+\mathrm{int}_2(L^\vee).$ By \Cref{lemmaselfdual}, we have $I=I^\vee.$

Now we are simply counting $\Lambda\in\Lat^\circ(V^\flat)$ such that $\Lambda\cap I=L+L_0.$ As $L+L_0\subseteq I,$ we are counting Lagrangians in $(L+L_0)^\vee/(L+L_0)$ not intersecting a given Lagrangian. Hence the answer is simply $\Dfactor(\dim(I/(L+L_0)))=\Dfactor(\dim(\varpi^{-1}L_-/L_0)).$
\end{proof}

The following proposition will only be used in the cases \gens.
\begin{proposition}\label{gencomp3}
    Assume we have $\Lambda\in\Lat^\circ(V^\flat)$ and $L\in\Lat(V^\flat)$ with $\varpi L^\vee\subseteq L\subseteq\Lambda.$ Denote $L_0=\mathrm{int}_2(\Lambda)$ and $L_-=\mathrm{int}_2(\varpi L^\vee),$ $L_+=\mathrm{int}_2(L).$ Assume furthermore that $L_+\subseteq\varpi L_+^\vee.$ Then given $L_1\in\Lat(V_2)$ with $L_-\subseteq L_1\sub{1}L_0$ but $L_+\not\subseteq L_1,$ the number of $\Lambda_1\in\Lat(V^\flat)$ satisfying i) $\varpi L^\vee\subseteq\Lambda_1\subseteq\Lambda,$ ii) $\mathrm{typ}(\Lambda/\Lambda_1)=(1,\epsilon\psi),$ iii) $\mathrm{int}_2(\Lambda_1)=L_1$ is
    \begin{equation*}
        q^{\cfactor r-\dim(L_0/L_+)-1+\alpha}(q-1)/(q^\alpha+1).
    \end{equation*}
\end{proposition}
\begin{proof}
Let $I\defeq \varpi L^\vee+L_+.$ By \Cref{lemmaselfdual}, we have $I=\varpi I^\vee.$

Denote $N\defeq\varpi L^\vee+L_1,$ and $M\defeq\varpi L^\vee+L_0.$ Then $N\sub{1}M$ and $I\subseteq M,$ but $I\not\subseteq N.$ Note that $N\subseteq N+\varpi M^\vee\subseteq M,$ and since $N\sub{1}M$ we must have either $\varpi M^\vee\subseteq N$ or $N+\varpi M^\vee=M.$ In the first case, we would have
\begin{equation*}
\begin{split}
L_+\cap \varpi L_0^\vee&=\mathrm{int}_2(L)\cap \varpi L_0^\vee=(\mathrm{proj}_2(L^\vee)+\varpi^{-1}L_0)^\vee\\&=\varpi\mathrm{proj}_2(M)^\vee=\mathrm{int}_2(\varpi M^\vee)\subseteq \mathrm{int}_2(N)=L_1.
\end{split}
\end{equation*}
But by assumption $L_+\subseteq L_0\subseteq\varpi L_0^\vee,$ and thus the above would contradict our assumption that $L_+\not\subseteq L_1.$ Thus, $N+\varpi M^\vee=M.$ In particular, $N\cap \varpi M^\vee\sub{1}\varpi M^\vee,$ and dualy $M\sub{1}M+\varpi N^\vee.$ Thus we obtain
\begin{equation*}
    N+\varpi N^\vee=N+(\varpi M^\vee+\varpi N^\vee)=(N+\varpi M^\vee)+\varpi N^\vee=M+\varpi N^\vee
\end{equation*}
and $N\sub{1}M\sub{1}N+\varpi N^\vee.$

We are counting $N\subseteq\Lambda_1\subseteq\Lambda$ with $M\cap\Lambda_1=N$ and $\mathrm{typ}(\Lambda/\Lambda_1)=(1,\epsilon\psi).$ Dually, this means counting a line of norm $\epsilon\psi$ in $\varpi N^\vee/\varpi L^\vee$ which is not in $\varpi M^\vee.$ Note that such a line cannot be in $M^\vee,$ as $M^\vee$ is isotropic. Recall that we have just seen that
\begin{equation*}
    \varpi L^\vee\subseteq N\cap\varpi N^\vee\sub{2}\varpi N^\vee.
\end{equation*}
Thus the number of such lines is
\begin{equation*}
    q^{\dim(N\cap\varpi N^\vee/\varpi L^\vee)}R\TypeC{1}{\epsilon\psi}{2}{\epsilon}=q^{\dim(N\cap\varpi N^\vee/\varpi L^\vee)}=q^\alpha(q-1)/(q^\alpha+1)
\end{equation*}
and note that
\begin{equation*}
    \dim(N\cap\varpi N^\vee/\varpi L^\vee)=\dim(\varpi M^\vee/\varpi L^\vee)-1=\cfactor r-\dim(L_0/L_+)-1.\qedhere
\end{equation*}
\end{proof}

\begin{proposition}\label{genMainProp}
    Assume we are given $\delta=(e^0,\chi^0)\in\Typ^0_r$ which is supported on $\Z_{\ge2}.$ Fix $L\in\Lat(V_2)$ such that $\mathrm{typ}^\natural(L)=\delta.$ Denote
    \begin{equation*}
        A_L^{(a)}\defeq\left\{(L_-,L_+,L_0)\in\Lat(V_2)^3\colon\begin{tikzcd}
            \varpi L\arrow[dash,dr]&&&L\\[-20pt]
            &L_-\arrow[r,dash,"\cfactor r-a"]&L_+\arrow[ur,dash]\arrow[dr,dash]&\\[-20pt]
            \varpi L_0\arrow[ur,dash]&&&L_0
        \end{tikzcd}\right\}.
    \end{equation*}
    Then
    \begin{equation*}
    \begin{split}
        T_{\le k,r}^{\flat,*}(\delta)=\ &\sum_{(L_-,L_+,L_0)\in A_L^{(k)}}\frac{\Dfactor(\dim(L/L_-))}{\Dfactor(\cfactor r-k)}\Dfactor(k-\dim(L_0/L_+))\cdot\mathrm{typ}^\natural(L_0),\\
        T_{\le k,r}^{\sharp,*}(\delta)=\ &T_{\le k,r}^{\flat,*}(\delta)+q^{r-k+\alpha}(q^{r-k+1}-1)T_{\le k-1,r}^{\flat,*}(\delta)\\&+(q-1)\sum_{(L_-,L_+,L_0)\in A_L^{(k-1)}}\frac{\Dfactor(\dim(L/L_-))}{\Dfactor(\cfactor r-k)}\Dfactor(k-\dim(L_0/L_+))\sum_{\substack{L_-\subseteq L_1\sub{1}L_0\\L_+\not\subseteq L_1}}\mathrm{typ}^\natural(L_0).
    \end{split}
    \end{equation*}
\end{proposition}
\begin{proof}
    First of all, given $\Lambda_2^\sharp=\O_Fu\oplus\Lambda_2^\flat$ with $\mathrm{typ}^\flat(\Lambda_2^\flat)=\delta,$ we have $\mathrm{int}_2(\Lambda_2^\flat)\subseteq\varpi^2\mathrm{int}_2(\Lambda_2^\flat)^\vee.$ Denote
    \begin{equation*}
    \begin{split}
        S^\flat&\defeq\sum_{\substack{L^\flat\in\Lat(V^\flat)\\\varpi L^{\flat,\vee}\subseteq L^\flat\sub{k}\Lambda_2^\flat}}\sum_{\substack{\Lambda_1^\flat\in\Lat^\circ(V^\flat)\\ L^\flat\subseteq\Lambda_1^\flat}}\mathrm{typ}^\flat(\Lambda_1^\flat)\\
        S^\sharp&\defeq q^\alpha(q^\alpha+1)\sum_{\substack{L^\flat\in\Lat(V^\flat)\\\varpi L^{\flat,\vee}\subseteq L^\flat\sub{k}\Lambda_2^\flat}}\sum_{\substack{\Lambda^\flat\in\Lat^\circ(V^\flat)\\ L^\flat\subseteq\Lambda^\flat}}\sum_{\substack{\Lambda^\flat_-\in\Lat(V^\flat)\\ \mathrm{typ}(\Lambda^\flat/\Lambda^\flat_-)=(1,\epsilon\psi)\\\varpi L^{\flat,\vee}\subseteq\Lambda^\flat_-}}\mathrm{typ}^\natural(\mathrm{int}_2(\Lambda_-^\flat)).
    \end{split}
    \end{equation*}
    Of course, by \eqref{TlessSum} and \Cref{TlessSums}, we have $T_{\le k}^{\flat,*}(\delta)=S^\flat$ and $T_{\le k+1}^{\sharp,*}(\delta)=T_{\le k+1}^{\flat,*}(\delta)+q^{r-k-1+\alpha}(q^{r-k}-1)T_{\le k}^{\flat,*}(\delta)+S^\sharp.$

    First, for $S^\sharp,$ if we are given $L^\flat,$ $\Lambda^\flat$ and  $\Lambda^\flat_-,$ we note that $\mathrm{int}_2(L^\flat)\not\subseteq\mathrm{int}_2(\Lambda^\flat_-)\sub{1}\mathrm{int}_2(\Lambda^\flat).$ Indeed, i) if $\mathrm{int}_2(L)\subseteq\mathrm{int}_2(\Lambda^\flat_-),$ then $I\defeq\varpi L^{\flat,\vee}+\mathrm{int}_2(L^\flat)\subseteq\Lambda^\flat_-,$ but by \Cref{lemmaselfdual} we have $I=\varpi I^\vee,$ and this contradicts that $\varpi\Lambda_-^{\flat,\vee}\not\subseteq\Lambda_-^\flat,$ ii) as $\mathrm{int}_2(L^\flat)\not\subseteq\mathrm{int}_2(\Lambda^\flat_-),$ we certainly cannot have $\mathrm{int}_2(\Lambda_-^\flat)=\mathrm{int}_2(\Lambda^\flat),$ and thus since $\Lambda^\flat_--\sub{1}\Lambda^\flat,$ we conclude $\mathrm{int}_2(\Lambda_-^\flat)\sub{1}\mathrm{int}_2(\Lambda^\flat).$ Hence by \Cref{gencomp3} we conclude
    \begin{equation*}
        S^\sharp=(q-1)\sum_{\substack{L^\flat\in\Lat(V^\flat)\\\varpi L^{\flat,\vee}\subseteq L^\flat\sub{k}\Lambda_2^\flat}}\sum_{\substack{\Lambda^\flat\in\Lat^\circ(V^\flat)\\ L^\flat\subseteq\Lambda^\flat}}q^{r-\dim(\mathrm{int}_2(\Lambda^\flat)/\mathrm{int}_2(L^\flat))-1+2\alpha}\sum_{\substack{L_1\in\Lat(V_2)\\\mathrm{int}_2(\varpi L^{\flat,\vee})\subseteq L_1\sub{1}\mathrm{int}_2(\Lambda^\flat)\\\mathrm{int}_2(L)\not\subseteq L_1}}\mathrm{typ}^\natural(L_1).
    \end{equation*}

    Note denote $c^\flat(L_-,L_+,L_0)=1$ and
    \begin{equation*}
        c^\sharp(L_-,L_+,L_0)\defeq (q-1)q^{r-\dim(L_0/L_+)-1+2\alpha}\sum_{\substack{L_1\in\Lat(V_2)\\L_-\subseteq L_1\sub{1}L_0\\L_+\not\subseteq L_1}}\mathrm{typ}^\natural(L_1),
    \end{equation*}
    so that we may uniformly write
    \begin{equation*}
        S^?=\sum_{\substack{L^\flat\in\Lat(V^\flat)\\\varpi L^{\flat,\vee}\subseteq L^\flat\sub{k}\Lambda_2^\flat}}\sum_{\substack{\Lambda^\flat\in\Lat^\circ(V^\flat)\\ L^\flat\subseteq\Lambda^\flat}}c^?(\mathrm{int}_2(\varpi L^{\flat,\vee}),\mathrm{int}_2(L^\flat),\mathrm{int}_2(\Lambda^\flat)).
    \end{equation*}
    Now given $L^\flat$ as above, denote $L_-\defeq\mathrm{int}_2(\varpi L^{\flat,\vee})$ and $L_+\defeq\mathrm{int}_2(L^\flat).$ Note that by \Cref{gencomp1} we have $L_-\sub{\cfactor r-k}L_+.$ Thus by \Cref{gencomp2} we have
    \begin{equation*}
        S^?=\sum_{\substack{L^\flat\in\Lat(V^\flat)\\\varpi L^{\flat,\vee}\subseteq L^\flat\sub{k}\Lambda_2^\flat}}\sum_{\substack{L_0\in\Lat(V_2)\\L_+\subseteq L_0\subseteq\varpi^{-1}L_-}}\Dfactor(\dim(\varpi^{-1}L_-/L_0))\cdot c^?(L_-,L_+,L_0).
    \end{equation*}
    Finally, by \Cref{gencomp1}, we conclude that $S^?$ is
    \begin{equation*}
        \sum_{\substack{L_-,L_+\in\Lat(V_2)\\\varpi\mathrm{int}_2(\Lambda_2^\flat)\subseteq L_-\sub{\cfactor r-k}L_+\subseteq\mathrm{int}_2(\Lambda_2^\flat)}}\sum_{\substack{L_0\in\Lat(V_2)\\L_+\subseteq L_0\subseteq\varpi^{-1}L_-}}\frac{\Dfactor(k-\dim(L_0/L_+))\Dfactor(\dim(L_0/L_-))}{\Dfactor(\cfactor r-k)}\cdot c^?(L_-,L_+,L_0).
    \end{equation*}
    Now the claim follows immediately for case \genf. For case \gens, note that
    \begin{equation*}
        q^{r-\dim(L_0/L_+)-1+2\alpha}\frac{\Dfactor(k-\dim(L_0/L_+))}{\Dfactor(\cfactor r-k)}=\frac{\Dfactor((k+1)-\dim(L_0/L_+))}{\Dfactor(\cfactor r-(k+1))}.\qedhere
    \end{equation*}
\end{proof}

Finally, the last step is to do the above computation in $V_2.$ We will repeatedly use the following generic case of \cite[\citemainLemma]{CoratoZanarella}.
\begin{lemma}\label{genMainLemma}
    Let $L\in\Lat(V_2)$ and $V_2=\bigoplus_{i\in\Z}V_2^{(i)}$ be an orthogonal decomposition such that if $L^{(i)}\defeq L\cap V^{(i)},$ then $L=\bigoplus_{i\in\Z}L^{(i)}$ and $L^{(i),\vee}=\varpi^{-i}L^{(i)}.$ Assume that $\dim V_2^{(i)}\in\{0,\cfactor\}$ for all $i,$ and assume also that for every $i,$ at least one of $V_2^{(i)}$ or $V_2^{(i+1)}$ is $0.$ We consider the flag $L^{(\ge i)}/\varpi L^{(\ge i)}$ of $L/\varpi L.$ In the case \SPgen, we further (arbitrarily) choose a line on each nonzero $L^{(i)}/\varpi L^{(i)}.$ This gives us a full flag
    \begin{equation*}
        0=F_0\subset F_1\subsetneq\cdots\subsetneq  F_{\gamma r}=L/\varpi L.
    \end{equation*}
    Given $L'\in\Lat(V_2)$ with $\varpi L\subseteq L'\subseteq L,$ we denote $\varepsilon\in\{0,1\}^{\cfactor r}$ to be the incidence of $L'$ in the above flag, that is, $\varepsilon_i\defeq\dim((L'\cap F_i)/(L'\cap F_{i-1})).$ Then $\mathrm{typ}^\natural(L')= t^2(\varepsilon)(\mathrm{typ}^\natural(L)).$
\end{lemma}

\begin{theorem}\label{generic}
    Let $\delta=\Type{e_i}{\chi_i}\in\Typ^0_r$ be \emph{generic}, in the sense that $e_{i+1}\ge e_i+4$ for all $1\le i< r$ and $e_r\ge4.$ Then we have
    \begin{equation*}
        T_{\le k,r}^{\flat,*}(\delta)=\mathrm{str}^\flat(\Delta_{\le k,r}^\flat(\delta))\quad\text{and}\quad T_{\le k,r}^{\sharp,*}(\delta)=\mathrm{str}^\sharp(\Delta_{\le k,r}^\sharp(\delta)).
    \end{equation*}
\end{theorem}
\begin{proof}
Let $L\in\Lat(V_2)$ be any lattice with $\mathrm{typ}^\natural(L)=\delta.$ For this entire proof, we fix a decomposition of $V_2$ as in \Cref{genMainLemma} with respect to $L,$ and, in the case \SPgen, we arbitrarily extend the $V_2^{(\ge i)}$ to a full flag of $V_2.$ We will consider other lattices $L_-,L_+,L_0,L_1$ throughout the proof, and we will repeatedly apply \Cref{genMainLemma} for them as well, with the same choice of flag. This will be possible due to our assumption on the genericity of $\delta.$

The starting point is \Cref{genMainProp}. We first parametrize $A_L^{(k)}.$ Let $\varepsilon_-,\varepsilon_+\in\{0,1\}^{\cfactor r}$ be the incidence datas of $L_-,L_+$ with respect to $L.$ Let $\varepsilon_0$ be the incidence data of $L_0$ with respect to $\varpi^{-1}L_-.$ in the case \gens, denote also $\varepsilon_1$ to be the incidence data of $L_1$ with respect to $\varpi^{-1}L_-.$ Pictorially:
\begin{equation*}
\begin{tikzcd}
    &&L&\\[-10pt]
    L_-\arrow[r,dash]\arrow[dr,dash]\arrow[urr,dash,bend left=25,"\varepsilon_-"]&L_+\arrow[ur,dash,"\varepsilon_+"]\arrow[dr,dash]&&\\[-10pt]
    &L_1\arrow[r,dash]\arrow[rr,dash,bend right=30,swap,"\varepsilon_1"]&L_0\arrow[r,dash,"\varepsilon_0"]&\varpi^{-1}L_-
\end{tikzcd}
\end{equation*}
Then by multiple applications of \Cref{genMainLemma}, we have
\begin{equation*}
\begin{split}
    \mathrm{typ}^\natural(L_0)&= t^2(\varepsilon_-+\varepsilon_0-1)\mathrm{typ}^\natural(L)\\
    \mathrm{typ}^\natural(L_1)&= t^2(\varepsilon_-+\varepsilon_1-1)\mathrm{typ}^\natural(L)
\end{split}
\end{equation*}
Denoting $\varepsilon_{dif}\defeq \varepsilon_-+\varepsilon_0-1,$ we partition $\{1,\ldots,\cfactor r\}=I_-\sqcup E_0\sqcup E\sqcup E_1\sqcup I_+$ so that the $\varepsilon_?$ are given by
\begin{equation*}
    \begin{array}{c|c|c|c|c|c}
        &I_-&E_1&E&E_0&I_+\\
        \hline
        \varepsilon_{dif}&-1&0&0&0&1\\
        \hline
        \varepsilon_0&0&1&0&0&1\\
        \hline
        \varepsilon_-&0&0&1&1&1\\
        \hline
        \varepsilon_+&0&0&0&1&1
    \end{array}
\end{equation*}
We also denote $I_0\defeq E_1\sqcup E\sqcup E_0.$ Note that $\lvert E\rvert=\dim(L_+/L_-)=\cfactor r-k.$ In the case \gens, we denote $\beta\in I_-\sqcup E_0\sqcup E$ to be the (unique) index of $\{1,\ldots,\gamma r\}$ for which $(\varepsilon_1)_\beta\neq(\varepsilon_0)_\beta.$ We will also abuse notation to denote $\varepsilon_{dif}+\beta\defeq \varepsilon_-+\varepsilon_1-1.$

Now given the partition $\{1,\ldots,\cfactor r\}=I_-\sqcup E_0\sqcup E\sqcup E_1\sqcup I_+$ with $\lvert E\rvert=\cfactor r-k,$ the number of $(L_-,L_+,L_0)\in A_L^{(k)}$ with such data is
\begin{equation*}
    q^{\inv(\varepsilon_-)}\cdot q^{\inv(\varepsilon_+\rvert_{E_0\sqcup E\sqcup I_+})}\cdot q^{\inv(\varepsilon_0\rvert_{I_-\sqcup E_0\sqcup E_1\sqcup I_+})}=q^{\inv(E\sqcup E_0\sqcup I_+,I_-\sqcup E_1)+\inv(E_0\sqcup I_+,E)+\inv(E_1\sqcup I_+,I_-\sqcup E_0)}.
\end{equation*}
Collecting terms, this is
\begin{equation*}
    q^{\widetilde{\inv}(\varepsilon_{dif})}\cdot q^{\lvert E_0\rvert\cdot\lvert E_1\rvert+\inv(E_0,E)+\inv(E,E_1)}=q^{\widetilde{\inv}(\varepsilon_{dif})+\lvert E_0\sqcup E\rvert\cdot\lvert E_1\rvert+\inv(E_0,E)-\inv(E_1,E)}
\end{equation*}
Note also that
\begin{equation*}
\begin{split}
    &\binom{\lvert E_0\sqcup E\sqcup I_+\rvert}{2}+\alpha\cdot\lvert E_0\sqcup E\sqcup I_+\rvert+\binom{\lvert E_1\sqcup I_+\rvert}{2}+\alpha\cdot\lvert E_1\sqcup I_+\rvert+\lvert E_0\sqcup E\rvert\cdot\lvert E_1\rvert\\
    =\ &\binom{\lvert I_+\rvert}{2}+\alpha\cdot\lvert I_+\rvert+\binom{\lvert E_0\sqcup E\sqcup E_1\sqcup I_+\rvert}{2}+\alpha\cdot\lvert E_0\sqcup E\sqcup E_1\sqcup I_+\rvert,
\end{split}
\end{equation*}
that is, that
\begin{equation*}
    q^{\lvert E_0\sqcup E\rvert\cdot\lvert E_1\rvert}=\frac{\Dfactor(\lvert I_+\rvert) \cdot\Dfactor(\lvert E_0\sqcup E\sqcup E_1\sqcup I_+\rvert)}{\Dfactor(\lvert E_0\sqcup E\sqcup I_+\rvert)\cdot\Dfactor(\lvert E_1\sqcup I_+\rvert)}=\frac{\Dfactor(\lvert I_+\rvert)\cdot\Dfactor(\lvert I_+\sqcup I_0\rvert)}{\Dfactor(\dim(L/L_-))\cdot\Dfactor(k.-\dim(L_0/L_+))}
\end{equation*}
Now the case \genf  follows from \Cref{genMainProp}.

For the case \gens, we note that given $(L_-,L_+,L_0)\in A_L^{(k)}$ and $\varepsilon_1$ as above, the number of $L_1\in\Lat(V_2)$ satisfying $L_-\subseteq L_1\subseteq L_0$ and $L_+\not\subseteq L_1$ is
\begin{equation*}
    q^{\inv((\varepsilon_1-\varepsilon_0)\rvert_{I_-\sqcup E\sqcup E_0})}-\begin{cases}q^{\inv((\varepsilon_1-\varepsilon_0)\rvert_{I_-\sqcup E_0})}&\text{if }\varepsilon_+\ge\varepsilon_{dif}+\beta,\\0&\text{otherwise.}
    \end{cases}
\end{equation*}
This is
\begin{equation*}
    q^{\inv(\beta,I_-\sqcup E_0)}\cdot\begin{cases}q^{\inv(\beta,E)}-1&\text{if }\beta\in I_-\sqcup E_0,\\q^{\inv(\beta,E)}&\text{if }\beta\in E.\end{cases}
\end{equation*}
Together with \Cref{genMainProp} and the above case \genf, we get that $T_{\le k,r}^{\sharp,*}(\delta)$ is
\begin{equation*}
\begin{split}
    &\sum_{\{1,\ldots,\cfactor r\}=I_-\sqcup I_0\sqcup I_+}q^{\widetilde{\inv}(\varepsilon_{dif})}\frac{\Dfactor(\lvert I_+\rvert)\cdot\Dfactor(\lvert I_+\sqcup I_0\rvert)}{\Dfactor(\cfactor r-k)}\\
    &\qquad\cdot\left(\begin{array}{cl}&(W_0(\cfactor r-k,\lvert I_0\rvert)+(q^{r-k+1}-1)W_0(\cfactor r-k+1,\lvert I_0\rvert))\cdot t^2(\varepsilon_{dif})(\delta)\\
    +&(q-1)\sum_{\substack{I_0=E_0\sqcup E\sqcup E_1\\\lvert E\rvert=\cfactor r-k+1\\\beta\in I_-\sqcup E_0\sqcup E}}q^{\inv(\beta,I_-\sqcup E_0)}(q^{\inv(\beta,E)}-1_{\beta\in I_-\sqcup E_0})\cdot t^2(\varepsilon_{dif}+\beta)(\delta)\end{array}\right).
\end{split}
\end{equation*}
Changing the notation so that $\{1,\ldots,\cfactor r\}=I_-\sqcup I_0\sqcup I_+$ corresponds to $\varepsilon_{dif}+\beta$ in the last term, and changing $k\mapsto\cfactor r-k$ for notational simplicity, we have
\begin{equation*}
\begin{split}
    T_{\le \cfactor r-k,r}^{\sharp,*}(\delta)=\sum_{\{1,\ldots,\cfactor r\}=I_-\sqcup I_0\sqcup I_+}&q^{\widetilde{\inv}(\varepsilon)}\frac{\Dfactor(\lvert I_+\rvert)\cdot\Dfactor(\lvert I_+\sqcup I_0\rvert)}{\Dfactor(k)}\\
    &\cdot\left(\begin{array}{cl}&(W_0(k,\lvert I_0\rvert)+(q^{k+1}-1)W_0(k+1,\lvert I_0\rvert))\\
    +&(q-1)Y_k(I_0,I_+)\end{array}\right)\cdot t^2(\varepsilon)(\delta)
\end{split}
\end{equation*}
where $Y_k(I_0,I_+)$ is
\begin{equation}\label{Ykdef}
\begin{split}
    &\sum_{\beta\in I_+}\sum_{\substack{I_0=E_0\sqcup E\sqcup E_1\\ \lvert E\rvert=k}}\exp_q(\inv(E_0,E)-\inv(E_1,E)+\inv(I_+\sqcup E_0,\beta))\\
    &\qquad+\sum_{\beta\in I_+}\sum_{\substack{I_0=E_0\sqcup E\sqcup E_1\\ \lvert E\rvert=k+1}}\exp_q(\inv(E_0,E)-\inv(E_1,E)+\inv(I_+\sqcup E_1,\beta))(q^{\inv(\beta,E)}-1)\\
    &\qquad+\sum_{\beta\in I_0}\sum_{\substack{I_0=\beta\sqcup E_0\sqcup E\sqcup E_1\\ \lvert E\rvert=k+1}}\exp_q(\inv(E_0,E)-\inv(E_1,E)+\inv(I_+\sqcup E_1,\beta))(1-q^{-\inv(\beta,E)}).
\end{split}
\end{equation}
Now the theorem follows from the following lemma.
\end{proof}

\begin{lemma}
    For $I_0\sqcup I_+$ a totally ordered set, define $Y_k(I_0,I_+)$ as in \eqref{Ykdef}. Then we have
    \begin{equation*}
        (q-1)Y_k(I_0,I_+)=q^{\lvert I_+\rvert+\frac{1}{2}(\lvert I_0\rvert-k)}W_{1/2}(k,\lvert I_0\rvert)-W_0(k,\lvert I_0\rvert)-(q^{k+1}-1)W_0(k+1,\lvert I_0\rvert).
    \end{equation*}
\end{lemma}
\begin{proof}
We first verify that $Y_k(I_0,I_+)$ only depends on $\lvert I_0\rvert$ and $\lvert I_+\rvert.$ In order to see this, we rewrite the terms in \eqref{Ykdef} as follows: for a given $\beta\in I_+\sqcup I_0,$ we write $I_?^<,I_?^>$ to denote the parts before and after $\beta.$ Then we can rewrite the first term for a given $\beta\in I_+$ as
\begin{equation*}
\begin{split}
    &q^{\lvert I_+^<\rvert}\sum_{a+b=k}\sum_{\substack{I_0^>=E_0^>\sqcup E^>\sqcup E_1^>\\ \lvert E^>\rvert=b}}\sum_{\substack{I_0^<=E_0^<\sqcup E^<\sqcup E_1^<\\ \lvert E^<\rvert=a}}\exp_q\left(\begin{array}{@{}c@{}}\inv(E_0^<,E^<)-\inv(E_1^<,E^<)\\+\inv(E_0^>,E^>)-\inv(E_1^>,E^>)\\+\lvert E_0^<\rvert\cdot\lvert E^>\rvert-\lvert E_1^<\rvert\cdot\lvert E^>\rvert+\lvert E_0^<\rvert\end{array}\right)\\
    =\ &q^{\lvert I_+^<\rvert}\sum_{a+b=k}W_0(b,\lvert I_0^>\rvert)\sum_{\substack{I_0^<=E_0^<\sqcup E^<\sqcup E_1^<\\ \lvert E^<\rvert=a}}\exp_q\left(\begin{array}{@{}c@{}}\inv(E_0^<,E^<)-\inv(E_1^<,E^<)\\+(b+\frac{1}{2})(\lvert E_0\rvert-\lvert E_1\rvert)+\frac{1}{2}(\lvert E_0\rvert+\lvert E_1\rvert)\end{array}\right)\\
    =\ &\sum_{a+b=k}q^{\lvert I_+^<\rvert+\frac{\lvert I_0^<\rvert-a}{2}}W_0(b,\lvert I_0^>\rvert)W_{b+1/2}(a,\lvert I_0^<\rvert).
\end{split}
\end{equation*}
Similarly, the second resp.\  third terms for a given $\beta\in I_+$ resp.\ $\beta\in I_0$ are
\begin{equation*}
    \sum_{a+b=k+1}q^{\lvert I_+^<\rvert+\frac{\lvert I_0^<\rvert-a}{2}}W_0(b,\lvert I_0^>\rvert)W_{b-1/2}(a,\lvert I_0^<\rvert)(q^b-1)
\end{equation*}
resp.\ 
\begin{equation*}
    \sum_{a+b=k+1}q^{\lvert I_+^<\rvert+\frac{\lvert I_0^<\rvert-a}{2}}W_0(b,\lvert I_0^>\rvert)W_{b-1/2}(a,\lvert I_0^<\rvert)(1-q^{-b}).
\end{equation*}
Now assume that we have two neighboring indices $i,j\in I_0\sqcup I_+$ where $i\in I_0$ and $j\in I_+.$ Let $Y_1\defeq Y_k(I_0,I_+)$ and $Y_2\defeq Y_k((I_0\sqcup\{j\})\setminus\{i\},(I_+\sqcup\{i\})\setminus\{j\}).$ We want to prove that $Y_1=Y_2.$ Without loss of generality, assume $i<j,$ and denote $a_0=\lvert I_0^{<i}\rvert,$ $b_0=\lvert I_0^{>j}\rvert$ and $a_+=\lvert I_+^{<i}\rvert.$ To compute the difference $Y_1-Y_2,$ we only need to consider the terms with $\beta\in\{i,j\}.$ Thus the difference is
\begin{equation*}
\begin{split}
    &\sum_{a+b=k}q^{a_++\frac{a_0-a}{2}}\left(q^{\frac{1}{2}}W_0(b,b_0)W_{b+1/2}(a,a_0+1)-W_0(b,b_0+1)W_{b+1/2}(a,a_0)\right)\\
    &\qquad+(q^b-1)\sum_{a+b=k+1}q^{a_++\frac{a_0-a}{2}}\left(q^{\frac{1}{2}}W_0(b,b_0)W_{b-1/2}(a,a_0+1)-W_0(b,b_0+1)W_{b-1/2}(a,a_0)\right)\\
    &\qquad+(1-q)\sum_{a+b=k+1}q^{a_++\frac{a_0-a}{2}}W_0(b,b_0)W_{b-1/2}(a,a_0)(1-q^{-b}).
\end{split}
\end{equation*}
Now we rewrite this by denoting $C(a,b)\defeq q^{a_++\frac{a_0-a}{2}}W_0(b,b_0)W_{k-a+1/2}(a,a_0)$ and using \eqref{Wrec1} for the terms $W_{k-a+1/2}(a,a_0+1)$ as well as \eqref{Wrec2} for the terms $W_0(b,b_0+1).$ We have
\begin{equation*}
\begin{split}
    Y_1-Y_2=\ &\sum_{a+b=k}(q^{b+1}+q^{-b})C(a,b)+C(a-1,b)-(q^b+q^{-b})C(a,b)-C(a,b-1)\\
    &+\sum_{a+b=k+1}(q^b-1)\left((q^b+q^{1-b})C(a,b)+C(a-1,b)-(q^b+q^{-b})C(a,b)-C(a,b-1)\right)\\
    &-(q-1)\sum_{a+b=k+1}(1-q^{-b})C(a,b)\\
    =\ &\sum_{a+b=k+1}C(a,b)(q^b-1)(q^b+q^{1-b}-q^b-q^{-b}-(q-1)q^{-b})\\
    &+\sum_{a+b=k}C(a,b)\left(q^{b+1}+q^{-b}-q^b-q^{-b}+q^b-1-(q^{b+1}-1)\right)\\\
    &+\sum_{a+b=k-1}C(a,b)(1-1)\\
    =\ &0.
\end{split}
\end{equation*}

Thus $Y_k(I_0,I_+)$ is a function of $\lvert I_0\rvert$ and $\lvert I_+\rvert.$ This means that we can evaluate it by assuming that all of $I_+$ come before all of $I_0.$ This give us
\begin{equation*}
\begin{split}
    Y_k(I_-,I_0)=\ &\frac{q^{\lvert I_+\rvert}-1}{q-1}W_0(k,\lvert I_0\rvert)\\
    &+\frac{q^{\lvert I_+\rvert}-1}{q-1}(q^{k+1}-1)W_0(k+1,\lvert I_0\rvert)\\
    &+q^{\lvert I_+\rvert}\sum_{\substack{I_0=\beta\sqcup E_0\sqcup E\sqcup E_1\\ \lvert E\rvert=k+1}}\exp_q(\inv(E_0,E)-\inv(E_1,E)+\inv(E_1,\beta))(1-q^{-\inv(\beta,E)}).
\end{split}
\end{equation*}
Denoting
\begin{equation*}
    \tilde{Y}(m,n)\defeq\sum_{\substack{\{1,\ldots,n\}=\beta\sqcup E_0\sqcup E\sqcup E_1\\\lvert E\rvert=m}}\exp_q(\inv(E_0,E)-\inv(E_1,E)+\inv(E_1,\beta))(1-q^{-\inv(\beta,E)}),
\end{equation*}
we are left to check that
\begin{equation*}
    (q-1)\tilde{Y}(m,n)\isequal q^{(n+1-m)/2}W_{1/2}(m-1,n)-W_0(m-1,n)-(q^{m}-1)W_0(m,n).
\end{equation*}
We do this by induction on $n+m,$ the base case being $m=0,$ which is trivial since both sides are $0.$ Thinking about the first element of $\{1,\ldots,n\}$ in the above expression for $\tilde{Y}(n,m)$ give us the recursion
\begin{equation*}
    \tilde{Y}(m+1,n+1)=\tilde{Y}(m,n)+(q^{m+1}+q^{-m})\tilde{Y}(m+1,n)+(1-q^{-m-1})W_0(m+1,n).
\end{equation*}
By the induction hypothesis we get that $(q-1)\tilde{Y}(m+1,n+1)$ is
\begin{equation*}
\begin{split}
    &q^{(n+1-m)/2}W_{1/2}(m-1,n)-W_0(m-1,n)-(q^m-1)W_0(m,n)\\
    &\qquad+(q^{m+1}+q^{-m})(q^{(n-m)/2}W_{1/2}(m,n)-W_0(m,n)-(q^{m+1}-1)W_0(m+1,n))\\
    &\qquad+(q-1)(1-q^{-m-1})W_0(m+1,n)\\\
    =\ &q^{(n+1-m)/2}\left(W_{1/2}(m-1,n)+(q^{m+1/2}+q^{-m-1/2})W_{1/2}(m,n)\right)\\
    &\qquad-\left(W_0(m-1,n)+(q^m+q^{-m})W_0(m,n)\right)\\
    &\qquad-(q^{m+1}-1)\left(W_0(m,n)+(q^{m+1}+q^{-m-1})W_0(m+1,n)\right)
\end{split}
\end{equation*}
and thus the induction is complete by using \eqref{Wrec2}.
\end{proof}

%% file: 4.4-SmallHeckeflat.tex
\begin{definition}\index{t1(epsilon)@$t^1(\varepsilon)$}
For $\varepsilon\in\{-1,1\}^{\cfactor r},$ we denote $t^{1}(\varepsilon)\colon R[\Typ_r]\to R[\Typ_r]$ by
\begin{equation*}
    t^{1}(\varepsilon)\Type{e_i}{\chi_i}\defeq\begin{cases}\Type{e_i+\varepsilon_i}{\chi_i}&\text{in the cases \Ugen, \Ogen,}\\
    \type{e_i+\frac{\varepsilon_{2i-1}+\varepsilon_{2i}}{2}}&\text{in the case \SPgen}.\end{cases}
\end{equation*}
\end{definition}
\begin{definition}\index{1Deltacircbullet@$\Delta^\intert_r$}
We define $\Delta^{\intert}_r\colon R[\Typ_r]\to R[\Typ_r]$ by
\begin{equation*}
    \Delta^\intert_r\defeq\sum_{\varepsilon\in\{-1,1\}^{\cfactor r}}q^{\sum_i(\cfactor r-i+\alpha)\frac{1+\varepsilon_i}{2}} t^1(\varepsilon).
\end{equation*}
\end{definition}
\begin{remark}\label{DeltaAlternative}
Note that if $\varepsilon\in\{0,1\}^{\cfactor r},$ then
\begin{equation*}
    \sum_{i=1}^{\cfactor r}(\cfactor r-i)\cdot \varepsilon_i=\binom{\lambda_1(\varepsilon)}{2}+\inv(\varepsilon),
\end{equation*}
so we may write
\begin{equation*}
    \Delta^\intert_r=\sum_{\varepsilon\in\{0,1\}^{\cfactor r}}q^{\binom{\lambda_1(\varepsilon)}{2}+\alpha\lambda_1(\varepsilon)}q^{\inv(\varepsilon)}t^2(\varepsilon)t(-1,\ldots,-1).
\end{equation*}
In other words,
\begin{equation*}
    \Delta^\intert_r=t(-1,\ldots,-1)\cdot\sum_{k=0}^{\cfactor r}\Dfactor(k)\Delta^\natural_{k,r}
\end{equation*}
where $\Delta^\natural_{k,r}\colon R[\Typ_r]\to R[\Typ_r]$\index{1Deltanatural@$\Delta^\natural_{k,r}$}\index{1Deltanatural@$\Delta^\natural_{k,r}$! |seealso{\cite[\citeDeltadefinition]{CoratoZanarella}}} are the operators appearing in \cite[\citeDeltadefinition]{CoratoZanarella}.
\end{remark}

\begin{proposition}\label{intertpreserves}
    $\Delta_r^\intert$ preserves $\Rel^\flat_r.$
\end{proposition}
\begin{proof}
    The above remark together with \cite[\citeDeltapreserves]{CoratoZanarella} (and the fact that $t(-1,\ldots,-1)$ preserves $\Rel^\natural_r$) implies that $\Delta^\intert_r$ preserves $\Rel^\natural_r.$
    
    For $\Rel_1^\flat,$ this follows from \Cref{preservesfs}, as we are looking at the operators $t(-1)+q_0\cdot t(1)$ in the case \Uf, $t(-1)+t(1)$ in the case \Of
    \begin{equation*}
        t^1(-1,-1)+q\cdot t^1(-1,1)+q^2\cdot t^1(1,-1)+q^3\cdot t^1(1,1)=(t(-1)+q^3\cdot t(1))+q(q+1)\cdot t(0)
    \end{equation*}
    in the case \SPf.
\end{proof}

\begin{proposition}\label{intertwiningProp}
    Fix $\Lambda^\bullet\in\Lat^\bullet(V^\flat)$ and denote $L^\bullet\defeq\mathrm{int}_2(\Lambda^\bullet).$  Assume we are given $\varpi L^\bullet\subseteq L^\circ\subseteq L^\bullet$ with $L^\circ\subseteq L^{\circ,\vee}.$ 
    
    Then the number of $\Lambda^\circ\in\Lat^\circ(V^\flat)$ with $\varpi\Lambda^\bullet\subseteq\Lambda^\circ\subseteq\Lambda^\bullet$ such that $\mathrm{int}_2(\Lambda^\circ)=L^\circ$ is
    \begin{equation*}
        L\typec{\dim\left(\frac{L^\bullet\cap L^{\bullet,\vee}}{L^\circ\cap L^{\bullet,\vee}}\right)}{\mathrm{typ}\left(\frac{L^\bullet\cap L^{\circ,\vee}}{L^\circ+L^\bullet\cap L^{\bullet,\vee}}\right)}{\dim\left(\frac{L^\bullet}{L^\circ}\right)}
    \end{equation*}
\end{proposition}
\begin{proof}
    First note that $\varpi\Lambda^\bullet+L^\circ$ is integral, as $L^\circ\subseteq\Lambda^\bullet$ and $L^\circ\subseteq L^{\circ,\vee}.$ So we are counting the number of $\Lambda^\circ\in\Lat^\circ(V^\flat)$ such that $(\varpi\Lambda^\bullet+L^\circ)\subseteq\Lambda^\circ\subseteq(\varpi\Lambda^\bullet+L^\circ)^\vee$ and $\Lambda^\circ\cap(\varpi\Lambda^\bullet+L^\bullet)=\varpi\Lambda^\bullet+L^\circ.$

   Let $N\defeq(\varpi\Lambda^\bullet+L^\bullet)\cap(\varpi\Lambda^\bullet+L^\circ)^\vee.$ By \Cref{Ldefinition}, the answer we seek is 
   \begin{equation*}
       L\left(\dim\left(\frac{N\cap N^\vee}{\varpi\Lambda^\bullet+L^\bullet}\right),\mathrm{typ}\left(\frac{N}{N\cap N^\vee}\right),\frac{1}{2}\dim\left(\frac{(\varpi\Lambda^\bullet+L^\circ)^\vee}{\varpi\Lambda^\bullet+L^\circ}\right)\right).
   \end{equation*}
   We have
   \begin{equation*}
   \begin{split}
       \frac{1}{2}\dim\left(\frac{(\varpi\Lambda^\bullet+L^\circ)^\vee}{\varpi\Lambda^\bullet+L^\circ}\right)&=\frac{1}{2}\dim\left(\frac{\Lambda^\bullet}{\varpi\Lambda^\bullet}\right)-\dim\left(\frac{\varpi\Lambda^\bullet+L^\circ}{\varpi\Lambda^\bullet}\right)\\&=\dim\left(\frac{L^\bullet}{\varpi L^\bullet}\right)-\dim\left(\frac{L^\circ}{\varpi L^\bullet}\right)\\&=\dim\left(\frac{L^\bullet}{L^\circ}\right).
    \end{split}
   \end{equation*}
    Also,
    \begin{equation*}
        N=(\varpi\Lambda^\bullet+L^\bullet)\cap(\Lambda^\bullet\cap\mathrm{proj}_2^{-1}(L^{\circ,\vee})=(\varpi\Lambda^\bullet+L^\bullet)\cap\mathrm{proj}_2^{-1}(L^{\circ,\vee})=\varpi\Lambda^\bullet+(L^\bullet\cap L^{\circ,\vee})
    \end{equation*}
    and
    \begin{equation*}
        N^\vee=(\varpi\Lambda^\bullet+L^\bullet)^\vee+(\varpi\Lambda^\bullet+L^\circ)=(\Lambda^\bullet\cap\mathrm{proj}_2^{-1}(L^{\bullet,\vee}))+(\varpi\Lambda^\bullet+L^\circ).
    \end{equation*}
    Thus
    \begin{equation*}
        N\cap N^\vee=\varpi\Lambda^\bullet+L^\circ+(L^\bullet\cap L^{\circ,\vee}\cap L^{\bullet,\vee})=\varpi\Lambda^\bullet+L^\circ+(L^\bullet\cap L^{\bullet,\vee}).
    \end{equation*}
    So we may compute that
    \begin{equation*}
        \frac{N\cap N^\vee}{\varpi\Lambda^\bullet+L^\circ}\iso \frac{L^\bullet\cap L^{\bullet,\vee}}{(L^\bullet\cap L^{\bullet,\vee})\cap(\varpi\Lambda^\bullet+L^\circ)}=\frac{L^\bullet\cap L^{\bullet,\vee}}{L^\bullet\cap L^{\bullet,\vee}\cap L^\circ}=\frac{L^\bullet\cap L^{\bullet,\vee}}{L^\circ\cap L^{\bullet,\vee}}
    \end{equation*}
    and
    \begin{equation*}
    \frac{N}{N\cap N^\vee}\iso\frac{L^\bullet\cap L^{\circ,\vee}}{(L^\circ+(L^\bullet\cap L^{\bullet,\vee}))\cap (L^\bullet\cap L^{\circ,\vee})}=\frac{L^\bullet\cap L^{\circ,\vee}}{L^\circ+L^\bullet\cap L^{\bullet,\vee}}.\qedhere
    \end{equation*}
\end{proof}

\begin{theorem}\label{intert}
The induced map $\Delta^\intert_r\colon\Z[\Typ_r^0]\to\Z[\Typ_r^0]$ agrees with $T^{\circ\bullet,*}_r.$
\end{theorem}
\begin{proof}
Assume we are given $\Lambda^\bullet\in\Lat^\bullet(V^\flat),$ and denote $L^\bullet\defeq\mathrm{int}_2(\Lambda^\bullet).$

Choose a decomposition $V_2=\bigoplus_{i\in\Z}V_2^{(i)}$ such that $L^{\bullet}=\bigoplus L^{\bullet,(i)}$ with $(L^{\bullet,(i)})^\vee=\varpi^{-i}L^{\bullet,(i)},$ then $V_2^{(i)}=0$ for $i<-1.$ Hence
\begin{equation*}
    L^\bullet\cap L^{\bullet,\vee}=\varpi L^{\bullet,(-1)}\oplus L^{\bullet,(\ge0)}.
\end{equation*}

Now assume we are  given $\varpi L^\bullet\subseteq L^\circ\subseteq L^\bullet$ with $L^\circ\subseteq L^{\circ,\vee}.$ Denote $N_i\defeq\mathrm{proj}_{V_2^{(i)}}(L^\circ\cap V_2^{(\ge i)}).$ Then we have
\begin{equation*}
\begin{split}
    L^\circ\cap L^{\bullet,\vee}&=\varpi L^{\bullet,(-1)}\oplus(L^{\bullet,(\ge0)}\cap L^\circ)=\varpi L^{\bullet,(-1)}\oplus(L^\circ\cap V_2^{(\ge0)}),\\
    L^\bullet\cap L^{\circ,\vee}&=(L^{\bullet,(-1)}\cap L^{\circ,\vee})\oplus L^{\bullet,(\ge0)}=(N_{-1})^\vee\oplus L^{\bullet,(\ge0)},\\
    L^\circ+L^\bullet\cap L^{\bullet,\vee}&=N_{-1}\oplus L^{\bullet,(\ge0)}.
\end{split}
\end{equation*}

Now assume $L^\circ$ corresponds to the data $D=(n,m,l,\psi_1,\psi_2)$ as in \cite[\citemainLemma]{CoratoZanarella}. This means that if $\mathrm{typ}^\bullet(\Lambda^\bullet)=(e^0,\chi^0),$ then
\begin{equation*}
    e^0(i)=n(i-1)+\frac{2}{\cfactor}\cdot m(i-1)+l(i-1)\quad\text{and}\quad\chi^0(i)=\psi_1(i-1)\cdot\epsilon^{m(i-1)}\cdot\psi_2(i-1),
\end{equation*}
and if $\Lambda^\circ\in\Lat^\circ(V^\flat)$ is such that $\mathrm{int}_2(\Lambda^\circ)=L^\circ,$ then $(f^0(D),\psi^0(D))\defeq\mathrm{typ}^\circ(\Lambda^\circ)$ is given by
\begin{equation*}
    f^0(D)(i)=n(i)+\frac{2}{\cfactor}\cdot m(i-1)+l(i-2)\quad\text{and}\quad\psi^0(i)=\psi_1(i)\cdot\epsilon^{m(i-1)}\cdot\psi_2(i-2).
\end{equation*}
Note that $n(-1)=0,$ and thus $\psi_1(-1)=1.$ Denote $a_i(D)=m(i)+\cfactor\cdot l(i)=\dim(L^{\bullet,(i)}/N_i).$ By \Cref{intertwiningProp} and \Cref{Lquantitysub}, the number of $\Lambda^\circ\in\Lat(V^\flat)$ with $\varpi\Lambda^\bullet\subseteq\Lambda^\circ$ and $\mathrm{int}_2(\Lambda^\circ)=L^\circ$ is
\begin{equation*}
    L\Typec{\sum_{i\ge0}a_i(D)}{}{l(-1)}{\psi_2(-1)}{\Sigma (a(D))}{}=L\Typec{0}{}{l(-1)}{\chi(0)\cdot\epsilon^{m(-1)}}{a_{-1}(D)}{}\frac{\Dfactor(\Sigma(a(D)))}{\Dfactor(a_{-1}(D))}
\end{equation*}
Note that
\begin{equation*}
\begin{split}
    &\frac{\Sigma(a(D))^2}{2}+\sum_{i>j}(m(i)+\cfactor\cdot l(i))\cdot(m(j)+\cfactor\cdot n(j))\\
    &\qquad=\frac{\Sigma(a(D))^2}{2}+\sum_{i>j}a_i(D)\cdot(\cfactor\cdot e^0(j+1)-a_j(D))\\
    &\qquad=\sum_i\frac{a_i(D)^2}{2}+\cfactor\cdot\sum_ia_i(D)\cdot \lambda_{\le i}(e^0).
\end{split}
\end{equation*}

At last, together with \cite[\citemainLemmaCount]{CoratoZanarella}, this means that $T^{\circ\bullet,*}_r(\delta(e^0,\chi^0))$ is
\begin{equation*}
    \sum_{D=(n,m,l,\psi_1,\psi_2)}\left(\prod_{i\ge0}Q_i(D)\right)\cdot(f^0(D),\psi^0(D))
\end{equation*}
where
\begin{equation*}
    Q_{0}(D)\defeq S\TypeC{m(-1)}{}{e^0(0)}{\chi^0(0)}
    \cdot L\Typec{0}{}{l(-1)}{\chi(0)\cdot\epsilon^{m(-1)}}{m(-1)+\cfactor \cdot l(-1)}{}.
\end{equation*}
and
\begin{equation*}
    Q_{i+1}(D)\defeq Q\Typec{n(i)}{\psi_1(i)}{m(i)}{}{l(i)}{\psi_2(i)}q^{\cfactor\cdot a_i(D)\cdot\lambda_{\le i}(e^0)}\cdot \Dfactor(a_i(D))\quad\text{for }i\ge0,
\end{equation*}
for $Q\Typec{n(i)}{\psi_1(i)}{m(i)}{}{l(i)}{\psi_2(i)}$ as defined in \cite[\citemainLemmaFinite]{CoratoZanarella}.

As in the proof of \cite[\citeinducedThm]{CoratoZanarella}, we have
\begin{equation*}
(f^0(D),\psi^0(D))\equiv\cdots\star\delta^{(i+1)}(D)\star\delta^{(i)}(D)\star\cdots\star\delta^{(-1)}(D)\mod\Rel^\natural
\end{equation*}
where
\begin{equation*}
    \delta^{(i)}(D)\defeq\Typec{(i+2)^{l(i)}}{\psi_2(i)}{(i+1)^{\frac{2}{\cfactor}\cdot m(i)}}{\epsilon^{m(i)}}{(i)^{n(i)}}{\psi_1(i)}.
\end{equation*}
This reduces to prove the theorem for the case $e=(k,\ldots,k)$ for each $k\ge0.$ When $k>0$ and in light of \Cref{DeltaAlternative}, this follows from \cite[\citeinducedThm]{CoratoZanarella}.

So it remains to analyze the case $e=(0,\ldots,0).$ If we consider the expression for $\Delta^\intert_{r+1}$ and separate the sum over $\varepsilon$ according to $\varepsilon_{1},$ we have that $\Delta^\intert_{r+1}$ is
\begin{equation*}
\begin{array}{c|c}
    \Ugen&(t(-1)+q_0^{2r+1}t(1))\star \Delta^\intert_{r}\\
    \hline
    \Ogen&(t(-1)+q^{r}t(1))\star \Delta^\intert_{r}\\
    \hline
    \SPgen&(t(-1)+q^{2r+1}(q+1)t(0)+q^{4r+3}t(1))\star \Delta^\intert_{r}
\end{array}
\end{equation*}
Now using the following straightening relations modulo $\Rel^\natural_2,$
\begin{equation*}
\begin{array}{c|c}
    \Ugen&(-1,(1)^l)\equiv(-q_0)^l\cdot ((1)^l,-1)+(1-(-q_0)^l)\cdot((1)^{l-1},(0)^2)\\
    \hline
    \Ogen&\begin{array}{l}
    \TypeC{-1}{\chi_1}{(1)^l}{\chi_2}\equiv(1-(q\epsilon)^{\lfloor l/2\rfloor}\chi_2(-\epsilon\chi_1)^l)\TypeC{(1)^{l-1}}{\epsilon\chi_1\chi_2}{(0)^{2}}{\epsilon}\\
    \quad\quad+(q\epsilon)^{\lfloor l/2\rfloor}\chi_2(-\epsilon\chi_1)^l
    \left(\frac{1+q\epsilon}{2}\TypeC{(1)^l}{\chi_2}{-1}{\chi_1}+\frac{1-q\epsilon}{2}
    \TypeC{(1)^l}{(-1)^l\chi_2}{-1}{(-1)^l\chi_1}\right)
    \end{array}\\
    \hline
    \SPgen&(-1,(1)^l)\equiv q^{2l}\cdot((1)^l,-1)+(1-q^{2l})\cdot((1)^{l-1},(0)^2)
\end{array}
\end{equation*}
the induction hypothesis implies that $\Delta^\intert_{r+1}\left(\TypeCC{(0)^r}{\chi}\star\TypeCC{0}{\eta}\right)$ is congruent, modulo $\Rel^\flat_{r+1},$ to
\begin{equation*}
\begin{array}{c|c}
    \Ugen&\displaystyle\sum_{2m+l=r}Q_0(m,l)\left(\begin{array}{@{}r@{}c@{}l@{}}&(1-(-q_0)^l)&\typeC{(1)^{l-1}}{(0)^{2(m+1)}}\\+&(-q_0)^l&\typeC{(1)^{l+1}}{(0)^{2m}}\\+&q_0^{2r+1}&\typeC{(1)^{l+1}}{(0)^{2m}}\end{array}\right)\\
    \hline
    \Ogen&\displaystyle\sum_{2m+l=r}Q_0\TypeC{m}{}{l}{\chi\epsilon^m}\left(\begin{array}{@{}r@{}c@{}l@{}}
    &(1-(q\epsilon)^{\lfloor l/2\rfloor}\chi\epsilon^m(-\epsilon\eta)^l)&\TypeC{(1)^{l-1}}{\chi\epsilon^{m+1}\eta}{(0)^{2(m+1)}}{\epsilon^{m+1}}\\
    +&(q\epsilon)^{\lfloor l/2\rfloor}\chi\epsilon^m(-\epsilon\eta)^l\frac{1+q\epsilon}{2}&\TypeC{(1)^{l+1}}{\chi\epsilon^m\eta}{(0)^{2m}}{\epsilon^m}\\
    +&(q\epsilon)^{\lfloor l/2\rfloor}\chi\epsilon^m(-\epsilon\eta)^l\frac{1-q\epsilon}{2}&\TypeC{(1)^{l+1}}{\chi\epsilon^m\eta}{(0)^{2m}}{\epsilon^m}\\
    +&q^{r}&\TypeC{(1)^{l+1}}{\chi\epsilon^m\eta}{(0)^{2m}}{\epsilon^m}\end{array}\right)\\
    \hline
    \SPgen&\displaystyle\sum_{m+l=r}Q_0(m,l)\left(\begin{array}{@{}r@{}c@{}l@{}}&(1-q^{2l})&\typeC{(1)^{l-1}}{(0)^{m+2}}\\-&q^{2l+1}&\typeC{(1)^{l+1}}{(0)^{m}}\\
    +&q^{2l}(q+1)&\typeC{(1)^{l}}{(0)^{m+1}}\\+&q^{2r+1}(q+1)&\typeC{(1)^l}{(0)^{m+1}}\\+&q^{4r+3}&\typeC{(1)^{l+1}}{(0)^{m}}\end{array}\right)
\end{array}
\end{equation*}
Collecting terms, it remains to check that $Q_0\TypeC{m}{}{l}{\chi\eta\epsilon^m}$ agrees with
\begin{equation*}
\begin{array}{c|c}
    \Ugen&\begin{array}{@{}r@{}c@{\,}l@{}}&(1-(-q_0)^{l+1})&Q_0(m-1,l+1)\\+&(-q_0)^{l-1}&Q_0(m,l-1)\\+&q_0^{4m+2l-1}&Q_0(m,l-1)\end{array}\\
    \hline
    \Ogen&\begin{array}{@{}r@{}c@{\,}l@{}}&(1-(q\epsilon)^{\lfloor (l+1)/2\rfloor}\chi\epsilon^{m-1}(-\epsilon\eta)^{l+1})&Q_0\TypeC{m-1}{}{l+1}{\chi\epsilon^{m-1}}\\
    +&((q\epsilon)^{\lfloor (l-1)/2\rfloor}\chi\epsilon^m(-\epsilon\eta)^{l-1}+q^{2m+l-1})&Q_0\TypeC{m}{}{l-1}{\chi\epsilon^m}\end{array}\\
    \hline
    \SPgen&\begin{array}{@{}r@{}c@{\,}l@{}}&(1-q^{2l+2})&Q_0(m-2,l+1)\\+&(q^{4m+4l-1}-q^{2l-1})&Q_0(m,l-1)\\+&(q^{2m+2l-1}+q^{2l})(q+1)&Q_0(m-1,l)\end{array}
\end{array}
\end{equation*}
Dividing by $Q_0\TypeC{m}{}{l}{\chi\eta\epsilon^m},$
we are looking at
\begin{equation*}
\begin{array}{c|c}
    \Ugen&\displaystyle\frac{1}{((-q_0)^{l+2m}-1)^2}\cdot\left(\begin{array}{@{}r@{}l@{}}&(-q_0)^{l}(q_0^{2m}-1)^2\\+&(1-(-q_0)^{l+4m})(1-(-q_0)^l)\end{array}\right)\\
    \hline
    \Ogen&\begin{array}{l}\displaystyle\frac{1}{(q^{m+\lfloor l/2\rfloor}-\epsilon^{m+\lfloor l/2\rfloor}\chi\eta(-\eta)^l)^2}\\
    \qquad\cdot\displaystyle\left(\begin{array}{@{}r@{}l@{}}&\epsilon^{m+\lfloor l/2\rfloor}\chi\eta(-\eta)^lq^{\lfloor l/2\rfloor}(q^m-1)^2\\
    +&((q\epsilon)^{\lfloor (l-1)/2\rfloor}\chi\epsilon^m(-\epsilon\eta)^{l-1}+q^{2m+l-1})q^{-\lfloor(l-1)/2\rfloor}(q^{\lfloor l/2\rfloor}-\epsilon^{\lfloor l/2\rfloor}\chi\eta\epsilon^m(-\eta)^l)\end{array}\right)\end{array}\\
    \hline
    \SPgen&\displaystyle\frac{1}{(q^{2m+2l}-1)^2}\cdot\left(\begin{array}{@{}r@{}l@{}}-&q^{2l+1}(q^m-1)^2(q^{m-1}-1)^2\\+&(q^{4m+2l}-1)(q^{2l}-1)\\+&q^{2l}(q+1)(q^{2m-1}+1)(q^m-1)^2\end{array}\right)
\end{array}
\end{equation*}
which simplify to $1$ after a routine computation.
\end{proof}

%% file: 4.5-SmallHeckesharp.tex
\begin{proposition}\label{deltadif}
We have
\begin{equation*}
    \Delta_{\le1,r}^\sharp-\Delta_{\le1,r}^\flat=q^{r-1+\alpha}(q^r-1)+(q-1)q^{r-1+2\alpha}\Delta_{1,r}^\natural
\end{equation*}
where
\begin{equation*}
    \Delta_{1,r}^{\natural}\defeq\sum_{\substack{\varepsilon\in\{0,1\}^{\cfactor r}\\\lambda_1(\varepsilon)=1}}q^{\inv(\varepsilon)}t^2(\varepsilon).
\end{equation*}
\end{proposition}
\begin{proof}
    This is a direct computation for each of the cases \Us and \Os.
\end{proof}

The following is \Cref{TlessSums} specialized to $k=1.$
\begin{corollary}\label{tdif}
    If $\Lambda^\sharp=\O_Fu\oplus\Lambda^\flat$ for $\Lambda^\flat\in\Lat^\circ(V^\flat),$ denote
    \begin{equation*}
        T_{\le1,r}^{\mathrm{err}}(\mathrm{typ}^\sharp(\Lambda^\sharp))\defeq\sum_{\substack{\Lambda^\flat_-\in\Lat(V^\flat)\\\Lambda^\flat_-\sub{(1,\epsilon\psi)}\Lambda^\flat}}\mathrm{typ}^\natural
        (\mathrm{int}_2(\Lambda^\flat_-)).
    \end{equation*}
    Then
    \begin{equation*}
        T_{\le 1,r}^{\sharp,*}(\mathrm{typ}^\sharp(\Lambda^\sharp))=T_{\le 1,r}^{\flat,*}(\mathrm{typ}^\flat(\Lambda^\flat))+q^{r-1+\alpha}(q^r-1)\mathrm{typ}^\sharp(\Lambda^\sharp)+q^\alpha(q^\alpha+1)T_{\le1,r}^{\mathrm{err}}(\mathrm{typ}^\flat(\Lambda^\flat)).
    \end{equation*}
\end{corollary}

The following is a more precise version of \Cref{gencomp3} in the simplest case $L=\Lambda.$
\begin{proposition}\label{comp3}
    Assume we have $\Lambda\in\Lat^\circ(V^\flat).$ Denote $L_0=\mathrm{int}_2(\Lambda),$ and $(e^0,\chi^0)=\mathrm{typ}^\natural(L_0).$ Then the number of $\Lambda_1\in\Lat(V^\flat)$ satisfying i) $\Lambda_1\sub{(1,\epsilon\psi)}\Lambda,$ ii) $\mathrm{int}_2(\Lambda)=L_0$ is
    \begin{equation*}
        q^{r-e^0(0)}R\TypeC{1}{\psi}{e^0(0)}{\chi^0(0)}.
    \end{equation*}
    
    Furthermore, given $L_1\in\Lat(V_2)$ with $L_1\sub{1}L_0,$ we denote $(f^0,\psi^0)=\mathrm{typ}^\natural(L_1).$ Then the number of $\Lambda_1\in\Lat(V^\flat)$ satisfying i) $\Lambda_1\sub{(1,\epsilon\psi)}\Lambda,$ ii) $\mathrm{int}_2(\Lambda_1)=L_1$ is
    \begin{equation*}
        q^{r-1-f^0(0)}R\TypeC{1}{\psi}{f^0(0)+2}{\epsilon\psi^0(0)}-q^{r-e^0(0)}R\TypeC{1}{\psi}{e^0(0)}{\chi^0(0)}.
    \end{equation*}
\end{proposition}
\begin{proof}
    Denote $I\defeq\varpi\Lambda+L_0.$

    For the first claim, we are simply counting the number of $I\subseteq\Lambda_1\sub{(1,\epsilon\psi)}\Lambda.$ By taking duals, we are counting the number of lines of type $(1,\epsilon\psi)$ in $\varpi I^\vee/\varpi\Lambda.$ Thus the answer is
    \begin{equation*}
        q^{\dim(I\cap\varpi I^\vee/\varpi\Lambda)}\cdot R\TypeC{1}{\epsilon\psi}{\mathrm{typ}(\varpi I^\vee/I\cap\varpi I^\vee)}{}.
    \end{equation*}
    Now note that $\dim(I\cap\varpi I^\vee/\varpi\Lambda)=r-e^0(0),$ and $\mathrm{typ}(\varpi I^\vee/I\cap\varpi I^\vee)=(e^0(0),\epsilon^{e^0(0)}\chi^0(0)).$ The first claim then follows from the fact that
    \begin{equation*}
        R\TypeC{1}{\chi_1}{r}{\chi_2}=R\TypeC{1}{\epsilon\chi_1}{r}{\epsilon^r\chi_2}.
    \end{equation*}

    For the second claim, we denote $I_1\defeq\varpi\Lambda+L_1.$ Then we are counting $I_1\subseteq\Lambda_1\sub{(1,\epsilon\psi)}\Lambda$ such that $I\not\subseteq \Lambda.$ Similarly as above, this is
    \begin{equation*}
    \begin{split}
        &q^{\dim(I_1\cap\varpi I_1^\vee/\varpi\Lambda)}\cdot R\TypeC{1}{\epsilon\psi}{\mathrm{typ}(\varpi I_1^\vee/I_1\cap\varpi I_1^\vee)}{}\\
        &\qquad-q^{\dim(I\cap\varpi I^\vee/\varpi\Lambda)}\cdot R\TypeC{1}{\epsilon\psi}{\mathrm{typ}(\varpi I^\vee/I\cap\varpi I^\vee)}{}.
    \end{split}
    \end{equation*}
    Note that
    \begin{equation*}
        \dim(I_1\cap\varpi I_1^\vee/\varpi\Lambda)=r-f^0(0)-1,\quad\text{and}\quad\mathrm{typ}(\varpi I_1^\vee/I_1\cap\varpi I_1^\vee)=(f^0(0)+2,\psi^0(0)\epsilon^{f^0(0)+1}),
    \end{equation*}
    and thus the claim follows as in the above.
\end{proof}
\begin{corollary}\label{Tlessfullformula}
    For $\delta\in\Typ^0_r,$ if we write
    \begin{equation*}
        \mathrm{str}^\natural(\Delta_{1,r}^\natural(\delta))=\sum_{(f^0,\psi^0)\in\Typ^0}a(f^0,\psi^0)\cdot(f^0,\psi^0),
    \end{equation*}
    and denote $\lambda=\lambda_0(\delta),$ then we have
    \begin{equation*}
    \begin{split}
        T_{\le1,r}^{\mathrm{err}}(\delta)=\ &q^{r-\lambda}R\TypeC{1}{\psi}{\lambda}{\chi}\cdot\delta\\
        &+\sum_{(f^0,\psi^0)\in\Typ^0}a(f^0,\psi^0)\cdot\left(q^{r-1-f^0(0)}R\TypeC{1}{\psi}{f^0(0)+2}{\epsilon\psi^0(0)}-q^{r-\lambda}R\TypeC{1}{\psi}{\lambda}{\chi}\right)\cdot(f^0,\psi^0).
    \end{split}
    \end{equation*}
\end{corollary}
\begin{proof}
    This follows at once from the definition of $T_{\le1,r}^{\mathrm{err}},$ \Cref{comp3} and \cite[\citeinducedThm]{CoratoZanarella}.
\end{proof}

\begin{corollary}\label{Terrred}
    Given $\delta=\delta'\star\TypeCC{(0)^\lambda}{\chi}\in\Typ^0_r$ where $\delta'=\Type{e_i}{\psi_i}$ satisfies $e_i\ge1,$ we have that
    \begin{equation*}
        T_{\le1,r}^{\mathrm{err}}(\delta)\equiv q^{r-1+\lambda+\alpha}\frac{q-1}{q^\alpha+1}\mathrm{str}^\natural(\Delta_{1,r-\lambda}^\natural(\delta'))\star\TypeCC{(0)^\lambda}{\chi}+q^{r-\lambda}\delta'\star T_{\le1,\lambda}^{\mathrm{err}}\TypeCC{(0)^\lambda}{\chi}.
    \end{equation*}
\end{corollary}
\begin{proof}
    We may write
    \begin{equation*}
        \Delta_{1,r}^\natural(\delta)=q^\lambda\Delta_{1,r-\lambda}^\natural(\delta')\star\TypeCC{(0)^\lambda}{\chi}+\delta'\star\Delta_{1,\lambda}^\natural\TypeCC{(0)^\lambda}{\chi}.
    \end{equation*}
    Note that all terms in $\mathrm{str}^\natural(\Delta_{1,r-\lambda}^\natural(\delta'))$ still satisfy that $e_i\ge1.$ Hence, by \Cref{Tlessfullformula}, we have
    \begin{equation*}
    \begin{split}
        T_{\le1,r}^{\mathrm{err}}(\delta)=\ & q^{r-\lambda}R\TypeC{1}{\psi}{\lambda}{\chi}\cdot\delta+\left(q^{r-1}R\TypeC{1}{\psi}{\lambda+2}{\epsilon\chi}-q^{r}R\TypeC{1}{\psi}{\lambda}{\chi}\right)\cdot\mathrm{str}^\natural(\Delta_{1,r-\lambda}^\natural(\delta'))\star\TypeCC{(0)^\lambda}{\chi}\\
        &+q^{r-\lambda}\left(\mathrm{str}^\natural\left(\delta'\star T_{\le1,\lambda}^{\mathrm{err}}\TypeCC{(0)^\lambda}{\chi}\right)-R\TypeC{1}{\psi}{\lambda}{\chi}\cdot\delta\right).
    \end{split}
    \end{equation*}
    Now the claim follows since in both cases \Us and \Os we can compute that
    \begin{equation*}
        q^{r-1}R\TypeC{1}{\psi}{\lambda+2}{\epsilon\chi}-q^{r}R\TypeC{1}{\psi}{\lambda}{\chi}=q^{r-1+\lambda+\alpha}\frac{q-1}{q^\alpha+1}.\qedhere
    \end{equation*}
\end{proof}

\begin{proposition}\label{strred}
     Given $\delta=\delta'\star\TypeCC{(0)^\lambda}{\chi}\in\Typ$ where $\delta'=\Type{e_i}{\psi_i}$ satisfies $e_i\ge1,$ we have that $\mathrm{str}^\sharp(\Delta_{\le1,r}^\flat(\delta))-\mathrm{str}^\flat(\Delta_{\le1,r}^\flat(\delta))$ is equal to
     \begin{equation*}
         q^{r-\lambda+\alpha}(q^\alpha+1)R\TypeC{1}{\psi}{\lambda}{\chi}\cdot\delta'\star\left(\TypeCC{(0)^{\lambda}}{\chi}-\mathrm{str}^\natural\TypeC{(0)^{\lambda-1}}{\psi\chi}{2}{\psi}\right).
     \end{equation*}
\end{proposition}
\begin{remark}
    This proposition does not depend on knowing that $\Delta_{\le1,r}^\flat$ preserves $\Rel^\flat_r$ or $\Rel^\sharp_r.$ In fact, $\Delta_{\le1,r}^\flat$ \emph{does not} preserve $\Rel^\sharp_r.$
\end{remark}
\begin{proof}
    First note that if $\varepsilon\in\{-1,0,1\}^r$ is such that $e_1+2\varepsilon_1,\ldots,e_r+2\varepsilon_r\ge0,$ then
    \begin{equation*}
        \mathrm{str}^\sharp(t(2\varepsilon)(\delta))=\mathrm{str}^\natural(t(2\varepsilon)(\delta))=\mathrm{str}^\flat(t(2\varepsilon)(\delta)).
    \end{equation*}
    Similarly, since both $\Rel^\flat$ and $\Rel^\sharp$ contain the element $\Rel^\flat\TypeCC{1}{\chi}=\Rel^\sharp\TypeCC{1}{\chi}=\TypeCC{-1}{\chi}-\TypeCC{1}{\chi},$ we have that if $e_1+2\varepsilon_1,\ldots,e_r+2\varepsilon_r\ge-1,$ then
    \begin{equation*}
        \mathrm{str}^\sharp(t(2\varepsilon)(\delta))=\mathrm{str}^\flat(t(2\varepsilon)(\delta)).
    \end{equation*}
    That is, the proposition reduces to proving the claim for $\delta=\TypeCC{(0)^\lambda}{\chi},$ for which we have
    \begin{equation*}
        \mathrm{str}^\sharp(\Delta_{\le1,\lambda}^\flat(\delta))-\mathrm{str}^\flat(\Delta_{\le1,\lambda}^\flat(\delta))=\sum_{i=0}^{\lambda-1}q^{\lambda-i-1}\TypeCC{(0)^{r-i-1}}{\chi_{\le r-i-1}}\star\left(\mathrm{str}^\sharp(\delta^{(i)})-\mathrm{str}^\flat(\delta^{(i)})\right)
    \end{equation*}
    where $\delta^{(i)}\defeq\TypeCC{-2}{\chi_{r-i}}\star\TypeCC{(0)^{i}}{\chi_{>r-i}}.$
\begin{case}{\Us}
We have
\begin{equation*}
\delta^{(n)}=\typeC{-2}{(0)^n}\equiv(1-(-q_0)^n)\typeC{(0)^{n-1}}{(-1)^2}+(-q_0)^{n}\typeC{(0)^n}{-2}\mod\Rel^\natural.
\end{equation*}
Since
\begin{equation*}
\begin{split}
    \mathrm{str}^\sharp((-1))-\mathrm{str}^\flat((-1))&=0,\\
    \mathrm{str}^\sharp((-2))-\mathrm{str}^\flat((-2))&=(q+1)(0)-q(2)-(1-q_0)(0)-q_0(2)=q_0(q_0+1)((0)-(2)),
\end{split}
\end{equation*}
we conclude that
\begin{equation*}
    \mathrm{str}^\sharp(\delta^{(n)})-\mathrm{str}^\flat(\delta^{(n)})=q_0(q_0+1)(-q_0)^n((0)^n)\star\left((0)-(2)\right).
\end{equation*}
That is,
\begin{equation*}
\begin{split}
    \mathrm{str}^\sharp(\Delta_{\le1,\lambda}^\flat(\delta))-\mathrm{str}^\flat(\Delta_{\le1,\lambda}^\flat(\delta))&=q_0(q_0+1)\left(\sum_{i=0}^{\lambda-1}q^{\lambda-i-1}(-q_0)^i\right)((0)^{\lambda-1})\star\left((0)-(2)\right)\\
    &=(-q_0)^\lambda((-q_0)^\lambda-1)\cdot((0)^{\lambda-1})\star\left((0)-(2)\right).
\end{split}
\end{equation*}
Note that
\begin{equation*}
    q^\alpha(q^\alpha+1)R(1,\lambda)=q_0(q_0+1)\qbinom{\lambda}{1}{-q_0}=(-q_0)^\lambda((-q_0)^\lambda-1).
\end{equation*}
\end{case}

\begin{case}{\Os}
We have
\begin{equation*}
\begin{split}
    \mathrm{str}^\natural\TypeC{-2}{\chi_1}{(0)^n}{\chi_2}=\ &(1-(q\epsilon)^{\lfloor n/2\rfloor}\chi_2(-\epsilon\chi_1)^n)\TypeC{(0)^{n-1}}{\epsilon\chi_1\chi_2}{(-1)^2}{\epsilon}\\
    &+(q\epsilon)^{\lfloor n/2\rfloor}\chi_2(-\epsilon\chi_1)^n\left(\frac{1+q\epsilon}{2}\TypeC{(0)^n}{\chi_2}{-2}{\chi_1}+\frac{1-q\epsilon}{2}\TypeC{(0)^n}{(-1)^n\chi_2}{-2}{(-1)^n\chi_1}\right).
\end{split}
\end{equation*}
Since
\begin{equation*}
\begin{split}
    \mathrm{str}^\sharp\TypeCC{-1}{\chi}-\mathrm{str}^\flat\TypeCC{-1}{\chi}&=0,\\
    \mathrm{str}^\sharp\TypeCC{-2}{\chi}-\mathrm{str}^\flat\TypeCC{-2}{\chi}&=(1+\psi\chi)\TypeCC{0}{\chi}-\psi\chi\TypeCC{2}{\chi}-\TypeCC{2}{\chi}=(1+\psi\chi)\left(\TypeCC{0}{\chi}-\TypeCC{2}{\chi}\right),
\end{split}
\end{equation*}
we conclude that $\mathrm{str}^\sharp\TypeC{-2}{\chi_1}{(0)^n}{\chi_2}-\mathrm{str}^\flat\TypeC{-2}{\chi_1}{(0)^n}{\chi_2}$ is
\begin{equation*}
\begin{split}
    &(q\epsilon)^{\lfloor n/2\rfloor}\chi_2(-\epsilon\chi_1)^n\frac{1+q\epsilon}{2}(1+\psi\chi_1)\TypeCC{(0)^n}{\chi_2}\star\left(\TypeCC{0}{\chi_1}-\TypeCC{2}{\chi_1}\right)\\
    &\qquad+(q\epsilon)^{\lfloor n/2\rfloor}\chi_2(-\epsilon\chi_1)^n\frac{1-q\epsilon}{2}(1+\psi(-1)^n\chi_1)\TypeCC{(0)^n}{(-1)^n\chi_2}\star\left(\TypeCC{0}{(-1)^n\chi_1}-\TypeCC{2}{(-1)^n\chi_1}\right).
\end{split}
\end{equation*}
We can uniformly write this as
\begin{equation*}
    \mathrm{str}^\sharp\TypeC{-2}{\chi_1}{(0)^n}{\chi_2}-\mathrm{str}^\flat\TypeC{-2}{\chi_1}{(0)^n}{\chi_2}=c\TypeCC{n}{\chi_1,\chi_2}\cdot\left(\TypeCC{(0)^{n+1}}{\chi_1\chi_2}-\TypeC{(0)^n}{\psi\chi_1\chi_2}{2}{\psi}\right)
\end{equation*}
for
\begin{equation*}
    c\TypeCC{n}{\chi_1,\chi_2}\defeq(q\epsilon)^{\lfloor n/2\rfloor}\chi_2(-\epsilon\chi_1)^n
    (1+\psi\chi_1(q\epsilon)^{n\text{ odd}}).
\end{equation*}
This means that
\begin{equation*}
    \mathrm{str}^\sharp(\Delta_{\le1,\lambda}^\flat(\delta))-\mathrm{str}^\flat(\Delta_{\le1,\lambda}^\flat(\delta))=\left(\sum_{i=0}^{\lambda-1}q^{\lambda-i-1}c\TypeCC{i}{\chi_{\lambda-i},\chi_{>\lambda-i}}\right)\cdot\left(\TypeCC{(0)^{\lambda}}{\chi}-\TypeC{(0)^{\lambda-1}}{\psi\chi}{2}{\psi}\right).
\end{equation*}

Write
\begin{equation*}
    a(n)\defeq (q\epsilon)^{\lfloor n/2\rfloor}\chi_{>\lambda-i}(-\epsilon\chi_{\lambda-i})^n,\quad b(n)\defeq (q\epsilon)^{\lfloor(n+1)/2\rfloor}\psi\chi_{\ge \lambda-i}(-\epsilon\chi_{\lambda-i})^n,
\end{equation*}
so that $c\TypeCC{n}{\chi_{r-i},\chi_{>r-i}}=a(n)+b(n).$ Note that
\begin{equation*}
    q\cdot a(2i-1)+a(2i)=0,\quad q\cdot b(2i)+b(2i+1)=0,
\end{equation*}
and thus
\begin{equation*}
\begin{split}
    \sum_{i=0}^{\lambda-1}q^{\lambda-i-1}c\TypeCC{i}{\chi_{\lambda-i},\chi_{>\lambda-i}}&=q^{\lambda-1}a(0)+\begin{cases}
        b(\lambda-1) & \text{if }\lambda\text{ is odd,} \\
        a(\lambda-1) & \text{if }\lambda\text{ is even}
    \end{cases}\\
    &=q^{\lambda-1}+(q\epsilon)^{\lfloor(\lambda-1)/2\rfloor}\chi\psi(-\epsilon\psi)^{\lambda-1}.
\end{split}
\end{equation*}
Finally, note that
\begin{equation*}
\begin{split}
    q^\alpha(q^\alpha+1)R\TypeC{1}{\psi}{\lambda}{\chi}&=2R\TypeC{1}{\psi}{\lambda}{\chi}=q^{\lfloor(\lambda-1)/2\rfloor}\frac{\lfloor\lambda/2\rfloor_{q^2}}{\lfloor(\lambda-1)/2\rfloor_{q^2}}\frac{e(\lambda-1,\chi\psi)}{e(\lambda,\chi)}\\
    &=q^{\lfloor(\lambda-1)/2\rfloor}(q^{\lfloor\lambda/2\rfloor}-\epsilon^{\lfloor\lambda/2\rfloor}\chi(-\psi)^\lambda)\\
    &=q^{\lambda-1}+(q\epsilon)^{\lfloor(\lambda-1)/2\rfloor}\chi\psi(-\epsilon\psi)^{\lambda-1}
\end{split}
\end{equation*}
agrees with the above.\qedhere
\end{case}
\end{proof}

To finish the analysis of $T_{\le1,r}^{\sharp},$ we will assume the corresponding result in the \genf case.
\begin{hypothesis}\label{flathyp}
The induced map $\mathrm{str}^{\flat,*}\Delta_{\le1,r}^\flat\colon R[\Typ^0_r]\to R[\Typ^0_r]$ agrees with $T_{\le1,r}^\flat.$
\end{hypothesis}
\begin{remark}
    As remarked in the introduction \Cref{nonlinear}, the validity of this hypothesis will be established in \Cref{flatfinal} without relying on the results of this subsection.
\end{remark}

\begin{theorem}\label{genericsharp}
    Assume \Cref{flathyp}. Then the induced map $\mathrm{str}^{\sharp,*}\Delta_{\le1,r}^\sharp\colon R[\Typ^0_r]\to R[\Typ^0_r]$ agrees with $T_{\le 1,r}^{\sharp,*}.$
\end{theorem}
\begin{proof}
    Given $\delta=\Type{e_i}{\chi_i}\in\Typ$ with $e_1\ge\cdots\ge e_r\ge0,$ we need to prove that
    \begin{equation*}
        T_{\le1.r}^{\sharp,*}(\delta)\isequal\mathrm{str}^\sharp(\Delta_{\le 1,r}^\sharp(\delta)).
    \end{equation*}
    By \Cref{deltadif} and \Cref{tdif}, this is the same as proving
    \begin{equation*}
        T_{\le1,r}^{\flat,*}(\delta)+q^\alpha(q^\alpha+1)T_{\le1,r}^{\mathrm{err}}(\delta)\isequal\mathrm{str}^\sharp(\Delta_{\le1,r}^\flat(\delta))+(q-1)q^{r-1+2\alpha}\mathrm{str}^\natural(\Delta_{1,r}^{\natural}(\delta)).
    \end{equation*}
    By \Cref{flathyp} this is the same as proving
    \begin{equation*}
        q^\alpha(q^\alpha+1)T_{\le1,r}^{\mathrm{err}}(\delta)\isequal(q-1)q^{r-1+2\alpha}\mathrm{str}^\natural(\Delta_{1,r}^{\natural}(\delta))+\mathrm{str}^\sharp(\Delta_{\le1,r}^{\flat}(\delta))-\mathrm{str}^\flat(\Delta_{\le1,r}^{\flat}(\delta)).
    \end{equation*}

    Write $\delta=\delta'\star\TypeCC{(0)^\lambda}{\chi}$ for $\delta'$ satisfying $e_i\ge1.$ Then
    \begin{equation*}
        \Delta_{1,r}^{\natural}(\delta)=q^{\lambda}\Delta_{1,r-\lambda}^{\natural}(\delta')\star\TypeCC{(0)^\lambda}{\chi}+\delta'\star\Delta_{1,\lambda}^{\natural}\TypeCC{(0)^\lambda}{\chi}.
    \end{equation*}
    Together with \Cref{Terrred} and \Cref{strred}, the claim is reduced to the case $\delta=\TypeCC{(0)^\lambda}{\chi}.$ In this case, we can compute all the terms by \Cref{Tlessfullformula} and \Cref{strred}. Using that $q^{\lambda-1+\alpha}\frac{q-1}{q^\alpha+1}=q^{-1}R\TypeC{1}{\psi}{\lambda+2}{\epsilon\chi}-R\TypeC{1}{\psi}{\lambda}{\chi},$ we are reduced to verify that $R\TypeC{1}{\psi}{\lambda}{\chi}\mathrm{str}^\natural\TypeC{(0)^{\lambda-1}}{\psi\chi}{2}{\psi}$ is
    \begin{equation*}
    \begin{split}
        &S\TypeC{1}{}{\lambda}{\chi}\cdot\left(q^{-1}\cdot R\TypeC{1}{\psi}{\lambda+2}{\epsilon\chi}-q\cdot R\TypeC{1}{\psi}{\lambda}{\chi}\right)\cdot\TypeC{(1)^2}{\epsilon}{(0)^{\lambda-2}}{\epsilon\chi}\\
        +&\sum_{\chi'\in\mathrm{Sign}}R\TypeC{1}{\chi'}{\lambda}{\chi}\cdot\left(q^{-1}\cdot R\TypeC{1}{\psi}{\lambda+2}{\epsilon\chi}-R\TypeC{1}{\psi}{\lambda+1}{\epsilon\chi\chi'}\right)\TypeC{2}{\chi'}{(0)^{\lambda-1}}{\chi\chi'}.
    \end{split}
    \end{equation*}
    Collecting terms and with the following straightening relations modulo $\Rel^\natural_2,$
    \begin{equation*}
    \begin{array}{c|c}
    \Ugen&((0)^n,2)\equiv(1-(-q_0)^n)\cdot ((1)^2,(0)^{n-1})+(-q_0)^n\cdot(2,(0)^n)\\
    \hline
    \Ogen&\begin{array}{l}
    \TypeC{(0)^n}{\chi_1}{2}{\chi_2}\equiv(1-(q\epsilon)^{\lfloor n/2\rfloor}\chi_1(-\epsilon\chi_2)^n)\TypeC{(1)^2}{\epsilon}{(0)^{n-1}}{\epsilon\chi_1\chi_2}\\
    \quad\quad+(q\epsilon)^{\lfloor n/2\rfloor}\chi_1(-\epsilon\chi_2)^n
    \left(\frac{1+q\epsilon}{2}\TypeC{2}{\chi_2}{(0)^n}{\chi_1}+\frac{1-q\epsilon}{2}
    \TypeC{2}{(-1)^n\chi_2}{(0)^n}{(-1)^n\chi_1}\right)
    \end{array}
\end{array}
\end{equation*}
this becomes a routine computation. In the case \Ugen, we are verifying that
\begin{equation*}
\begin{split}
    (-q_0)^{\lambda-1}\qbinom{\lambda}{1}{-q_0}(1-(-q_0)^{\lambda-1})&\isequal(q_0+1)\qbinom{\lambda}{2}{-q_0}\left((-q_0)^{\lambda-1}\qbinom{\lambda+2}{1}{-q_0}-(-q_0)^{\lambda+1}\qbinom{\lambda}{1}{-q_0}\right),\\
    (-q_0)^{\lambda-1}&\isequal\left((-q_0)^{\lambda-1}\qbinom{\lambda+2}{1}{-q_0}-(-q_0)^{\lambda}\qbinom{\lambda+1}{1}{-q_0}\right).
\end{split}
\end{equation*}
In the case \Ogen, writing
\begin{equation*}
\begin{split}
    S\TypeC{1}{}{\lambda}{\chi}&=\frac{q^{\lambda-1}-1}{q-1}+\frac{1+(-1)^\lambda}{2}(q^{\lambda/2-1}\epsilon^{\lambda/2}\chi),\\
    R\TypeC{1}{\chi_1}{\lambda}{\chi_2}&=\frac{1}{2}\left(q^{\lambda-1}-(q\epsilon)^{\lfloor(\lambda-1)/2\rfloor}\epsilon\chi_2(-\epsilon\chi_1)^\lambda\right),
\end{split}
\end{equation*}
we are verifying that
\begin{equation*}
\begin{split}
    &\frac{1}{2}\left(q^{\lambda-1}-(q\epsilon)^{\lfloor(\lambda-1)/2\rfloor}\epsilon\chi(-\epsilon\psi)^\lambda\right)\cdot\left(1-(q\epsilon)^{\lfloor(\lambda-1)/2\rfloor}\chi\psi(-\epsilon\psi)^{\lambda-1}\right)\\
    &\qquad\isequal\left(\frac{q^{\lambda-1}-1}{q-1}+\frac{1+(-1)^\lambda}{2}(q^{\lambda/2-1}\epsilon^{\lambda/2}\chi)\right)\cdot\frac{(q\epsilon)^{\lfloor(\lambda-1)/2\rfloor}\epsilon\chi(-\epsilon\psi)^\lambda(q-1)}{2}
\end{split}
\end{equation*}
and, for the term with $\chi'=\pm\psi,$ that
\begin{equation*}
\begin{split}
    &(q\epsilon)^{\lfloor(\lambda-1)/2\rfloor}\psi\chi(-\epsilon\psi)^{\lambda-1}\cdot\begin{cases}\frac{1\pm q\epsilon}{2}&\text{if }\lambda\text{ is even,}\\\frac{1\pm1}{2}&\text{if }\lambda\text{ is odd}\end{cases}\\
    &\qquad\isequal\frac{\pm(q\epsilon)^{\lfloor\lambda/2\rfloor}\psi\chi(-\epsilon\psi)^{\lambda+1}-q^{-1}(q\epsilon)^{\lfloor(\lambda+1)/2\rfloor}\chi(-\epsilon\psi)^{\lambda+2}}{2}.\qedhere
\end{split}
\end{equation*}
\end{proof}

%% file: 4.6-Extension.tex
\begin{definition}
    For a $k\ge0,$ We say that $(f,\chi)\in\Typ_r$ is \emph{$k$-generic} if $f_i\ge f_{i+1}+k$ for $1\le i<r,$ and $f_r\ge k.$
\end{definition}
\begin{definition}
    For an operator $\Delta\colon R[\Typ_r]\to R[\Typ_r]$ given by a finite sum $\Delta=\sum_{\varepsilon}a_\varepsilon\cdot t(\varepsilon),$ we say that $T\colon R[\Typ^0_r]\to R[\Typ^0_r]$ is \emph{generically given by $\Delta$} if there exist $k\ge1$ such that for any $\delta$ which is $k$-generic, we have $T(\delta)=\Delta(\delta).$
\end{definition}

\begin{lemma}\label{extcommute}
    Suppose $T_1,T_2\colon R[\Typ^0_r]\to R[\Typ^0_r]$ are two operators which commute. Suppose that $T_1$ resp.\ $T_2$ are generically given by $\Delta_1$ resp.\ $\Delta_2.$ Then $\Delta_1,\Delta_2$ commute.
\end{lemma}
\begin{proof}
Let $s\ge1$ be such that both $\Delta_1$ and $\Delta_2$ only contain translations of magnitude smaller than $s.$ Let $k$ be such that $T_1,T_2$ agree with $\Delta_1,\Delta_2$ for $k$-generic $\delta.$

We choose $\delta\in\Typ^0_r$ such that its entire $s+1$-neighborhood is $k$-generic. Then we have
\begin{equation*}
    \Delta_1(\Delta_2(\delta))=T_1(\Delta(\delta))=T_1(T_2(\delta))=T_2(T_1(\delta))=T_2(\Delta_1(\delta))=\Delta_2(\Delta_1(\delta)).
\end{equation*}
Since both $\Delta_1,\Delta_2$ are a linear combination of translations, this implies that $\Delta_1,\Delta_2$ commute.
\end{proof}

We will repeatedly use the following lemma to perform the extension arguments.
\begin{lemma}\label{extlemma}
    Fix some $?\in\{\flat,\sharp\}.$ Assume that we have two operators $T_1,T_2\colon R[\Typ^0_r]\to R[\Typ^0_r]$ satisfying
    \begin{enumerate}
        \item $T_1,T_2$ commute,
        \item $T_1$ resp.\ $T_2$ is generically given by $\Delta_1$ resp.\ $\Delta_2,$
        \item both $\Delta_1,\Delta_2$ preserve $\Rel^?,$
        \item the induced map $\Delta_2\colon R[\Typ^0_r]\to R[\Typ^0_r]$ agrees with $T_2,$
        \item there exist $s\ge0$ such that if $\delta,\delta'\in\Typ^0_r$ are such that $\delta'$ is in the support of $T(\delta),$ then $\delta,\delta'$ have distance at most $s/2.$
    \end{enumerate}
    Denote $\Typ_{known}\defeq\{\delta\in\Typ^0_r\colon \mathrm{str}^?(\Delta_1(\delta))=T_1(\delta)\}.$ Suppose we are given $\delta\in\Typ^0_r$ and $Z\in R[\Typ^0_r]$ such that:
    \begin{enumerate}[label=(\alph*)]
        \item $Z\in R[\Typ_{known}],$
        \item there exist a non zero divisor $d\in R$ such that we have
        \begin{equation*}
            T_2(Z)=d\cdot \delta +\sum_{\delta'}a_{\delta'}\delta'
        \end{equation*}
        where for every $\delta'\neq\delta$ with $a_{\delta'}\neq0,$ we have either
        \begin{enumerate}[label=(b.\arabic*)]
            \item $\delta'\in \Typ_{known},$ or
            \item the distance between $\delta$ and $\delta'$ is at least $s.$
        \end{enumerate}
    \end{enumerate}
    Then we may conclude that $\delta\in\Typ_{known}.$
\end{lemma}
\begin{proof}
We have
\begin{equation*}
    \Delta_2(\Delta_1(Z))\stackrel{\text{(4)}}{\equiv} T_2(\Delta_1(Z))\stackrel{\text{(a)}}{\equiv}T_2(T_1(Z))\stackrel{\text{(1)}}{=}T_1(T_2(Z))\mod\Rel^?.
\end{equation*}
By (1), (2) and \Cref{extcommute}, we have that $\Delta_2$ and $\Delta_1$ commute, and thus
\begin{equation*}
    \Delta_1(T_2(Z))\stackrel{\text{(3), (4)}}{\equiv}\Delta_1(\Delta_2(Z))=\Delta_2(\Delta_1(Z))\equiv T_1(T_2(Z))\mod\Rel^?.
\end{equation*}
Thus by (b) and (5), we conclude that
\begin{equation*}
    \mathrm{str}(\Delta_1(d\cdot\delta))=T_1(d\cdot\delta).
\end{equation*}
Since $d$ is not a zero divisor, this then means that $\delta\in\Typ_{known}.$
\end{proof}

\subsubsection{Case $\flat$}
\begin{lemma}
    Let $\chi\in\mathrm{Sign}^r$ and consider a box $B\defeq[a_1,b_1]\times\cdots\times[a_r,b_r]$ with $a_i,b_i\in\Z$ and $b_i-a_1\in2\Z_{\ge1}.$ We think of the integral points of $B$ as elements of $\Typ_r$ with the signs $\chi.$ Then there exist $Z_B\in R[\Typ_r]$ such that i) $Z_B$ is supported in the interior of $B,$ ii) $\Delta^\intert_r(Z_B)$ is supported in $\{a_1,\frac{a_1+b_1}{2},b_1\}\times\cdots\times\{a_r,\frac{a_r+b_r}{2},b_r\},$ iii) for $v$ a vertex of $B,$ the coefficient of $\Delta^\intert_r(Z_B)$ at $v$ is $\pm q^{a_v}$ for some positive integer $a_v.$
\end{lemma}
\begin{proof}
We write $x_1,\ldots,x_{\cfactor r},y_1,\ldots,y_{\cfactor r}$ to be $(x_i,y_i)=(a_i,b_i-2)$ in the cases \Ugen, \Ogen, and $(x_{2i+1},y_{2i+1}),(x_{2i},y_{2i})=(a_i,b_i-2)$ in the case \SPgen.

We will only use that $\Delta^\intert_r$ is of the form $\Delta^\intert_r=\sum_{\varepsilon\in\{0,1\}^{\cfactor r}} q^{\sum_iw_i\varepsilon_i}t^2(\varepsilon)t(-1,\ldots,-1).$ Consider the operator
\begin{equation*}
    T_B\defeq\sum_{\substack{c_1,\ldots,c_{\cfactor r}\\ c_i\in[0,(y_i-x_i)/2]}}\prod_i(-q^{w_i})^{c_i}\cdot t^2(c)
\end{equation*}
Then we have
\begin{equation*}
\begin{split}
    \Delta^\intert_r T_B&=t(-1,\ldots,-1)\sum_{\varepsilon\in\{0,1\}^{\cfactor r}}q^{\sum_iw_i\varepsilon_i}\sum_{\substack{c_1,\ldots,c_{\cfactor r}\\ c_i\in[0,(y_i-x_i)/2]}}\prod_i(-q^{w_i})^{\frac{1}{2}(c_i-x_i-1)}\cdot t^2(c+\varepsilon)\\
    &=t(-1,\ldots,-1)\sum_{\substack{c_1,\ldots,c_{\cfactor}\\c_i\in[0,1+(y_i-x_i)/2]}}(-q)^{\sum w_ic_i}t^2(c)\cdot\sum_{\substack{\varepsilon\in\{0,1\}^{\cfactor r}\\
    c_i-\varepsilon_i\in[0,(y_i-x_i)/2]}}(-1)^{\Sigma(\varepsilon)}.
\end{split}
\end{equation*}
This simplifies to
\begin{equation*}
    \Delta^\intert_r T_B=t(-1,\ldots,-1)\sum_{\substack{c_1,\ldots,c_{\cfactor r}\\ c_i\in\{0,1+(y_i-x_i)/2\}}}(-q)^{\sum w_ic_i}(-1)^{\#\{c_i\neq0\}}t^2(c)
\end{equation*}
and thus we can take $Z_B=T_B\Type{a_i+1}{\chi_i}.$
\end{proof}

\begin{theorem}
    Assume $q$ is not a zero divisor on $R.$ Suppose that we have $T\colon R[\Typ^0_r]\to R[\Typ^0_r]$ satisfying:
    \begin{enumerate}
        \item There is $s\ge1$ such that $T(\delta)$ is supported between $\delta-s$ and $\delta+s.$
        \item $T$ is generically given by some $\Delta.$
        \item $T$ commutes with $T^\intert_r.$
        \item $\Delta$ preserves $\Rel^\flat_r.$
    \end{enumerate}
    Then the induced map $\mathrm{str}^{\flat,*}\Delta\colon R[\Typ^0_r]\to R[\Typ^0_r]$ agrees with $T.$
\end{theorem}
\begin{proof}
Note that by \Cref{intert}, we have
\begin{enumerate}\setcounter{enumi}{4}
    \item $\Delta^\intert_r$ preserves $\Rel^\flat_r,$ and the induced map $\Delta^\intert_r\colon R[\Typ^0_r]\to R[\Typ^0_r]$ agrees with $T^\intert_r.$
\end{enumerate}

Denote $\Typ_{known}\subseteq\Typ^0_r$ to be the subset
\begin{equation*}
    \Typ_{known}=\{\delta\in\Typ^0_r\colon T(\delta)\equiv \Delta(\delta)\mod\Rel^\flat_r\}.
\end{equation*}
(2) means that $\Typ_{known}$ contain all $k$-generic $\delta$ for some $k\ge1.$ We will improve this to $\Typ_{known}=\Typ^0_r$ by repeatedly using \Cref{extlemma} for $Z=\mathrm{str}^\flat(Z_B)$ for well-chosen boxes $B.$ Let $s$ be as in the lemma.

We will denote $\Typ_{n,j}\defeq\{(f,\chi)\in\Typ^0_r\colon f_i\ge f_{i+1}+n\text{ for all }j\le i\le r\},$ where if $i=r$ we take $f_{r+1}$ to mean $0.$ Note (2) exactly means $T_{k,0}\subseteq \Typ_{known},$ and note that $\Typ_{0,r}=\Typ^0_r.$ We will prove the following:
\begin{itemize}[label=(Step)]
    \item If $\Typ_{n,j}\subseteq\Typ_{known}$ for some $1\le j\le r,$ then there exist $n'$ (dependent on $n,j,s$) such that $\Typ_{n',j+1}\subseteq\Typ_{known}.$
\end{itemize}

We consider the intermediary sets $\Typ_{n',j+1}^{(m)}=\Typ_{n',j+1}\cap\{(f,\chi)\in\Typ^0\colon f_{j}\ge f_{j+1}+m\}.$ Of course, $\Typ_{n',j+1}^{(0)}=\Typ_{n',j+1},$ and $\Typ_{n,j+1}^{(n)}=\Typ_{n,j}.$ So now it suffices to prove
\begin{enumerate}[label=(Step')]
    \item If $\Typ_{n,j+1}^{(m)}\subseteq\Typ_{known}$ for some $m\ge1,$ then there exist $n'$ with $\Typ_{n',j+1}^{(m-1)}\subseteq\Typ_{known}.$
\end{enumerate}
In fact, we may take $n'=n+2s.$ Consider $\delta=(f,\chi)\in\Typ_{n+a,j+1}^{(m-1)}.$ We look at
\begin{equation*}
    Z\defeq \mathrm{str}^\flat(Z_{[f_1,f_1+2N],\ldots,[f_j,f_j+2N],[f_{j+1}-2s,f_{j+1}],\ldots,[f_{r}-2s,f_{r}]})
\end{equation*}
for $N=2sr.$ Note that the straightening relations can only happen in the indices $[1,j].$ Since the straightening does not decrease the minimum nor increase the maximum, we must have that if $(g,\chi)$ is in the support of $Z,$ then we have $g_j\ge f_j+1\ge (f_{j+1}+(m-1))+1=f_{j+1}+m\ge g_{j+1}+m,$ as well as, for $i\in[j+1,r],$ that $g_i\ge f_i-s\ge(f_{i+1}+n')-2s=f_{i+1}+n\ge g_{i+1}+n.$ Thus we have (a). 

As for $T^\intert_r(Z),$ note that if $g=(g_1,\ldots,g_r)$ is in its support and is not $f,$ then either i) $g_1=f_1,\ldots,g_j=f_j,$ in which case $\Sigma(g)\le \Sigma(f)-s,$ or ii) $\Sigma(g)\ge\Sigma(f)+N-2(r-j)s.$ By our choice of $N,$ in either case we have $\lvert\Sigma(g)-\Sigma(f)\rvert\ge s,$ and thus (b.1) is satisfied.
\end{proof}
\begin{corollary}\label{flatfinal}
    The induced map $\mathrm{str}^{\flat,*}\Delta_{\le k,r}^\flat\colon R[\Typ^0_r]\to R[\Typ^0_r]$ agrees with $T_{\le k,r}^\flat.$
\end{corollary}
\begin{proof}
    If $q$ is not a zero divisor of $R,$ this follows by the theorem above and \Cref{preserve} and \Cref{generic}.

    Note that the general case follows immediately from the fact that we can apply the theorem for $R=\Z[\frac{1}{\#\mathrm{Sign}}].$
\end{proof}

\subsubsection{Case $\sharp$}
\begin{theorem}
    Assume $q$ is not a zero divisor on $R.$ Suppose that we have $T\colon R[\Typ^0_r]\to R[\Typ^0_r]$ satisfying:
    \begin{enumerate}
        \item There is $s\ge1$ such that $T(\delta)$ is supported between $\delta-s$ and $\delta+s.$
        \item $T$ is generically given by some $\Delta.$
        \item $T$ commutes with $T_{\le 1,r}^\sharp.$
        \item $\Delta$ preserves $\Rel^\sharp_r.$
    \end{enumerate}
    Then the induced map $\mathrm{str}^{\sharp,*}\Delta\colon R[\Typ^0_r]\to R[\Typ^0_r]$ agrees with $T.$
\end{theorem}
\begin{proof}
By \Cref{flatfinal} and \Cref{genericsharp} we have
\begin{enumerate}\setcounter{enumi}{4}
    \item $\Delta_{\le1,r}^\sharp$ preserves $\Rel^\sharp_r,$ and the induced map $\Delta_{\le1,r}^\sharp\colon R[\Typ^0_r]\to R[\Typ^0_r]$ agrees with $T_{\le1,r}^\sharp.$
\end{enumerate}

Denote $\Typ_{known}\subseteq\Typ^0_r$ to be the subset
\begin{equation*}
    \Typ_{known}=\{\delta\in\Typ^0_r\colon T(\delta)\equiv \Delta(\delta)\mod\Rel_r^\sharp\}.
\end{equation*}
(2) means that $\Typ_{known}$ contain all $k$-generic $\delta$ for some $k\ge1.$ We will improve this to $\Typ_{known}=\Typ^0_r$ by repeatedly using \Cref{extlemma}. Let $s$ be as in the lemma.

We will denote $\Typ_{n,j}\defeq\{(f,\chi)\in\Typ^0_r\colon f_i\ge f_{i+1}+n\text{ for all }1\le i\le j\},$ where if $i=j=r$ we take $f_{r+1}$ to mean $0.$ Note (2) exactly means $T_{k,r}\subseteq \Typ_{known},$ and note that $\Typ_{0,1}=\Typ^0_r.$

We will prove the following:
\begin{itemize}[label=(Step)]
    \item If $\Typ_{n,j}\subseteq\Typ_{known}$ for some $1\le j\le r,$ then there exist $n'$ (dependent on $n,j,s$) such that $\Typ_{n',j-1}\subseteq\Typ_{known}.$
\end{itemize}

We consider the intermediary sets $\Typ_{n,j-1}^{(m)}=\Typ_{n,j}\cap\{(f,\chi)\colon f_j\ge f_{j+1}+m\}$ for $0\le m\le n.$ Note that $\Typ_{n,j-1}^{(n)}=\Typ_{n,j}$ and $\Typ_{n,j-1}^{(0)}=\Typ_{n,j-1}.$ So it suffices to prove
\begin{enumerate}[label=(Step')]
    \item If $\Typ_{n,j-1}^{(m)}\subseteq\Typ_{known}$ for some $m\ge1,$ then there exist $n'$ with $\Typ_{n',j-1}^{(m-1)}\subseteq\Typ_{known}.$
\end{enumerate}

First, for the case $j=r,$ we may take $n'=n+4.$ Note that if $(f,\chi)\in\Typ_{n+2,r-1}^{(m-1)},$ then we may simply take $Z=(f_1,\ldots,f_r+2)\in\Typ_{n,r-1}^{(m)}.$ Note that except for $(f,\chi),$ all the other terms in $T_{1}^{\sharp}(Z)$ also belong to $\Typ_{n,r-1}^{(m)}.$

Now assume $1\le j<r.$ Let $\delta=(f,\chi)\in\Typ_{n',j-1}^{(m-1)}$ for $n'=n+2s+2.$ For $l\ge0,$ denote
\begin{equation*}
\begin{split}
    e^{(l)}&\defeq(f_1,\ldots,f_{j-1},f_j+2(l+1),f_{j+1}+2l,f_{j+2},\ldots,f_r),\\
    f^{(l)}&\defeq(f_1,\ldots,f_{j-1},f_j+2l,f_{j+1}+2l,f_{j+2},\ldots,f_r).
\end{split}
\end{equation*}
Note that if $0\le l< s,$ we have $e^{(l)}\in\Typ_{n,j-1}^{(m)},$ and we also have
\begin{equation*}
    T_{\le1,r}^\sharp(\delta(e^{(l)},\chi))\equiv q^{j-1}\delta(f^{(l)},\chi)+q^{r-j-1}\delta(f^{(l+1)},\chi)\mod R[\Typ_{n,j-1}^{(m)}]\quad\text{for }1\le l\le s.
\end{equation*}

As for $T_{\le1,r}^\sharp\delta(e^{(0)},\chi),$ let $a\ge1$ be such that $f_{j+a}=f_{j+1}$ but $f_{j+a+1}<f_{j+1}.$ This is so that
\begin{equation*}
    \delta(f,\chi)=\Typeee{f_j}{\chi_j}{(f_{j+1})^a}{\psi_1}{f_{j+a+1}}{\chi_{j+a+1}}
\end{equation*}
for $\psi_1\defeq\prod_{j<i\le j+a}\chi_i.$ Then we have, modulo $R[\Typ_{n,j-1}^{(m)}],$
\begin{equation*}
\begin{split}
    T_{\le1,r}^\sharp(\delta(e^{(0)},\chi))\equiv\ & q^{j-1}\delta(f^{(0)},\chi)+\sum_{\psi_2\in\mathrm{Sign}}\beta_{\psi_2}q^{r-j-1}\Typeee{f_j+2}{\chi_j}{f_{j+1}+2}{\psi_2}{(f_{j+1})^{a-1}}{\psi_1\psi_2}
\end{split}
\end{equation*}
for some $\beta_{\psi_2}\in R$ depending on $(a,\psi_1,\psi_2).$ For example, in the case \Us, this is
\begin{equation*}
    \beta=\begin{cases}
    \#\{j+1\le i\le r\colon f_i=f_{j+1}\} & \text{if }f_{j+1}>0, \\
    2\#\{j+1\le i\le r\colon f_i=f_{j+1}\} & \text{if }f_{j+1}=0.
\end{cases}
\end{equation*}
This is more complicated in the case \Os, but we won't need to know explicitly what the $\beta_{\psi_2}$ are. In fact, we only need to know that
\begin{equation*}
    T_{\le1,r}^\sharp(\delta(e^{(0)},\chi))\equiv q^{j-1}\delta(f^{(0)},\chi)+\sum_{\chi'\in I}\beta_{\chi'}q^{r-j-1}\delta(f^{(1)},\chi')\mod R[\Typ_{n,j-1}^{(m)}]
\end{equation*}
for some finite collection $I$ of $\chi'\in\mathrm{Sign}^r$ and for some coefficients $\beta_{\chi'}\in R.$

We take
\begin{equation*}
    Z=a_0\cdot \delta(e^{(0)},\chi)+\sum_{\chi'\in I}\sum_{l=1}^{s-1}(-1)^la_{l,\chi'}\cdot \delta(e^{(l)},\chi')
\end{equation*}
where
\begin{equation*}
    a_0=q^{2(s-1)j},\quad a_{l,\chi'}=\beta_{\chi'}\cdot q^{2(s-1-l)j+lr}\quad\text{for }l\ge1,\ \chi'\in I.
\end{equation*}
This satisfies (a). By the computations above, note that the only possible terms in the support of $T_{\le1,r}^\sharp(Z)$ that do not satisfy (b.2) are the $\delta(f^{(l)},\chi')$ for $1\le l\le s$ and $\chi'\in I.$ Note that $\delta(f^{(s)},\chi')$ satisfy (b.1). Finally, for $1\le l < s,$ the coefficient of $\delta(f^{(l)},\chi')$ is
\begin{equation*}
    \begin{cases}q^{j-1}\cdot (-1)^la_{l,\chi'}+q^{r-j-1}\cdot (-1)^{l-1}a_{l-1,\chi'}&\text{for }2\le l<s,\\
    -q^{j-1}a_{1,\chi'}+\beta_{\chi'}q^{r-j-1}a_0&\text{for }l=1\end{cases}=0
\end{equation*}
and thus we can indeed apply \Cref{extlemma}.
\end{proof}

\begin{corollary}
    The induced map $\mathrm{str}^{\sharp,*}\Delta_{\le k,r}^\sharp\colon R[\Typ^0_r]\to R[\Typ^0_r]$ agrees with $T_{\le k,r}^\sharp.$
\end{corollary}
\begin{proof}
    If $q$ is not a zero divisor of $R,$ this follows by the theorem above and \Cref{preserve} and \Cref{generic}.

    Note that the general case follows immediately from the fact that we can apply the theorem for $R=\Z[\frac{1}{\#\mathrm{Sign}}].$
\end{proof}

%% file: 4.7-Satake.tex
Assume that $R$ contains a square root of $q,$ so we may consider the Satake transform
\begin{equation*}
    \mathrm{Sat}\colon\mathcal{H}(G^?,K^?)\to R[\boldsymbol{\mu}_1,\ldots,\boldsymbol{\mu}_{\gamma r}]^{\mathrm{sym}}
\end{equation*}
as defined in \Cref{Satdef}.

\begin{theorem}
    For $0\le k\le \cfactor r,$ consider the elements $S_{k,r}^\flat\in\mathcal{H}(G^\flat,K^\flat)$ and $S_{k,r}^\sharp\in\mathcal{H}(G^\sharp,K^\sharp)$ described by
    \begin{equation*}
        S_k^{\flat,*}\defeq\sum_{\substack{\varepsilon\in\{-1,0,1\}^{\cfactor r}\\\lambda_0(\varepsilon)=\cfactor r-k}}q^{\widetilde{\inv}(\varepsilon)}\frac{\Dfactor(\lambda_1(\varepsilon))\Dfactor(\cfactor r-\lambda_{-1}(\varepsilon))}{\Dfactor(\lambda_0(\varepsilon))}t^2(\varepsilon)
    \end{equation*}
    and
    \begin{equation*}
        S_k^{\sharp,*}\defeq\sum_{\substack{\varepsilon\in\{-1,0,1\}^{\cfactor r}\\\lambda_0(\varepsilon)=\cfactor r-k}}q^{\widetilde{\inv}(\varepsilon)+\lambda_1(\varepsilon)}\frac{\Dfactor(\lambda_1(\varepsilon))\Dfactor(\cfactor r-\lambda_{-1}(\varepsilon))}{\Dfactor(\lambda_0(\varepsilon))}t^2(\varepsilon).
    \end{equation*}
    Then
    \begin{equation*}
        \mathrm{Sat}(S_{k,r}^\flat)=\frac{\Dfactor(\cfactor r)}{\Dfactor(\cfactor r-k)}s_k(\boldsymbol{\mu})\quad\text{and}\quad \mathrm{Sat}(S_{k,r}^\sharp)=q^{\frac{k}{2}}\frac{\Dfactor(\cfactor r)}{\Dfactor(\cfactor r-k)}s_k(\boldsymbol{\mu}).
    \end{equation*}
\end{theorem}
\begin{proof}
Consider the functions $A,B,C,D\in I(\Z_{\ge0},R)$ given by
\begin{center}
\begin{tabular}{c|c|c|c|c}
    &$A(x,y)$&$B(x,y)$&$C(x,y)$&$D(x,y)$\\
    \hline
    \SPf&$(-q;q)_{y-x}\qbinom{y}{x}{q^2}$&$W_{0}(x,y)q^{\frac{1}{2}(y^2+y-x^2-x)}$&$\qbinom{y}{\lfloor\frac{y-x}{2}\rfloor}{q^2}$&$q^{\frac{1}{2}(y^2+y-x^2+x)}\binom{y}{\lfloor\frac{y-x}{2}\rfloor}$\\
    \hline
    \Uf&$(-q_0;q)_{y-x}\qbinom{2y}{2x}{-q_0}$&$W_{0}(x,y)q_0^{y^2-x^2}$&$\qbinom{2y}{y-x}{-q_0}$&$q_0^{y^2-x^2}\binom{y}{\frac{y-x}{2}}$\\
    \hline
    \Us&$(-q_0;q)_{y-x}\qbinom{2y+1}{2x+1}{-q_0}$&$W_{1/2}(x,y)q_0^{y^2+y-x^2-x}$&$\qbinom{2y+1}{y-x}{-q_0}$&$q_0^{y^2+y-x^2-x}\binom{y}{\lfloor\frac{y-x}{2}\rfloor}$\\
    \hline
    \Of&$(-q;q)_{y-x}\qbinom{y}{x}{q^2}\frac{q^x+1}{q^y+1}$&$W_{0}(x,y)q^{\frac{1}{2}(y^2-y-x^2+x)}$&$\qbinom{y}{\frac{y-x}{2}}{q^2}\frac{q^x+1}{q^y+1}$&$q^{\frac{1}{2}(y^2-y-x^2+x)}\binom{y}{\frac{y-x}{2}}$\\
    \hline
    \Os&$(-q;q)_{y-x}\qbinom{y}{x}{q^2}$&$W_{1/2}(x,y)q^{\frac{1}{2}(y^2-x^2)}$&$\qbinom{y}{\frac{y-x}{2}}{q^2}$&$q^{\frac{1}{2}(y^2-x^2)}\binom{y}{\frac{y-x}{2}}$
\end{tabular}
\end{center}
Since $A(x,x)=B(x,x)=C(x,x)=D(x,x)=1,$ they are all invertible by \Cref{circledastinv}.

By the definitions of $T_{\le i,r}^?$ and $\Delta_{\le i,r}^?,$ we have
\begin{equation*}
    T_{\le\cfactor r-k,r}^?=\sum_{i=k}^{\cfactor r}A(k,i)\cdot T_{\cfactor r-i,r}^?\quad\text{and}\quad T_{\le \cfactor r-k,r}^?=\sum_{i=k}^{\cfactor r}B(k,i)\cdot S_{\cfactor r-i,r}^?.
\end{equation*}
Now if we let $\beta=0$ in cases \genf and $\beta=1/2$ in cases \gens, \Cref{LKform} says that
\begin{equation*}
    \sum_{i=k}^{\cfactor r}C(k,i)\cdot \mathrm{Sat}(T_{\cfactor r-i,r}^?)=\sum_{i=k}^{\cfactor r}D(k,i)\cdot q^{(\gamma r-i)\beta}\frac{\Dfactor(\cfactor r)}{\Dfactor(i)}s_{\cfactor r-i}(\boldsymbol{\mu}).
\end{equation*}
Hence the claim is reduced to verifying the following formal implication: for any two functions $T,S$ we have
\begin{equation*}
    A\circledast T= B\circledast S\implies C\circledast T=D\circledast S.
\end{equation*}
It remains to verify the following identity in $I(\Z_{\ge0},R)$:
\begin{equation*}
    C\circledast A^\circleddash\circledast B\isequal D.
\end{equation*}
This follows from \Cref{circledast1} and \Cref{circledast2}.
\end{proof}

%% file: 5-Consequences.tex
\section{Applications}\label{ResultsSection}
\subsection{Explicit bases}
Let $\beta,r\ge1$ be integers.

\begin{definition}
    We consider the following total order: For $e,f\in\Z^r,$ we say $e<f$ if either i) $\Sigma(e)<\Sigma(f)$ or ii) $\Sigma(e)=\Sigma(f)$ and $e$ is lexicographically smaller than $f.$
\end{definition}
We denote
\index{Typ0r<e@$\Typ_r^{0,<e}$}
\begin{equation*}
    \Typ^{0,\le e}_r\defeq\{\delta(f,\chi)\in\Typ^0_r\colon f\le e\},\quad \Typ^{0,< e}_r\defeq\{\delta(f,\chi)\in\Typ^0_r\colon f< e\}.
\end{equation*}

Suppose that we have commuting operators $S_0,\ldots,S_r\colon R[\Typ^0_r]\to R[\Typ^0_r]$ such that their duals are of the form
\begin{equation*}
    \sum_{k=0}^rS_k^*x^k\equiv\sum_{\varepsilon\in\{-1,0,1\}^r}a(\varepsilon)x^{\lambda_1(\varepsilon)+\lambda_{-1}(\varepsilon)}t(\beta\cdot\varepsilon)\mod\Rel^?_r
\end{equation*}
for some $?\in\{\flat,\sharp\}$ and certain coefficients $a(\varepsilon)\in R.$ We assume that $a(\underbrace{-1,\ldots,-1}_{k},0,\ldots,0)=1$ for all $0\le k\le r.$

\begin{proposition}\label{largestInSk}
Let $f=(f_1,\ldots,f_r)$ with $f_1\ge\cdots\ge f_r\ge0.$ Denote $s_k(f)\defeq f+\beta\cdot(\underbrace{1,\ldots,1}_{k},0,\ldots,0).$ Assume $\delta_s(f)\not\in\Rel^?_r.$ Then we have that
\begin{equation*}
    S_k(\delta_s(f))\equiv\delta_s(s_k(f))\mod R[\Typ^{0,<s_k(f)}_r].
\end{equation*}
\end{proposition}
\begin{proof}
Consider $\delta_s(e)\in\mathrm{Supp}(S_k(\delta_{s'}(f))).$ This is the same as having $\delta_{s'}(f)\in\mathrm{Supp}(S_k^*(\delta_s(e))).$ In the expression for $S_k^*$ above, the terms $\delta_s(e+\beta\cdot\varepsilon)$ get straightened to terms that satisfy either i) $\Sigma(f)>\Sigma(e)-k\beta$ (if $\varepsilon$ contains a $1$ or if a relation in $\Rel_1^?$ was used) or ii) such that $\Sigma(f)=\Sigma(e)=k\beta$ and $f\in[e_1-\beta,e_1]\times\cdots\times[e_r-\beta,e_r].$ Switching this around, we get that either i) $\Sigma(e)<\Sigma(f)+k\beta$ or ii)
\begin{equation*}
    e\in[f_1,f_1+\beta]\times\cdots[f_r,f_r+\beta],\quad \Sigma(e)=\Sigma(f)+k\beta.
\end{equation*}
Now it is easy to see that the point $e$ satisfying the above which is maximal with respect to $\le$ is simply $e=s_k(f).$ Now the claim follows from the assumption that $a(\underbrace{-1,\ldots,-1}_{k},0,\ldots,0)=1.$
\end{proof}

\begin{theorem}\label{basisTheorem}
Consider the set
\begin{equation*}
    \mathcal{B}\defeq \left\{\delta_s(e)\colon s\colon\mathrm{Sign}^r\to\{\pm1\},\ (e_{i}-e_{i+1}-\left(1\text{ if }s_i\neq s_{i+1}\right))\in\left[0,\beta-1\right],\ 0\le e_r<\beta\right\}.
\end{equation*}
Under this basis, $R[\Typ_r^0]$ is a free $R[S_1,\ldots,S_r]$-module.
\end{theorem}
\begin{proof}
By \Cref{largestInSk}, for any $a_1,\ldots,a_{r}\ge0$ we have
\begin{equation*}
    S_1^{a_1}\cdots S_r^{a_r}(\delta_s(f))\equiv \delta_s(e)\mod\Z[\Typ^{0,\prec e}_r]
\end{equation*}
where $e_i=f_i+\beta\cdot(a_i+\cdots+a_r).$ Conversely, given $e\in\Typ^0$ and $\delta_s(e)$ nonzero, there is a unique $\delta_s(f)\in\mathcal{B}$ and integers $a_1,\ldots,a_r$ as above with $S_1^{a_1}\cdots S_r^{a_r}(\delta_s(f))\equiv\delta_s(e)\mod\Z[\Typ^{\prec e}].$ Since the nonzero $\delta_s(e)$ for $e\in\Typ^0_r$ form an $R$-basis of $R[\Typ_0^r],$ this implies that $\mathcal{B}$ is a basis of $R[\Typ_0^r]$ as a $R[S_1,\ldots,S_r]$-module.
\end{proof}

\begin{corollary}\label{rankThm}
    In the cases \Uf, \Us, \Of, \Os, we have that
    \begin{equation*}
        \mathcal{B}\defeq \left\{\delta_s(e)\colon s\colon\mathrm{Sign}^r\to\{\pm1\},\ (e_{i}-e_{i+1}-\left(1\text{ if }s_i\neq s_{i+1}\right))\in\{0,1\},\ e_r\in\{0,1\}\right\}.
    \end{equation*}
    is a basis of the $\mathcal{H}(G^?,K^?)$-module $\mathcal{S}(X^?/K^?).$ In particular, $\mathcal{S}(X^?/K^?)$ is free of rank $2^r$ in the cases \Uf, \Us and free of rank $4^r$ in the cases \Of, \Os.
\end{corollary}

\subsection{Case \Ugen} Assume we are in the case \Ugen for this subsection. We will give another treatment of \cite[Theorem 4(1)]{uHflat} and \cite[Theorem 4(1)]{uHsharp}.
\begin{definition}\index{1Deltahalfflat@$\Delta_{k,r}^{1/2,\flat}$}\index{1Deltahalfsharp@$\Delta_{k,r}^{1/2,\sharp}$}
For $0\le k\le r,$ we consider the operators $\Delta_{k,r}^{1/2,\flat},\Delta_{k,r}^{1/2,\sharp}\colon R[\Typ_r]\to R[\Typ_r]$ given by
\begin{equation*}
\begin{split}
    \Delta_{k,r}^{1/2,\flat}&\defeq\sum_{\substack{\varepsilon\in\{-1,0,1\}^r\\\lambda_0(\varepsilon)=r-k}}(-1)^{\lambda_1(\varepsilon)}(-q_0)^{\lambda_1(\varepsilon)(\lambda_0(\varepsilon)+\lambda_1(\varepsilon))+\widetilde{\inv}(\varepsilon)}t(\varepsilon)\\
    \Delta_{k,r}^{1/2,\sharp}&\defeq\sum_{\substack{\varepsilon\in\{-1,0,1\}^r\\\lambda_0(\varepsilon)=r-k}}(-1)^{\lambda_1(\varepsilon)}(-q_0)^{\lambda_1(\varepsilon)(\lambda_0(\varepsilon)+\lambda_1(\varepsilon)+1)+\widetilde{\inv}(\varepsilon)}t(\varepsilon).
\end{split}
\end{equation*}
\end{definition}
\begin{remark}
    Note that when $\lambda_0(\varepsilon)=r-k,$
    \begin{equation*}
        \frac{\lambda_1(\varepsilon)^2+(r-\lambda_{-1}(\varepsilon))^2-(r-k)^2}{2}=\lambda_1(\varepsilon)(\lambda_0(\varepsilon)+\lambda_1(\varepsilon)).
    \end{equation*}
\end{remark}
\begin{proposition}\label{halfpreservesu}
$\Delta_{k,r}^{1/2,\flat}$ preserves $\Rel^\flat_r,$ and $\Delta_{k,r}^{1/2,\sharp}$ preserves $\Rel^\sharp_r.$
\end{proposition}
\begin{proof}
To see that they preserve $\Rel^\natural_r,$ it reduces to the case $r=2.$ For that, the case $k=1$ follows from \cite[\citepreserveshalf]{CoratoZanarella}. The case $k=2$ follows from \cite[\citeDeltapreserves]{CoratoZanarella}.

As for $\Rel_1^\flat$ and $\Rel_1^\sharp,$ note that if we denote $\varepsilon_1=\varepsilon\star(1)$ and $\varepsilon_2=\varepsilon\star(-1),$ then
\begin{equation*}
\begin{split}
    &\lambda_1(\varepsilon_1)(\lambda_0(\varepsilon_1)+\lambda_1(\varepsilon_1))+\widetilde{\inv}(\varepsilon_1)\\
    &\qquad=(\lambda_1(\varepsilon_2)+1)(\lambda_0(\varepsilon_2)+\lambda_1(\varepsilon_2)+1)+\widetilde{\inv}(\varepsilon_2)-\lambda_0(\varepsilon_2)-2\lambda_1(\varepsilon_2)\\
    &\qquad=\lambda_1(\varepsilon_2)(\lambda_0(\varepsilon_2)+\lambda_1(\varepsilon_2))+\widetilde{\inv}(\varepsilon_2)+1.
\end{split}
\end{equation*}
Thus the claim is reduced to the case $r=k=1.$ This follows from \Cref{preservesfs} as we have $\Delta_{1,1}^{1/2,\flat}=t(-1)+q_0\cdot t(1)$ and $\Delta_{1,1}^{1/2,\sharp}=t(-1)-q\cdot t(1).$
\end{proof}

\begin{definition}\index{Shalfflat@$S_{k,r}^{1/2,\flat}$}\index{Shalfsharp@$S_{k,r}^{1/2,\sharp}$}
For $0\le k\le r$ and $?\in\{\flat,\sharp\},$ denote $S^{1/2,?}_{k,r}\colon R[\Typ^0]\to R[\Typ^0]$ to be the adjoint of $\Delta_{k,r}^{1/2,?}.$ We let $\mathcal{H}^{1/2,?}\defeq R[S^{1/2,?}_{1,r},\ldots,S^{1/2,?}_{r,r}].$
\end{definition}
\begin{corollary}
    For $?\in\{\flat,\sharp\},$ we have that $R[\Typ^0_r]$ is a free $\mathcal{H}^{1/2,?}$-module of rank $1$ generated by $\delta(0,\ldots,0).$
\end{corollary}
\begin{proof}
    This follows from \Cref{basisTheorem}.
\end{proof}

\begin{theorem}[Spherical transform for \Uf and \Us]\label{sphtransu}
    Assume that $R$ contains a square root of $-q_0.$ Define $\mathrm{Sat}^{1/2,?}\colon\mathcal{H}^{1/2,?}\to R[\boldsymbol{\nu}_1,\ldots,\boldsymbol{\nu}_r]^{\mathrm{sym}}$ by
    \begin{equation*}
    \begin{split}
        \mathrm{Sat}^{1/2}(S^{1/2,\flat}_{k,r})&=(-q_0)^{\frac{r^2-(r-k)^2}{2}}s_k(\boldsymbol{\nu})\\
        \mathrm{Sat}^{1/2}(S^{1/2,\sharp}_{k,r})&=(-q_0)^{\frac{r^2+r-(r-k)^2-(r-k)}{2}}s_k(\boldsymbol{\nu}).
    \end{split}
    \end{equation*}
    Then we have the commutative diagrams
    \begin{equation*}
        \begin{tikzcd}
        \mathcal{H}(G^\flat,K^\flat)\arrow[d,hook]\arrow[r,"\mathrm{Sat}"]& R[\boldsymbol{\mu}_1,\ldots,\boldsymbol{\mu}_r]^{\mathrm{sym}}\arrow[d]\\
        \mathcal{H}^{1/2,\flat}\arrow[r,"\mathrm{Sat}^{1/2,\flat}"]& R[\boldsymbol{\nu}_1,\ldots,\boldsymbol{\nu}_r]^{\mathrm{sym}}
        \end{tikzcd}\quad
        \begin{tikzcd}
        \mathcal{H}(G^\sharp,K^\sharp)\arrow[d,hook]\arrow[r,"\mathrm{Sat}"]&R[\boldsymbol{\mu}_1,\ldots,\boldsymbol{\mu}_r]^{\mathrm{sym}}\arrow[d]\\
        \mathcal{H}^{1/2,\sharp}\arrow[r,"\mathrm{Sat}^{1/2,\sharp}"]&R[\boldsymbol{\nu}_1,\ldots,\boldsymbol{\nu}_r]^{\mathrm{sym}}
        \end{tikzcd}
    \end{equation*}
    where the right vertical maps are induced by $\boldsymbol{\mu}_k\mapsto -\boldsymbol{\nu}_k^2-2$ resp.\ $\boldsymbol{\mu}_k\mapsto \boldsymbol{\nu}_k^2+2.$
\end{theorem}
\begin{proof}
Let $\beta=\begin{cases}
    1 & \text{if }?=\sharp, \\
    0 & \text{if }?=\flat.
\end{cases}$
Consider
\begin{equation*}
\begin{split}
    S_r^?(x,y)&\defeq\sum_{k}q_0^{-k(2r+\beta-k)}S_{k,r}^{?}(-x)^ky^{r-k},\\
    S_r^{1/2,?}(x,y)&\defeq\sum_{k}(-q_0)^{-\frac{1}{2}k(2r+\beta-k)}S_{k,r}^{1/2,?}(-x)^ky^{r-k}.
\end{split}
\end{equation*}
Then
\begin{equation*}
    \mathrm{Sat}^?(S_r^?(x,y))=\prod_i(y-x\boldsymbol{\mu}_i),\quad\mathrm{Sat}^{1/2,?}(S_r^{1/2,?}(x,y))=\prod_i(y-x\boldsymbol{\nu}_i).
\end{equation*}
So we need to prove that
\begin{equation*}
    S_r^{1/2,?}(x,y)\cdot S_r^{1/2,?}(-x,y)\isequal S_r^?((-1)^{\beta+1}x^2,y^2+2x^2).
\end{equation*}
We do this by induction on $r.$ For this, we have
\begin{equation*}
    S_r^{1/2,?,*}(x,y)=\sum_{\varepsilon\in\{-1,0,1\}^r}(-1)^{\lambda_1(\varepsilon)}(-q_0)^{w(\varepsilon)}t(\varepsilon)(-x)^{\lambda_1(\varepsilon)+\lambda_{-1}(\varepsilon)}y^{\lambda_0(\varepsilon)}
\end{equation*}
where
\begin{equation*}
    w(\varepsilon)=\frac{1}{2}(\lambda_1(\varepsilon)^2-\lambda_{-1}(\varepsilon)^2)-\lambda_{-1}(\varepsilon)(\lambda_0(\varepsilon)+\lambda_1(\varepsilon))+\frac{1}{2}\beta(\lambda_1(\varepsilon)-\lambda_{-1}(\varepsilon))+\widetilde{\inv}(\varepsilon).
\end{equation*}
Now we note that if $\varepsilon\in\{-1,0,1\}^r,$
\begin{equation*}
\begin{split}
    w(1,\varepsilon)-w(\varepsilon)&=\lambda_1(\varepsilon)+\frac{1}{2}-\lambda_{-1}(\varepsilon)+\frac{1}{2}\beta+\lambda_0(\varepsilon)+2\lambda_{-1}(\varepsilon)=r+\frac{1}{2}(\beta+1),\\
    w(0,\varepsilon)-w(\varepsilon)&=-\lambda_{-1}(\varepsilon)+\lambda_{-1}(\varepsilon)=0,\\
    w(-1,\varepsilon)-w(\varepsilon)&=-\lambda_{-1}(\varepsilon)-\frac{1}{2}-\lambda_0(\varepsilon)-\lambda_1(\varepsilon)-\frac{1}{2}\beta=-r-\frac{1}{2}(\beta+1).
\end{split}
\end{equation*}
Hence
\begin{equation*}
    S_{r+1}^{1/2,?,*}(x,y)=A(x,y)\star S_r^{1/2,?,*}(x,y).
\end{equation*}
where
\begin{equation*}
A(x,y)\defeq(-q_0)^{r+\frac{1}{2}(\beta+1)}x\cdot t(1)+y\cdot t(0)-(-q_0)^{-r-\frac{1}{2}(\beta+1)}x\cdot t(-1).
\end{equation*}
Similarly,
\begin{equation*}
    S_{r+1}^{?,*}(x,y)=B(x,y)\star S_r^{?,*}(x,y).
\end{equation*}
where
\begin{equation*}
    B(x,y)\defeq-q_0^{2r+\beta+1}x\cdot t(2)+y\cdot t(0)-q_0^{-2r-\beta-1}x\cdot t(-2).
\end{equation*}
And now the claim follows since one can check that
\begin{equation*}
    A(x,y)A(-x,y)=B((-1)^{\beta+1}x^2,y^2+2x^2).\qedhere
\end{equation*}
\end{proof}

\begin{remark}
    Similar to \cite[\citeuHsphremark]{CoratoZanarella}, we don't know a priori how the above maps $\mathrm{Sat}^{1/2,\flat}$ and $\mathrm{Sat}^{1/2,\sharp},$ in the case $R=\C,$ compare to the Hironaka--Komori spherical transforms of \cite{uHflat,uHsharp}, and it would be interesting to compare them.
\end{remark}

\subsection{Case \SPgen}
\begin{definition}\index{1Deltahalfflat@$\Delta_{k,r}^{1/2,\flat}$}
For $0\le k\le r,$ we consider the operators $\Delta_{k,r}^{1/2,\flat}\colon R[\Typ]\to R[\Typ]$ given by
\begin{equation*}
\begin{split}
    \Delta_{k,r}^{1/2,\flat}&\defeq\sum_{\substack{\varepsilon\in\{-1,0,1\}^r\\\lambda_0(\varepsilon)=r-k}}q^{2\lambda_1(\varepsilon)(\lambda_0(\varepsilon)+\lambda_1(\varepsilon))+2\widetilde{\inv}(\varepsilon)+\lambda_1(\varepsilon)}t(\varepsilon).
\end{split}
\end{equation*}
\end{definition}
\begin{proposition}
$\Delta_{k,r}^{1/2,\flat}$ preserves $\Rel^\flat_r.$
\end{proposition}
\begin{proof}
To see that they preserve $\Rel^\natural_r,$ it reduces to the case $r=2.$ For that, the case $k=1$ follows from \cite[\citepreserveshalf]{CoratoZanarella}. The case $k=2$ follows from \cite[\citeDeltapreserves]{CoratoZanarella}.

As for $\Rel_1^\flat,$ note that if we denote $\varepsilon_1=\varepsilon\star(1)$ and $\varepsilon_2=\varepsilon\star(-1),$ then
\begin{equation*}
\begin{split}
    &2\lambda_1(\varepsilon_1)(\lambda_0(\varepsilon_1)+\lambda_1(\varepsilon_1))+2\widetilde{\inv}(\varepsilon_1)+\lambda_1(\varepsilon_1)\\
    =\ &2(\lambda_1(\varepsilon_2)+1)(\lambda_0(\varepsilon_2)+\lambda_1(\varepsilon_2)+1)+2\widetilde{\inv}(\varepsilon_2)-2\lambda_0(\varepsilon_2)-4\lambda_1(\varepsilon_2)+\lambda_1(\varepsilon_2)+1\\
    =\ &2\lambda_1(\varepsilon_2)(\lambda_0(\varepsilon_2)+\lambda_1(\varepsilon_2))+2\widetilde{\inv}(\varepsilon_2)+\lambda_1(\varepsilon_2)+3.
\end{split}
\end{equation*}
Thus the claim is reduced to the case $r=k=1.$ This then follows from \Cref{preservesfs}.
\end{proof}

\begin{definition}\index{Shalfflat@$S_{k,r}^{1/2,\flat}$}
For $0\le k\le r,$  denote $S^{1/2,\flat}_{k,r}\colon R[\Typ^0]\to R[\Typ^0_r]$ to be the adjoint of $\Delta_{k,r}^{1/2,\flat}.$ We let $\mathcal{H}^{1/2,\flat}\defeq R[S^{1/2,\flat}_{1,r},\ldots,S^{1/2,\flat}_{r,r}].$
\end{definition}
\begin{corollary}
    $R[\Typ^0_r]$ is a free $\mathcal{H}^{1/2,\flat}$-module of rank $1$ generated by $\delta(0,\ldots,0).$
\end{corollary}
\begin{proof}
    This follows from \Cref{basisTheorem}.
\end{proof}

\begin{theorem}[Spherical transform for \SPf]\label{sphtranssp}
    Assume that $R$ contains a square root of $q.$ Define $\mathrm{Sat}^{1/2,\flat}\colon\mathcal{H}^{1/2,\flat}\to R[\boldsymbol{\nu}_1,\ldots,\boldsymbol{\nu}_r]^{\mathrm{sym}}$ by
    \begin{equation*}
        \mathrm{Sat}^{1/2,\flat}(S^{1/2,\flat}_{k,r})=q^{r^2-(r-k)^2+\frac{r-(r-k)}{2}}s_k(\boldsymbol{\nu}).
    \end{equation*}
    Then we have the commutative diagram
    \begin{equation*}
        \begin{tikzcd}
        \mathcal{H}(G^\flat,K^\flat)\arrow[d,twoheadrightarrow]\arrow[r,"\mathrm{Sat}"]&R[\boldsymbol{\mu}_1,\ldots,\boldsymbol{\mu}_{2r}]^{\mathrm{sym}}\arrow[d]\\
        \mathcal{H}^{1/2,\flat}\arrow[r,"\mathrm{Sat}^{1/2,\flat}"]&R[\boldsymbol{\nu}_1,\ldots,\boldsymbol{\nu}_r]^{\mathrm{sym}}
        \end{tikzcd}
    \end{equation*}
    where the right vertical map is induced by: if we write $\boldsymbol{\nu}_i=\boldsymbol{\alpha}_i+1/\boldsymbol{\alpha}_i,$ then $\{\boldsymbol{\mu}_{2i-1},\boldsymbol{\mu}_{2i}\}\mapsto\{\boldsymbol{\alpha}_i\sqrt{q}+1/(\boldsymbol{\alpha}_i\sqrt{q}), \boldsymbol{\alpha}_i/\sqrt{q}+\sqrt{q}/\boldsymbol{\alpha}_i\}.$
\end{theorem}
\begin{proof}
Consider
\begin{equation*}
\begin{split}
    S_r^\flat(x,y)&\defeq\sum_{k}q^{-k(2r+1-k)}S_{k,r}^\flat(-x)^ky^{r-k},\\
    S_r^{1/2,\flat}(x,y)&\defeq\sum_{k}q^{-\frac{1}{2}k(4r+1-2k)}S_{k,r}^{1/2,\flat}(-x)^ky^{r-k}.
\end{split}
\end{equation*}

Then
\begin{equation*}
    \mathrm{Sat}^\flat(S_r^\flat(x,y))=\prod_i(y-x\boldsymbol{\mu}_i),\quad\mathrm{Sat}^{1/2,\flat}(S_r^{1/2,\flat}(x,y))=\prod_i(y-x\boldsymbol{\nu}_i).
\end{equation*}
We consider $S_r^{1/2,\flat}(xy,x^2+y^2).$ Note that
\begin{equation*}
    \mathrm{Sat}^{1/2,\flat}(S_r^{1/2,\flat}(xy,x^2+y^2))=\prod_i(x^2+y^2-xy\boldsymbol{\nu}_i)=\prod_i(y-x\boldsymbol{\alpha}_i)(y-x/\boldsymbol{\alpha}_i).
\end{equation*}
So we need to prove that
\begin{equation*}
    S_r^{\flat}(xy,x^2+y^2)\isequal S_r^{1/2,\flat}(xy\sqrt{q},x^2q+y^2)S_r^{1/2,\flat}(xy/\sqrt{q},x^2/q+y^2).
\end{equation*}
Since $S_r^{1/2,\flat}$ is homogeneous, we may write this as 
\begin{equation*}
    S_r^{\flat}(xy,x^2+y^2)\isequal S_r^{1/2,\flat}(xy,x^2\sqrt{q}+y^2/\sqrt{q})S_r^{1/2,\flat}(xy,x^2/\sqrt{q}+y^2\sqrt{q}).
\end{equation*}

We do this by induction on $r.$ For this,we have
\begin{equation*}
    S_r^{1/2,\flat,*}(x,y)=\sum_{\varepsilon\in\{-1,0,1\}^r}q^{w^{1/2}(\varepsilon)}t(\varepsilon)(-x)^{\lambda_1(\varepsilon)+\lambda_{-1}(\varepsilon)}y^{\lambda_0(\varepsilon)}
\end{equation*}
where
\begin{equation*}
    w^{1/2}(\varepsilon)=\lambda_1(\varepsilon)^2+(\lambda_0(\varepsilon)+\lambda_1(\varepsilon))^2-r^2+\frac{\lambda_1(\varepsilon)-\lambda_{-1}(\varepsilon)}{2}+2\widetilde{\inv}(\varepsilon).
\end{equation*}
Now we note that if $\varepsilon\in\{-1,0,1\}^r,$
\begin{equation*}
\begin{split}
    w^{1/2}(1,\varepsilon)-w^{1/2}(\varepsilon)&=2\lambda_0(\varepsilon)+4\lambda_1(\varepsilon)+2-2r-1+\frac{1}{2}+2\lambda_0(\varepsilon)+4\lambda_{-1}(\varepsilon)=2r+1+\frac{1}{2}\\
    w^{1/2}(0,\varepsilon)-w^{1/2}(\varepsilon)&=2\lambda_0(\varepsilon)+2\lambda_1(\varepsilon)+1-2r-1+2\lambda_{-1}(\varepsilon)=0,\\
    w^{1/2}(-1,\varepsilon)-w^{1/2}(\varepsilon)&=-2r-1-\frac{1}{2}.
\end{split}
\end{equation*}
Hence
\begin{equation*}
    S_{r+1}^{1/2,\flat,*}(x,y)=A(x,y)\star S_{r}^{1/2,\flat,*}(x,y)
\end{equation*}
where
\begin{equation*}
    A(x,y)\defeq -q^{2r+1+\frac{1}{2}}x\cdot t(1)+y\cdot t(0)-q^{-2r-1-\frac{1}{2}}x\cdot  t(-1).
\end{equation*}

Similarly, we have
\begin{equation*}
    S_r^{\flat,*}(x,y)=\sum_{\varepsilon\in\{-1,0,1\}^{2r}}q^{w(\varepsilon)}t(\varepsilon)(-x)^{\lambda_1(\varepsilon)+\lambda_{-1}(\varepsilon)}y^{\lambda_0(\varepsilon)}
\end{equation*}
where
\begin{equation*}
    w(\varepsilon)=\binom{\lambda_1(\varepsilon)+1}{2}-\binom{\lambda_{-1}(\varepsilon)+1}{2}-\lambda_{-1}(\varepsilon)(\lambda_0(\varepsilon)+\lambda_1(\varepsilon))+\widetilde{\inv}(\varepsilon).
\end{equation*}
Now we note that if $\varepsilon\in\{-1,0,1\}^{2r},$
\begin{equation*}
\begin{split}
w(-1,-1,\varepsilon)-w(\varepsilon)&=-2\lambda_{-1}(\varepsilon)-3-2\lambda_0(\varepsilon)-2\lambda_1(\varepsilon)=-4r-3,\\
w(-1,0,\varepsilon)-w(\varepsilon)&=-\lambda_{-1}(\varepsilon)-1-(2r+1)+\lambda_{-1}(\varepsilon)=-2r-2,\\
w(0,-1,\varepsilon)-w(\varepsilon)&=w(-1,0,\varepsilon)-w(\varepsilon)+1=-2r-1,\\
w(-1,1,\varepsilon)-w(\varepsilon)&=\lambda_1(\varepsilon)+1-\lambda_{-1}(\varepsilon)-1-(2r+1)+2\lambda_{-1}(\varepsilon)+\lambda_0(\varepsilon)=-1,\\
w(0,0,\varepsilon)-w(\varepsilon)&=-2\lambda_{-1}(\varepsilon)+2\lambda_{-1}(\varepsilon)=0,\\
w(1,-1,\varepsilon)-w(\varepsilon)&=w(-1,1,\varepsilon)-w(\varepsilon)+2=1,\\
w(0,1,\varepsilon)-w(\varepsilon)&=\lambda_1(\varepsilon)+1-2\lambda_{-1}(\varepsilon)+3\lambda_{-1}(\varepsilon)+\lambda_0(\varepsilon)=2r+1,\\
w(1,0,\varepsilon)-w(\varepsilon)&=w(0,1,\varepsilon)-w(\varepsilon)+1=2r+2,\\
w(1,1,\varepsilon)-w(\varepsilon)&=2\lambda_1(\varepsilon)+3-2\lambda_{-1}(\varepsilon)+2\lambda_0(\varepsilon)+4\lambda_{-1}(\varepsilon)=4r+3,\\
\end{split}
\end{equation*}
and thus
\begin{equation*}
    S_{r+1}^{\flat,*}(x,y)=B(x,y)\star S_{r}^{\flat,*}(x,y)
\end{equation*}
where
\begin{equation*}
\begin{split}
    B(x,y)\defeq\ & q^{4r+3}x^2\cdot t(2)-q^{2r+1}(q+1)xy\cdot t(1)+(y^2+(q+q^{-1})x^2)\cdot t(0)\\
    &-q^{-2r-2}(q+1)xy\cdot t(-1)+q^{-4r-3}x^2\cdot t(-2).
\end{split}
\end{equation*}

Now the claim follows from the computation that
\begin{equation*}
    A(xy,x^2\sqrt{q}+y^2/\sqrt{q})A(xy,x^2/\sqrt{q}+y^2\sqrt{q})=B(xy,x^2+y^2).\qedhere
\end{equation*}
\end{proof}

%% file: structural/Bibliography.tex
\begingroup
\bibliographystyle{alpha}
\bibliography{structural/references}
\endgroup

%% file: table.tex
\newcommand{\twoline}[2]{\begin{pmatrix}#1\\#2\end{pmatrix}}
\newcommand{\fourline}[4]{\begin{pmatrix}#1&#3\\#2&#4\end{pmatrix}}
\begin{landscape}
\begin{table}[!htp]
\caption{A summary of the different straightening relations.}
\begin{adjustbox}{max width=\linewidth}
\begin{tabular}{|c|c|c|cc|}
\hline
 & $\Typ^0$ & $\Typ$ & \multicolumn{2}{c|}{$\Rel$} \\ \hline
\SPn & \multirow{2}{*}[-1.2em]{$\{(e_1 \ge \cdots \ge e_r)\}$} & \multirow{5}{*}[-4em]{$\{(e_1, \ldots, e_r)\}$} & \multicolumn{2}{c|}{$\begin{array}{c}(0,1)\equiv(1,0)\\(0,2)+(q^2-1)(1,1)\equiv q^2(2,0)\\\cdots\end{array}$} \\ \cline{1-1} \cline{4-5} 
\Un &  &  & \multicolumn{2}{c|}{$\begin{array}{c}(0,1)\equiv(1,0)\\(0,2)+q_0(2,0)\equiv(1+q_0)(1,1)\\\cdots\end{array}$} \\ \cline{1-2} \cline{4-5} 
\SPf & \multirow{3}{*}[-2.3em]{$\{(e_1 \ge \cdots \ge e_r \ge 0)\}$} &  & \multicolumn{1}{c|}{\multirow{3}{*}[-2.3em]{$\Rel^\natural+$}} & $\begin{array}{c}(-1)+q(1)\equiv(q+1)(0)\\(-2)+q(q^2+1)(0)+q^4(2)\equiv (q+1)(-1)+q^3(q+1)(1)\\\cdots\end{array}$ \\ \cline{1-1} \cline{5-5} 
\Uf &  &  & \multicolumn{1}{c|}{} & $\begin{array}{c}(-1)\equiv(1)\\(-2)+(q_0-1)(0)\equiv q_0(2)\\\cdots\end{array}$ \\ \cline{1-1} \cline{5-5} 
\Us &  &  & \multicolumn{1}{c|}{} & $\begin{array}{c}(-1)\equiv(1)\\(-2)+q(2)\equiv(q+1)(0)\\\cdots\end{array}$ \\ \hline
\On & $\left\{\left(\begin{array}{c}e_1\\\chi_1\end{array}\begin{array}{c}\ge\\\ \end{array}\cdots\begin{array}{c}\ge\\\ \end{array}\begin{array}{c}e_r\\\chi_r\end{array}\right)\right\}$ & \multirow{3}{*}[-4.6em]{$\left\{\left(\begin{array}{c}e_1\\\chi_1\end{array},\cdots,\begin{array}{c}e_r\\\chi_r\end{array}\right)\right\}$} & \multicolumn{2}{c|}{$\begin{array}{c}\fourline{0}{\chi_1}{1}{\chi_2}\equiv\fourline{1}{\chi_2}{0}{\chi_1}\\\fourline{0}{\chi_1}{2}{\chi_2}\equiv\left(\begin{array}{l}(1+\epsilon\chi_1\chi_2)\fourline{1}{\chi_1}{1}{\chi_2}\\-\epsilon\chi_1\chi_2\left(\frac{1+q\epsilon}{2}\fourline{2}{\chi_2}{0}{\chi_1}+\frac{1-q\epsilon}{2}\fourline{2}{-\chi_2}{0}{-\chi_1}\right)\end{array}\right)\\\cdots\end{array}$} \\ \cline{1-2} \cline{4-5} 
\Of & \multirow{2}{*}[-2.3em]{$\left\{\left(\begin{array}{c}e_1\\\chi_1\end{array}\begin{array}{c}\ge\\\ \end{array}\cdots\begin{array}{c}\ge\\\ \end{array}\begin{array}{c}e_r\\\chi_r\end{array}\begin{array}{c}\ge0\\\ \end{array}\right)\right\}$} &  & \multicolumn{1}{c|}{\multirow{2}{*}[-2.3em]{$\Rel^\natural+$}} & $\begin{array}{c}\twoline{-1}{\chi}\equiv\twoline{1}{\chi}\\\twoline{-2}{\chi}\equiv\twoline{2}{\chi}\\\cdots\end{array}$ \\ \cline{1-1} \cline{5-5} 
\Os &  &  & \multicolumn{1}{c|}{} & $\begin{array}{c}\twoline{-1}{\chi}\equiv\twoline{1}{\chi}\\\twoline{-2}{\chi}+\psi\chi\twoline{2}{\chi}\equiv(1+\psi\chi)\twoline{0}{\chi}\\\cdots\end{array}$ \\ \hline
\end{tabular}
\label{reltable}
\end{adjustbox}\end{table}
\end{landscape}

%% file: structural/References.bib
@article {uHflat,
    AUTHOR = {Hironaka, Yumiko and Komori, Yasushi},
     TITLE = {Spherical functions on the space of {$p$}-adic unitary
              {H}ermitian matrices},
   JOURNAL = {Int. J. Number Theory},
  FJOURNAL = {International Journal of Number Theory},
    VOLUME = {10},
      YEAR = {2014},
    NUMBER = {2},
     PAGES = {513--558},
      ISSN = {1793-0421},
   MRCLASS = {11F85 (11E95 11F70 22E50 33D52)},
  MRNUMBER = {3189992},
MRREVIEWER = {Ivan Mati\'{c}},
       DOI = {10.1142/S1793042113501066},
       URL = {https://doi.org/10.1142/S1793042113501066},
}

@article {uHsharp,
    AUTHOR = {Hironaka, Yumiko and Komori, Yasushi},
     TITLE = {Spherical functions on the space of {$p$}-adic unitary
              hermitian matrices {II}, the case of odd size},
   JOURNAL = {Comment. Math. Univ. St. Pauli},
  FJOURNAL = {Commentarii Mathematici Universitatis Sancti Pauli},
    VOLUME = {63},
      YEAR = {2014},
    NUMBER = {1-2},
     PAGES = {47--78},
      ISSN = {0010-258X},
   MRCLASS = {11F85 (11E95 22E35 33D52)},
  MRNUMBER = {3328424},
MRREVIEWER = {Ivan Mati\'{c}},
}

@article {Bergman,
    AUTHOR = {Bergman, George M.},
     TITLE = {The diamond lemma for ring theory},
   JOURNAL = {Adv. in Math.},
  FJOURNAL = {Advances in Mathematics},
    VOLUME = {29},
      YEAR = {1978},
    NUMBER = {2},
     PAGES = {178--218},
      ISSN = {0001-8708},
   MRCLASS = {16-02},
  MRNUMBER = {506890},
MRREVIEWER = {P. M. Cohn},
       DOI = {10.1016/0001-8708(78)90010-5},
       URL = {https://doi.org/10.1016/0001-8708(78)90010-5},
}

@article {LTXZZ,
    AUTHOR = {Liu, Yifeng and Tian, Yichao and Xiao, Liang and Zhang, Wei
              and Zhu, Xinwen},
     TITLE = {On the {B}eilinson-{B}loch-{K}ato conjecture for
              {R}ankin-{S}elberg motives},
   JOURNAL = {Invent. Math.},
  FJOURNAL = {Inventiones Mathematicae},
    VOLUME = {228},
      YEAR = {2022},
    NUMBER = {1},
     PAGES = {107--375},
      ISSN = {0020-9910},
   MRCLASS = {11G40 (11G05 11G18 11R34)},
  MRNUMBER = {4392458},
       DOI = {10.1007/s00222-021-01088-4},
       URL = {https://doi.org/10.1007/s00222-021-01088-4},
}

@article {Kato,
    AUTHOR = {Kato, Shin-ichi},
     TITLE = {Spherical functions and a {$q$}-analogue of {K}ostant's weight
              multiplicity formula},
   JOURNAL = {Invent. Math.},
  FJOURNAL = {Inventiones Mathematicae},
    VOLUME = {66},
      YEAR = {1982},
    NUMBER = {3},
     PAGES = {461--468},
      ISSN = {0020-9910},
   MRCLASS = {22E50 (05A19 20G05 43A90)},
  MRNUMBER = {662602},
MRREVIEWER = {Ryoshi Hotta},
       DOI = {10.1007/BF01389223},
       URL = {https://doi.org/10.1007/BF01389223},
}

@book {Macdonald,
    AUTHOR = {Macdonald, I. G.},
     TITLE = {Spherical functions on a group of {$p$}-adic type},
    SERIES = {Publications of the Ramanujan Institute, No. 2},
 PUBLISHER = {University of Madras, Centre for Advanced Study in
              Mathematics, Ramanujan Institute, Madras},
      YEAR = {1971},
     PAGES = {vii+79},
   MRCLASS = {22E50},
  MRNUMBER = {0435301},
MRREVIEWER = {Allan J. Silberger},
}

@article {Lecouvey-Lenart,
    AUTHOR = {Lecouvey, C\'{e}dric and Lenart, Cristian},
     TITLE = {Combinatorics of generalized exponents},
   JOURNAL = {S\'{e}m. Lothar. Combin.},
  FJOURNAL = {S\'{e}minaire Lotharingien de Combinatoire},
    VOLUME = {82B},
      YEAR = {2020},
     PAGES = {Art. 14, 12},
   MRCLASS = {05E10 (17B10 17B22)},
  MRNUMBER = {4098235},
}

@article {Lecouvey,
    AUTHOR = {Lecouvey, C\'{e}dric},
     TITLE = {Combinatorics of crystal graphs and {K}ostka-{F}oulkes
              polynomials for the root systems {$B_n,\ C_n$} and {$D_n$}},
   JOURNAL = {European J. Combin.},
  FJOURNAL = {European Journal of Combinatorics},
    VOLUME = {27},
      YEAR = {2006},
    NUMBER = {4},
     PAGES = {526--557},
      ISSN = {0195-6698},
   MRCLASS = {05E15 (17B10)},
  MRNUMBER = {2215214},
MRREVIEWER = {Jae-Hoon Kwon},
       DOI = {10.1016/j.ejc.2005.01.006},
       URL = {https://doi.org/10.1016/j.ejc.2005.01.006},
}

@article {Jang-Kwon,
    AUTHOR = {Jang, Il-Seung and Kwon, Jae-Hoon},
     TITLE = {Flagged {L}ittlewood-{R}ichardson tableaux and branching rule
              for classical groups},
   JOURNAL = {J. Combin. Theory Ser. A},
  FJOURNAL = {Journal of Combinatorial Theory. Series A},
    VOLUME = {181},
      YEAR = {2021},
     PAGES = {Paper No. 105419, 51},
      ISSN = {0097-3165},
   MRCLASS = {17B37 (05E10 22E46)},
  MRNUMBER = {4216338},
MRREVIEWER = {Iwan Praton},
       DOI = {10.1016/j.jcta.2021.105419},
       URL = {https://doi.org/10.1016/j.jcta.2021.105419},
}

@book {Macdonaldbook,
    AUTHOR = {Macdonald, I. G.},
     TITLE = {Symmetric functions and {H}all polynomials},
    SERIES = {Oxford Classic Texts in the Physical Sciences},
   EDITION = {Second},
      NOTE = {With contribution by A. V. Zelevinsky and a foreword by
              Richard Stanley,
              Reprint of the 2008 paperback edition [ MR1354144]},
 PUBLISHER = {The Clarendon Press, Oxford University Press, New York},
      YEAR = {2015},
     PAGES = {xii+475},
      ISBN = {978-0-19-873912-8},
   MRCLASS = {05E05 (01A75 05-02 20C30 20C33 20K01 33C80 33D80)},
  MRNUMBER = {3443860},
}

@article {Sakellaridis,
    AUTHOR = {Sakellaridis, Yiannis},
     TITLE = {Spherical functions on spherical varieties},
   JOURNAL = {Amer. J. Math.},
  FJOURNAL = {American Journal of Mathematics},
    VOLUME = {135},
      YEAR = {2013},
    NUMBER = {5},
     PAGES = {1291--1381},
      ISSN = {0002-9327},
   MRCLASS = {22E50 (11F67 14M27 20G25 22E55)},
  MRNUMBER = {3117308},
MRREVIEWER = {Alexander Alldridge},
       DOI = {10.1353/ajm.2013.0046},
       URL = {https://doi.org/10.1353/ajm.2013.0046},
}

@unpublished {BZSV,
    AUTHOR = {Ben-Zvi, David and Sakellaridis, Yiannis and Venkatesh, Akshay},
    TITLE = {Duality in the relative {L}anglands program},
    NOTE = {In preparation},
}

@article {Cornut2,
    AUTHOR = {Cornut, Christophe},
     TITLE = {On {$p$}-adic norms and quadratic extensions, {II}},
   JOURNAL = {Manuscripta Math.},
  FJOURNAL = {Manuscripta Mathematica},
    VOLUME = {136},
      YEAR = {2011},
    NUMBER = {1-2},
     PAGES = {199--236},
      ISSN = {0025-2611},
   MRCLASS = {11E95 (20E42 20G25 22E35)},
  MRNUMBER = {2820402},
MRREVIEWER = {Guy Rousseau},
       DOI = {10.1007/s00229-011-0442-0},
       URL = {https://doi.org/10.1007/s00229-011-0442-0},
}

@article {Cornut1,
    AUTHOR = {Cornut, Christophe},
     TITLE = {Normes {$p$}-adiques et extensions quadratiques},
   JOURNAL = {Ann. Inst. Fourier (Grenoble)},
  FJOURNAL = {Universit\'{e} de Grenoble. Annales de l'Institut Fourier},
    VOLUME = {59},
      YEAR = {2009},
    NUMBER = {6},
     PAGES = {2223--2254},
      ISSN = {0373-0956},
   MRCLASS = {11E95 (20G25 22E35)},
  MRNUMBER = {2640919},
MRREVIEWER = {Maarten Sander Solleveld},
       URL = {http://aif.cedram.org/item?id=AIF_2009__59_6_2223_0},
}

@unpublished {Cornut,
    AUTHOR = {Cornut, Christophe},
    TITLE = {AN {E}ULER SYSTEM OF {H}EEGNER TYPE},
    NOTE = {In preparation},
}

@misc{Graham-Shah,
      title={Anticyclotomic Euler systems for unitary groups}, 
      author={Andrew Graham and Syed Waqar Ali Shah},
      year={2023},
      eprint={2001.07825},
      archivePrefix={arXiv},
      primaryClass={math.NT}
}

@article {LSZ-SP4,
    AUTHOR = {Loeffler, David and Skinner, Christopher and Zerbes, Sarah
              Livia},
     TITLE = {Euler systems for {${\rm GSp}(4)$}},
   JOURNAL = {J. Eur. Math. Soc. (JEMS)},
  FJOURNAL = {Journal of the European Mathematical Society (JEMS)},
    VOLUME = {24},
      YEAR = {2022},
    NUMBER = {2},
     PAGES = {669--733},
      ISSN = {1435-9855},
   MRCLASS = {11F46 (11F67 11F80 11R23)},
  MRNUMBER = {4382481},
MRREVIEWER = {Meng Fai Lim},
       DOI = {10.4171/jems/1124},
       URL = {https://doi.org/10.4171/jems/1124},
}

@article {LSZ-U21,
    AUTHOR = {Loeffler, David and Skinner, Christopher and Zerbes, Sarah
              Livia},
     TITLE = {An {E}uler system for {$\rm GU(2,1)$}},
   JOURNAL = {Math. Ann.},
  FJOURNAL = {Mathematische Annalen},
    VOLUME = {382},
      YEAR = {2022},
    NUMBER = {3-4},
     PAGES = {1091--1141},
      ISSN = {0025-5831},
   MRCLASS = {11R23 (11F70 14G35)},
  MRNUMBER = {4403219},
MRREVIEWER = {Meng Fai Lim},
       DOI = {10.1007/s00208-021-02224-4},
       URL = {https://doi.org/10.1007/s00208-021-02224-4},
}

@misc{Jetchev,
      title={Hecke and Galois Properties of Special Cycles on Unitary Shimura Varieties}, 
      author={Dimitar P. Jetchev},
      year={2016},
      eprint={1410.6692},
      archivePrefix={arXiv},
      primaryClass={math.NT}
}

@article{BBJVertical,
    author = {Boumasmoud, Réda and Brooks, Ernest Hunter and Jetchev, Dimitar P},
    title = "{Vertical Distribution Relations for Special Cycles on Unitary Shimura Varieties}",
    journal = {International Mathematics Research Notices},
    volume = {2020},
    number = {13},
    pages = {3902-3926},
    year = {2018},
    month = {06},
    issn = {1073-7928},
    doi = {10.1093/imrn/rny119},
    url = {https://doi.org/10.1093/imrn/rny119},
    eprint = {https://academic.oup.com/imrn/article-pdf/2020/13/3902/33422435/rny119.pdf},
}

@misc{BBJHorizontal,
      title={Horizontal Distribution Relations for Special Cycles on Unitary Shimura Varieties: Split Case}, 
      author={Reda Boumasmoud and Ernest Hunter Brooks and Dimitar Jetchev},
      year={2016},
      eprint={1611.09425},
      archivePrefix={arXiv},
      primaryClass={math.NT}
}

@misc{Boumasmoud,
      title={General Horizontal Norm Compatible Systems on unitary Shimura varieties}, 
      author={Reda Boumasmoud},
      year={2021},
      eprint={2106.12540},
      archivePrefix={arXiv},
      primaryClass={math.NT}
}

@misc{CoratoZanarella,
      title={Spherical functions of symmetric forms and a conjecture of {H}ironaka}, 
      author={Murilo Corato-Zanarella},
      year={2023},
}

@misc{CoratoZanarellaEuler,
      title={First explicit reciprocity law for unitary {F}riedberg--{J}acquet periods}, 
      author={Murilo Corato-Zanarella},
      NOTE = {In preparation},
}
